\documentclass[11pt]{article}
\usepackage{amsmath}
\usepackage{amssymb,amsbsy,amsthm,dsfont}
\usepackage{graphicx}
\usepackage[dvipsnames]{xcolor}
\usepackage{tikz}
\tikzstyle{vertex}=[circle, draw, inner sep=0pt, minimum size=11pt]
\newcommand{\vertex}{\node[vertex]}
\usepackage{caption}
\usepackage{subcaption}
\usepackage{enumerate}
\usepackage[normalem]{ulem}
\usepackage{cases}
\usepackage[margin=1in]{geometry}
\usepackage[colorlinks,linkcolor=blue,citecolor=blue]{hyperref}
\usepackage{verbatim} 
\usepackage{bbm}

\usepackage{pgfplots}		
\usepackage{tikz}
\usetikzlibrary{calc,decorations.pathreplacing,decorations.pathmorphing}
\tikzset{ 
	protovertex/.style={
		draw,
		circle,
		inner sep=0,
		minimum size=.15cm}
}
\usepgfplotslibrary{statistics}

\definecolor{TUMBlue}{HTML}{0065BD}

\usepackage{mathtools}
\mathtoolsset{showonlyrefs}
\usepackage{algorithm}
\usepackage{algpseudocode}

\makeatletter
\newenvironment{breakablealgorithm}
  {
   \begin{center}
     \refstepcounter{algorithm}
     \hrule height.8pt depth0pt \kern2pt
     \renewcommand{\caption}[2][\relax]{
       {\raggedright\textbf{\ALG@name~\thealgorithm} ##2\par}%
       \ifx\relax##1\relax 
         \addcontentsline{loa}{algorithm}{\protect\numberline{\thealgorithm}##2}%
       \else 
         \addcontentsline{loa}{algorithm}{\protect\numberline{\thealgorithm}##1}%
       \fi
       \kern2pt\hrule\kern2pt
     }
  }{
     \kern2pt\hrule\relax
   \end{center}
  }
\makeatother

\allowdisplaybreaks[1]

\numberwithin{equation}{section}

\DeclareMathOperator{\R}{\mathbb{R}} 
\DeclareMathOperator{\N}{\mathbb{N}} 

\newcommand{\p}{\mathbb{P}} 
\renewcommand{\P}{\mathbb{P}} 
\newcommand{\E}{\mathbb{E}} 
\newcommand{\eps}{\varepsilon} 

\newcommand{\Exp}{\mathrm{Exp}} 

\newcommand{\wt}{\widetilde} 

\newcommand{\UH}{\mathrm{UH}} 
\newcommand{\MCST}{\mathrm{LCS}} 

\newcommand{\LCS}{\mathrm{LCS}} 

\newcommand{\UHCS}{\mathrm{UHCS}} 

\newcommand{\ycp}{\mathrm{YCP}} 

\newcommand{\prob}[1]{\mathbb{P} \left( #1 \right)}

\newcommand{\od}{\text{outdeg}}

\newcommand{\be}{\begin{equation}}
\newcommand{\ee}{\end{equation}}
\newcommand{\ba}{\begin{aligned}}
\newcommand{\ea}{\end{aligned}}

\newcommand{\indicator}{\mathbbm 1}
\newcommand{\bp}{\mathrm{BP}}
\newcommand{\e}{\mathrm e}
\newcommand{\dd}{\mathrm d}
\newcommand{\invisible}[1]{}

\newcommand{\basc}[1]{{\color{red}{ \bf [~Bas:\ }\emph{#1}\textbf{~]}}}

\makeatletter
\def\namedlabel#1#2{\begingroup
	#2%
	\def\@currentlabel{#2}%
	\phantomsection\label{#1}\endgroup
}

\def\cB{{\mathcal B}}
\def\cC{{\mathcal C}}
\def\cD{{\mathcal D}}
\def\cE{{\mathcal E}}

\def\cO{{\mathcal O}}
\def\cP{{\mathcal P}}

\def\cT{{\mathcal T}}
\def\cU{{\mathcal U}}

\newtheorem{theorem}{Theorem}[section]
\newtheorem{lemma}[theorem]{Lemma}
\newtheorem{observation}[theorem]{Observation}

\newtheorem{proposition}[theorem]{Proposition}
\newtheorem{claim}[theorem]{Claim}

\newtheorem{corollary}[theorem]{Corollary}
\newtheorem{conjecture}{Conjecture}[section]

\theoremstyle{definition}
\newtheorem{definition}[theorem]{Definition}
\newtheorem{remark}[theorem]{Remark}

\begin{document}

\title{On the largest common subtree of uniform attachment trees}
\author{
	Johannes B\"aumler\thanks{University of Koblenz; \url{jbaeumler@uni-koblenz.de}}
	\and 
    C\'eline Kerriou\thanks{Stockholm University; \url{celine.kerriou@mail.mcgill.ca}}
	\and 
	Bas Lodewijks\thanks{University of Sheffield; \url{bas.lodewijks@sheffield.ac.uk}}
	\and 
	James Martin\thanks{University of Oxford; \url{martin@stats.ox.ac.uk}} 
	\and 
	Emil Powierski\thanks{TU Dresden; \url{emil.powierski@tu-dresden.de}}
	\and
	Mikl\'os Z.\ R\'acz\thanks{Northwestern University; \url{miklos.racz@northwestern.edu}}
	\and
	Anirudh Sridhar\thanks{New Jersey Institute of Technology; \url{anirudh.sridhar@njit.edu}}
}
\date{\today}

\maketitle


\vspace{-1cm}
\begin{abstract}
We study the largest common subtree of two independent unlabeled uniform attachment trees (also known as random recursive trees). Our main result shows that, when the two trees have $n$ vertices each, their largest common subtree has at least $n^{0.83}$ vertices with high probability. This is obtained by starting with the common subtree induced by the Ulam--Harris labels in the two trees and improving using local optimization steps. We also give some upper bounds and bounds for general random tree growth models. We leave as an intriguing open question to understand the magnitude of the size of the largest common subtree. 
\end{abstract} 

\begin{figure}[h!]
    \centering
    \includegraphics[width=0.7\textwidth]{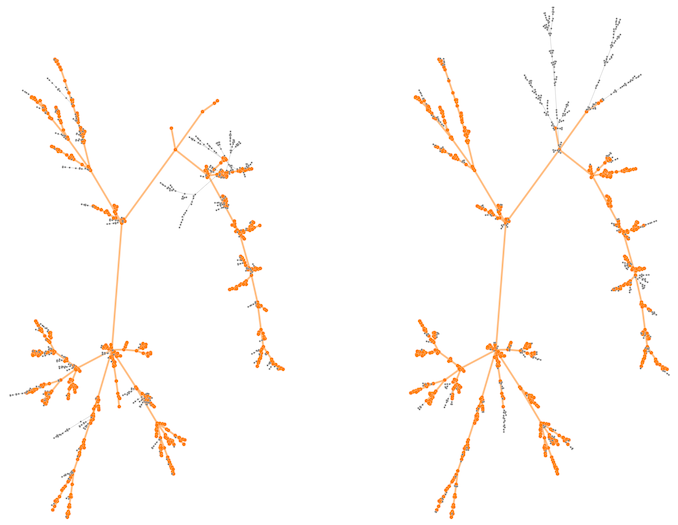}
    \caption{Two independent uniform attachment trees with $1000$ vertices each. Their largest common subtree, which contains 484 vertices, is highlighted in orange.}
    \label{fig:front}
\end{figure}

\clearpage


\tableofcontents

\clearpage

\section{Introduction} \label{sec:intro} 

A uniform attachment tree, also known as a random recursive tree, is defined recursively as follows. 
Let $T_1$ be a graph consisting of a single node (and no edges). 
Given $T_n$, the tree $T_{n+1}$ is formed from $T_n$ by adding a single node that is connected with a single edge to a node of $T_n$, chosen uniformly at random, independently of everything else.  
Uniform attachment (UA) trees are the simplest model of randomly growing graphs and hence are very well studied~\cite{drmota2009random}.

Consider two independent unlabeled 
UA 
trees with $n$ nodes, $T_{n}^{1}$ and~$T_{n}^{2}$. We are interested in the size of the largest common subtree between $T_{n}^{1}$ and $T_{n}^{2}$. That is, let 
\[
X_{n} 
:= \LCS \left( T_{n}^{1}, T_{n}^{2} \right) 
:= \max_{S \subseteq T_{n}^{1}, T_{n}^{2}} |S|, 
\]
where the maximum goes over all trees $S$ such that $S$ is a subtree of both $T_{n}^{1}$ and $T_{n}^{2}$, and $|S|$ denotes the number of nodes in $S$. 
The focus of this paper is to understand the random variable~$X_{n}$. 
See Figure~\ref{fig:front} for an illustration, which shows two independent UA trees, as well as their largest common subtree. 
We note that while the maximum common subgraph problem is NP-hard for general graphs, the restriction to trees allows for polynomial time algorithms to compute~$X_n$~\cite{droschinsky2016faster}.

To the best of our knowledge, $X_{n}$ has not been studied before. 
Similar questions concerning largest common substructures have been studied in probabilistic combinatorics for several decades~\cite{AldousOpenProblem_substructures}, 
including for other models of random trees and random graphs---see Section~\ref{sec:related} for a discussion of related work. 
In particular, a growing line of work 
has been studying the size of the largest common subtree (also referred to as the maximum agreement subtree) 
of two independent uniformly random cladograms (leaf-labeled binary trees)~\cite{AldousOpenProblem_substructures,bryant2003size,bernstein2015bounds,misra2019bounds,bordewich2022maximum,aldous2022largest,pittel2023expected,khezeli2024improved,budzinski2024maximum}. 
The motivation here comes from phylogenetics, where this quantity can be used as a similarity metric to compare two phylogenetic trees (e.g.,~\cite{de2007congruence}). 
In such statistical applications, it is imperative to understand the distribution of the relevant statistic (e.g., similarity metric) under the null model, which in this setting is naturally two independent trees. 

Our initial motivation comes from graph matching: $X_{n}$ is a natural candidate statistic to detect whether two UA trees are correlated\footnote{There are several natural models for correlated UA trees; however, we refrain from defining these, as this discussion is orthogonal to the main content of the paper. We refer the interested reader to~\cite{RS22AAP,BRRS26}.} or not. Heuristically, if two UA trees are correlated, then they should have a larger common subtree than if they were independent. 
As in the literature discussed above, to rigorously establish such a heuristic, one must understand the distribution of $X_{n}$ under the null model, which is independent UA trees. 
In the past decade there has been extensive work on matching correlated random graphs, focusing primarily on Erd\H{o}s--R\'enyi random graphs. 
In particular, recently Ding, Du, and Gong~\cite{ding2024polynomial} studied the random optimization problem of maximizing the overlap (i.e., the number of common edges) of two independent Erd\H{o}s--R\'enyi random graphs, which is similar in spirit to our work (we also refer to~\cite{ding2024polynomial} and the references within for a comprehensive overview of the literature on random graph matching). 
An important emerging research direction is to understand graph matching on random graphs beyond the Erd\H{o}s--R\'enyi model, such as correlated randomly growing graphs~\cite{RS22AAP,BRRS26}, correlated stochastic block models~\cite{RS21SBM,GRS22}, correlated random geometric graphs~\cite{wang2022geometric}, and more~\cite{RS23,ding2025efficiently}. The study of $X_{n}$ is thus motivated by this line of work, particularly on correlated randomly growing graphs~\cite{RS22AAP,BRRS26}. 
Understanding $X_{n}$ is also of inherent mathematical interest as a natural and simple mathematical object.

In the rest of the introduction we first present our results on $X_{n}$, including lower and upper bounds, as well as simulation results (Section~\ref{sec:results_main}). 
We then discuss results on general randomly growing tree models (Section~\ref{sec:results_general}), open problems and future directions (Section~\ref{sec:open}), and related work (Section~\ref{sec:related}). The rest of the paper is dedicated to the proofs.

\subsection{Main results} \label{sec:results_main}

Our main result is the following lower bound on the size of the largest common subtree of two UA trees. 
Throughout the paper we say that an event holds with high probability if its probability tends to $1$ as $n \to \infty$. We also use standard asymptotic notation (such as $O(\cdot), o(\cdot)$, etc.).

\begin{theorem}\label{thm:main}
We have that $X_{n} \geq n^{0.83}$ with high probability.
\end{theorem}

We present three, increasingly more complex lower bound arguments, which result in increasingly larger lower bounds, culminating in Theorem~\ref{thm:main}. 
The first argument is based on the following simple observation: 
if $v$ is a node of a UA tree with $n$ nodes, 
then with probability at least $1/2$, at least one incoming node will attach to $v$ by the time the tree has $2n$ nodes. 
This implies that $\E[X_{2n}] \geq (5/4) \E[X_{n}]$, 
and applying this recursively leads to a polynomial lower bound in $n$ (albeit a relatively small one; roughly, $n^{0.32}$). 
See Section~\ref{sec:polyLB} for details. 

To improve upon this, we study a specific common subtree construction, which we term the 
\emph{Ulam--Harris common subtree} 
of the two trees $T_{n}^{1}$ and $T_{n}^{2}$. 
The basis of this construction is the \emph{Ulam--Harris labeling} of a randomly growing tree, which is a systematic way of keeping track of the evolution of the tree. 
The initial vertex, termed the root, has Ulam--Harris label $\varnothing$. 
When a new node $v$ enters the tree, it receives an Ulam--Harris label based on the Ulam--Harris label of the parent of $v$ (i.e., the node that $v$ attaches to).
Without getting into technical details (see Section~\ref{sec:UlamHarris} for these), the basic idea, through an example, is as follows. 
A node $v$ has an Ulam--Harris label of $(2,5,4)$ 
if $v$ is the $4$th node attaching to the parent of $v$, 
which in turn is the $5$th node attaching to the grandparent of $v$, 
which in turn is the $2$nd node attaching to the great-grandparent of $v$, 
which is the root. 
The Ulam--Harris labels can be viewed as a \emph{random labeling} of a randomly growing~tree. 

The nodes of two independent UA trees $T_{n}^{1}$ and $T_{n}^{2}$ will (typically) have different (random) Ulam--Harris labels. 
A natural idea is to consider the subset of nodes in the two trees whose Ulam--Harris label appears in both trees; 
it is not hard to see that in both trees these nodes form a subtree 
and, moreover, that these two subtrees are isomorphic, 
with the Ulam--Harris labels providing an isomorphism between them. 
In short, this is a common subtree of $T_{n}^{1}$ and~$T_{n}^{2}$;
see Figure \ref{fig:ulam_harris_example} for an example.
We analyze the size of this common subtree using large deviation arguments, and show that it has size roughly $n^{2(\sqrt{2}-1)}$. 
More precisely, we show that for every $\eps > 0$ 
we have that $X_{n} \geq n^{2(\sqrt{2}-1)-\eps}$ with high probability; in particular, $X_{n} \geq n^{0.828}$. See Section~\ref{sec:UHLB} for details. 

\begin{figure}
    \centering
    \includegraphics[width=0.95\textwidth]{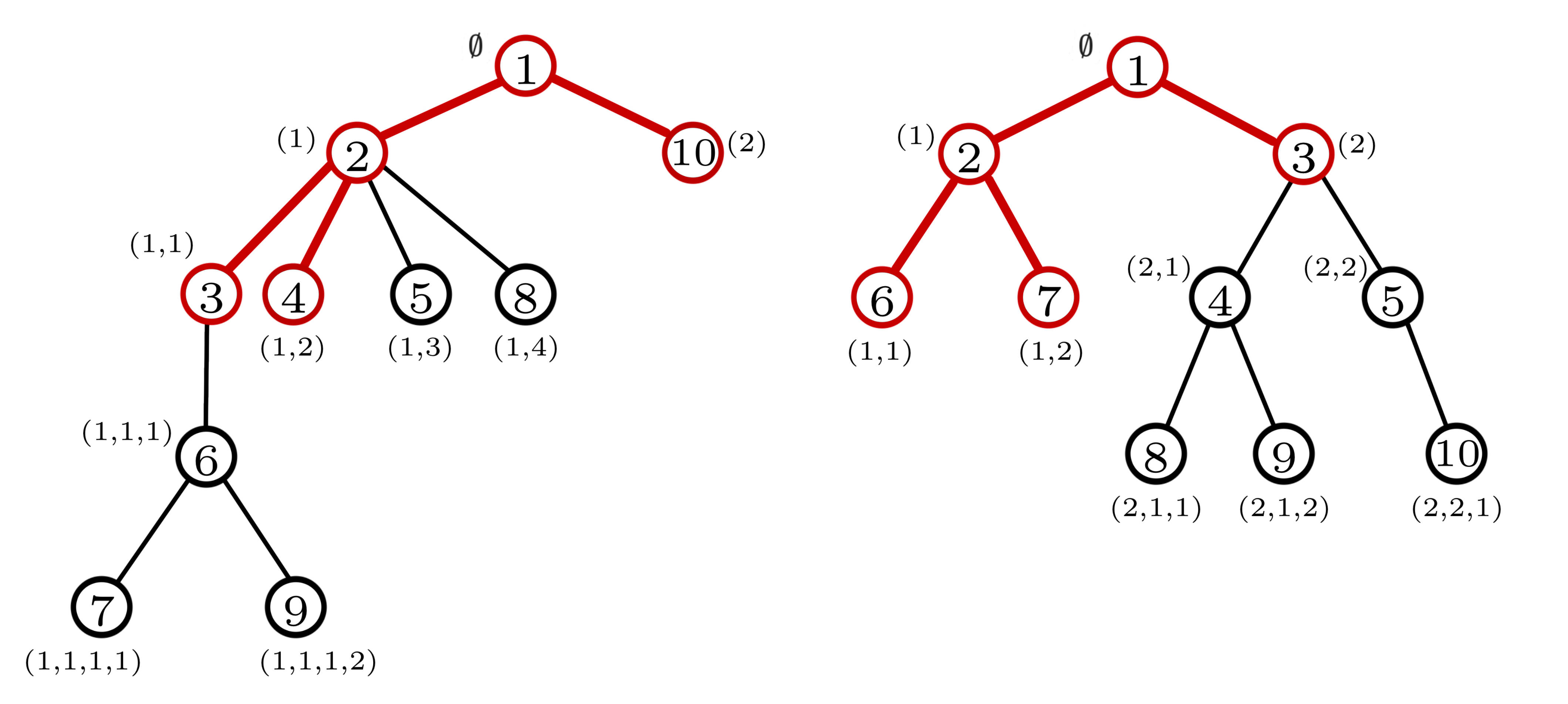}
    \caption{\textit{Ulam--Harris labels and the associated common subtree.} Visualization of two independent uniform attachment trees with $n = 10$ vertices. 
    The vertex labels within each node denote the arrival time of each vertex ($1$ through $n$). 
    The \emph{Ulam--Harris labels} are outside of the nodes.
    The associated \emph{Ulam--Harris common subtree} of the two trees is highlighted in red. 
    In this example the Ulam--Harris common subtree has $5$ nodes, whereas the largest common subtree has $8$ nodes.}
    \label{fig:ulam_harris_example}
\end{figure}

The Ulam--Harris 
common subtree 
discussed above matches ``early'' vertices in $T_{n}^{1}$ and $T_{n}^{2}$. 
However, this is suboptimal due to the 
random fluctuations in which parts of the trees grow faster. 
In particular, the subtree of descendants of later vertices may (and, often, do) grow faster than the corresponding subtree of earlier vertices, 
and hence it is more advantageous to match the later vertices where the tree is growing faster. 
To show this rigorously, 
we start from the 
Ulam--Harris 
common subtree
and study particular local optimizations that switch subtrees, obtaining a larger common subtree. 
The analysis of the size of this common subtree requires much more delicate large deviation bounds, 
culminating in a proof of Theorem~\ref{thm:main} (see Section~\ref{sec:beyondUH} for details). 
While the quantitative jump in the lower bound from the 
Ulam--Harris 
common subtree
to Theorem~\ref{thm:main} is relatively small, 
we emphasize that the main point here is the qualitative result of rigorously going beyond the 
Ulam--Harris 
common subtree. 

\begin{figure}[t]
    \centering
    \begin{subfigure}{0.43 \textwidth}
        \centering
        \includegraphics[width= \textwidth]{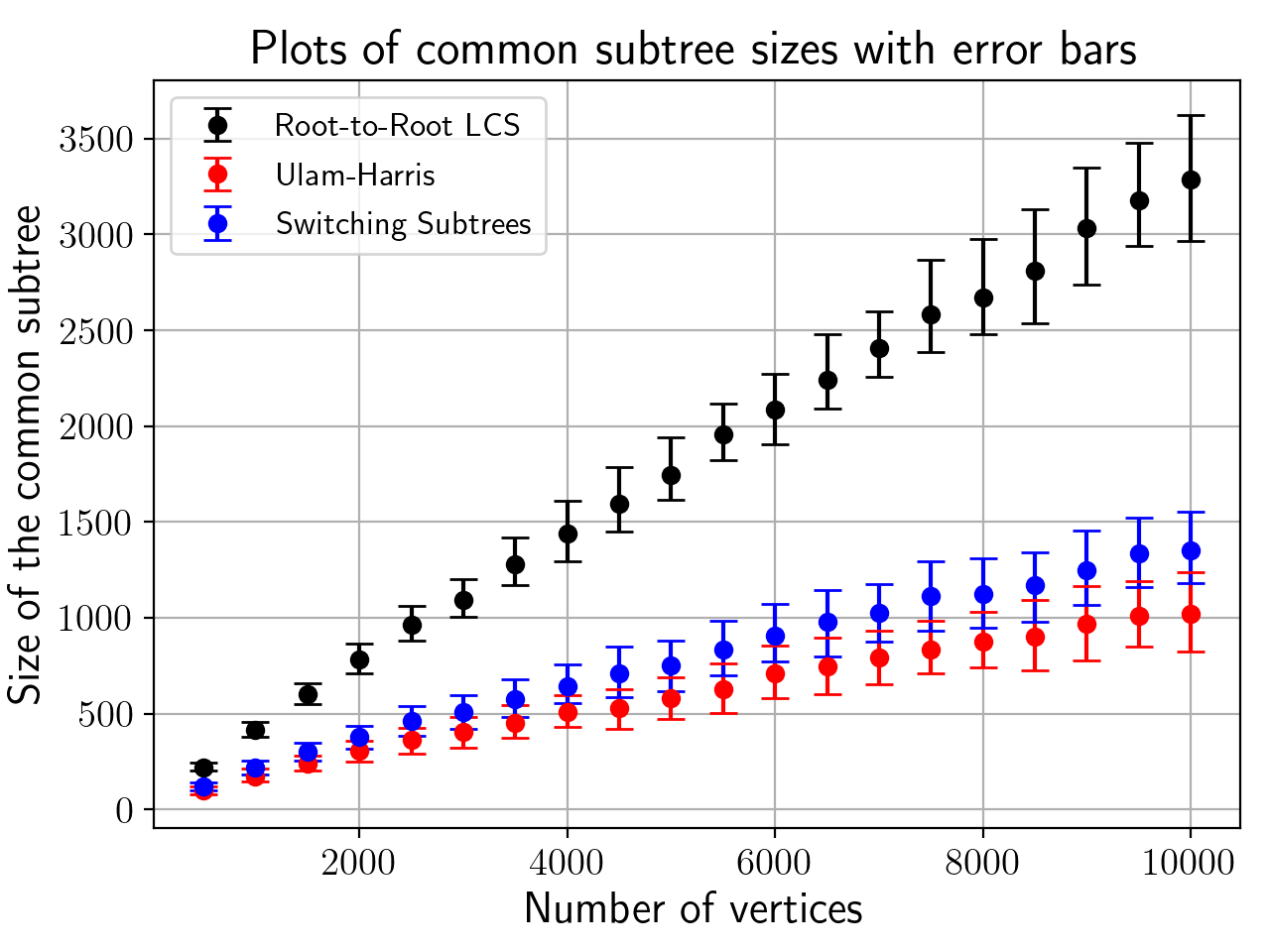}
        \caption{ }
    \end{subfigure}%
    \qquad \qquad
    \begin{subfigure}{0.43 \textwidth}
        \centering
        \includegraphics[width= \textwidth]{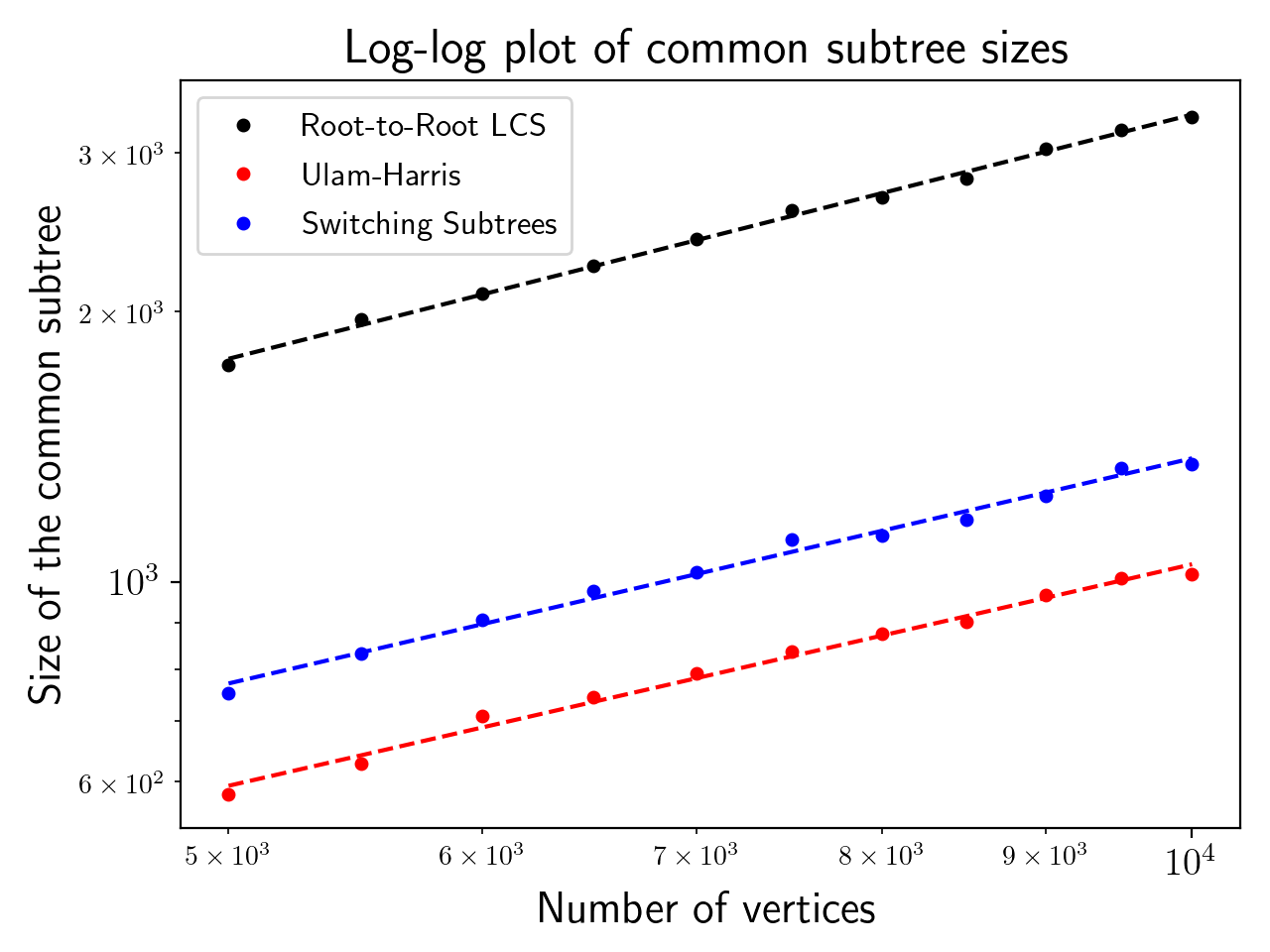}
        \caption{ }
    \end{subfigure}%
    \caption{Simulation results for the common subtree schemes considered in this paper. \emph{Root-to-Root~LCS} refers to the size of the largest common subtree where the roots (i.e., the initial nodes) of the two trees are matched to each other; \emph{Ulam--Harris} refers to the common subtree induced by the Ulam--Harris labels; and \emph{Switching subtrees} refers to our improvement to the Ulam--Harris common subtree (see Algorithm~\ref{alg:switching_subtrees} in Section~\ref{sec:beyondUH}). 
    For each value of $n$, 200 independent trials were averaged. In Fig.~\ref{fig:simulations}(a), the lower and upper error bars represent the 25th and 75th percentiles of the data collected, respectively. In Fig.~\ref{fig:simulations}(b), the dotted lines represent the corresponding best-fit curves for the log-log plots of the simulation data. The slopes of the common subtree schemes \emph{Root-to-Root LCS}, \emph{Ulam--Harris}, and \emph{Switching subtrees} are 0.9014, 0.8177, and 0.8311, respectively.}
    \label{fig:simulations}
\end{figure}

It is instructive to look at simulations which empirically show how various common subtree sizes grow with $n$. In Figure~\ref{fig:simulations}, we plot the average common subtree sizes of three constructions for $n \in [1000, 10000]$: (1) the largest common subtree where the roots (i.e., the initial nodes) of the two trees are matched to each other (termed \emph{Root-to-Root LCS}), (2) the 
Ulam--Harris 
common subtree 
(termed \emph{Ulam--Harris}), and (3) our improvement over the Ulam--Harris 
common subtree
(termed \emph{Switching subtrees}; see Algorithm~\ref{alg:switching_subtrees} in Section~\ref{sec:beyondUH}). For the considered range of~$n$, computing $X_n$ exactly requires a significant computational overhead, so we did not include it in our simulations.\footnote{Computing the Root-to-Root LCS for UA trees can be done in quasi-quadratic time using a dynamic programming approach (see the proof of Proposition~1 in~\cite{droschinsky2016faster}, using also that the maximum degree in UA trees is $O(\log n)$ with high probability). However, using the Root-to-Root LCS as a subroutine for computing $X_n$ through a brute-force approach requires us to iterate over all pairs of ``roots'' across the two trees, which makes the resulting algorithm run in quasi-quartic time with respect to $n$. For $n$ much larger than 1000 (which is the regime where we expect limiting behavior to emerge, as seen in Figure~\ref{fig:simulations}), this requires significant computational resources.} Figure~\ref{fig:simulations} shows that for the subtrees obtained by \emph{Ulam--Harris} and \emph{Switching subtrees}, the empirical results align quite closely with our theoretical results. Moreover, Figure~\ref{fig:simulations} suggests that the \emph{Root-to-Root LCS} is asymptotically at least $n^{0.9}$, which also serves as an empirical lower bound for~$X_n$.

To complement the lower bound of Theorem~\ref{thm:main}, 
we provide some upper bounds that show that $X_{n}$ cannot be too large. 
Specifically, the next two results show that $X_{n}/n$ is bounded away from~$1$ (see Theorem~\ref{thm:UB_near1}) and that $X_{n}/n$ is not bounded away from~$0$ (see Theorem~\ref{thm:UB_near0}). 

\begin{theorem}\label{thm:UB_near1}
We have that $X_{n} \leq 0.999 n$ with high probability. 
\end{theorem}

\begin{theorem}\label{thm:UB_near0}
For every $\eps > 0$ there exists $\delta > 0$ such that, for all sufficiently large $n$, we have that $\p \left( X_{n} \leq \eps n \right) \geq \delta$. 
\end{theorem}

Overall, for a variety of reasons---including the empirical results and a lack of stronger upper bounds---we conjecture that $X_{n}$ grows near-linearly (though not necessarily linearly) in $n$. 

\begin{conjecture}\label{conj:asymptotics}
We have that $X_{n} = n^{1-o(1)}$ with high probability. 
\end{conjecture}

In other words, we conjecture that there exists a sequence $\eps_{n} \to 0$ such that $X_{n} \geq n^{1-\eps_{n}}$ with high probability. 
Proving or disproving Conjecture~\ref{conj:asymptotics} would be very interesting. If this conjecture is true, then it would be interesting to understand the more precise finer asymptotics as well. 
We refer to Section~\ref{sec:open} for further discussion of this and other open problems.

\subsection{Results on general randomly growing tree models} \label{sec:results_general}

Some of our techniques apply not only to UA trees but also to general randomly growing tree models. 
Another well-studied model is \emph{preferential attachment} (PA)~\cite{Mah92,BA99,BRST01}, 
where the probability that the incoming node attaches to a particular existing node is proportional to its degree. 
More generally~\cite{RTV07}, given an attachment function 
$f: \N \to (0,\infty)$, 
one can consider the model where the probability that the incoming node attaches to a particular node $v$ is proportional to 
$f(\mathrm{deg}(v))$, 
where $\mathrm{deg}(v)$ is the degree of $v$ in the current tree (see Section~\ref{sec:general} for formal details). 

Let us generate two independent trees from such a model, with both trees having $n$ nodes, 
and let $X_{n}$ denote again the size of their largest common subtree 
(suppressing the dependence on~$f$ in the notation for simplicity). 
The following result provides a polynomial lower bound on $X_{n}$ under mild conditions on $f$. 
For simplicity, we state a lower bound in expectation and leave the corresponding high-probability result as an exercise for the reader. 

\begin{theorem}\label{thm:general_lb} 
Suppose that $f$ grows at most linearly (that is, $\limsup_{n\to\infty} f(n) / n < \infty$) and that $\inf_{n\in \N} f(n)>0$. Then there exists $\eps = \eps(f) >0$ such that for all $n \geq 1$ we have that $\E \left[ X_n \right] \geq n^\eps$.
\end{theorem}

The proof follows a recursive argument, similar to the one mentioned above in Section~\ref{sec:results_main}; 
see Section~\ref{sec:linear} for details. 
A different lower bound is based on the maximum degree: 
$X_{n}$ is always larger than the smaller of the two maximum degrees in the two trees; 
this is because a tree contains as a subtree a star centered at a maximum degree node (see Lemma~\ref{lem:max_degree_lb}). 
This lower bound is complementary to the one in Theorem~\ref{thm:general_lb}: each can be better in certain cases. 
In Section~\ref{sec:linear} we detail this trade-off for affine preferential attachment, where $f(n) = n + a$ for some $a > -1$.

For certain models it is possible to obtain precise results on $X_{n}$. 
In particular, for superlinear preferential attachment~\cite{oliveira2005connectivity}, where $f(n) = n^{\alpha}$ for some $\alpha > 1$, 
we show that $X_{n} = (1-o(1))n$ with high probability; 
see Section~\ref{sec:superlinear} for more precise statements and details.

\subsection{Open problems and future directions} \label{sec:open}

Our work initiates numerous interesting open questions and directions for future research. We collect here some of them.

\begin{itemize}
\item \textbf{Asymptotics of $X_{n}$.} The main question that our work leaves open is to understand the order of magnitude of $X_{n}$ for independent UA trees. In particular, proving or disproving Conjecture~\ref{conj:asymptotics} would be very interesting. If Conjecture~\ref{conj:asymptotics} is true, then understanding the 
precise 
finer asymptotics hidden in the $o(1)$ in the exponent would be of interest as well. 

Partial progress towards Conjecture~\ref{conj:asymptotics} would be exciting too. For instance, can one prove that, say, $X_{n} \geq n^{0.9}$ with high probability (if this is true)? 
It seems plausible that further local optimizations (beyond those in Section~\ref{sec:beyondUH}) could improve upon Theorem~\ref{thm:main} to a certain extent, but it is unclear if these methods could be pushed as far as to give a lower bound of, say, $n^{0.9}$. For this reason, it would be very interesting to have completely different approaches to the problem (and in particular to constructing lower bounds). 

In a different direction, is it true that $X_{n} = o(n)$ with high probability?

\item \textbf{Limit theorems.} Once the first order asymptotics of $X_{n}$ are understood, it would be interesting to understand distributional limit theorems around them (if they exist). 
For instance, does $X_{n} / \E[X_{n}]$ converge to a constant or does it have a nontrivial limiting distribution? 
Simulations suggest the latter (at least for the \emph{Root-to-Root LCS}), see Figure~\ref{fig:histogram_normalized_expectation}.

\item \textbf{General tree growth models.} 
How does $X_{n}$ behave for other tree growth models beyond UA? 
Sections~\ref{sec:results_general} and~\ref{sec:general} contain bounds, but a broader exploration would be of interest, both generally and for  specific models. 
For instance, how does $X_{n}$ grow for PA? Is it true that $X_{n} = n^{1-o(1)}$ with high probability for PA? 
Are there natural models for which $X_{n} \leq n^{1-\eps}$ with high probability for some fixed $\eps > 0$? 

At a technical level, can the Ulam--Harris common subtree and related arguments (Sections~\ref{sec:UHLB} and~\ref{sec:beyondUH}) be analyzed for other models beyond UA, such as for PA? 
In a different direction, can the simple recursive argument (Sections~\ref{sec:polyLB} and~\ref{sec:linear}) be improved by incorporating aspects of the attachment rule (e.g., incorporating degrees)?

\item \textbf{The structure of the largest common subtree.} 
Beyond the size of the largest common subtree, 
what can we say about its structure? 
For instance, how large is its diameter? 
What is its degree distribution? 
How large is its maximum degree? How many leaves does it have?

Another perspective is to study the complement of the largest common subtree---note that if we remove the largest common subtree from the two trees, then they both break into a forest of small trees. 
What can we say about these forests? For instance, how many components do they have? How large is their largest component? What is the distribution of component sizes?

\item \textbf{Three or more trees.} 
How large is the largest common subtree of $k \geq 3$ independent UA trees with $n$ vertices each? 
The Ulam--Harris common subtree shows that this is at least $n^{k(2^{1/k}-1) - o(1)}$ in expectation (note that the exponent $k(2^{1/k}-1)$ is a decreasing function of~$k$, converging to $\log 2 \approx 0.693$ as $k \to \infty$), see Remark~\ref{remark:many_trees} for details.

\item \textbf{General graphs.} This paper focuses on tree growth models (in part for simplicity), but all the questions studied and asked here can be studied for general graph growth models (replacing common subtree with common subgraph). 

\end{itemize}

\begin{figure}
    \centering
    \includegraphics[width=0.43\textwidth]{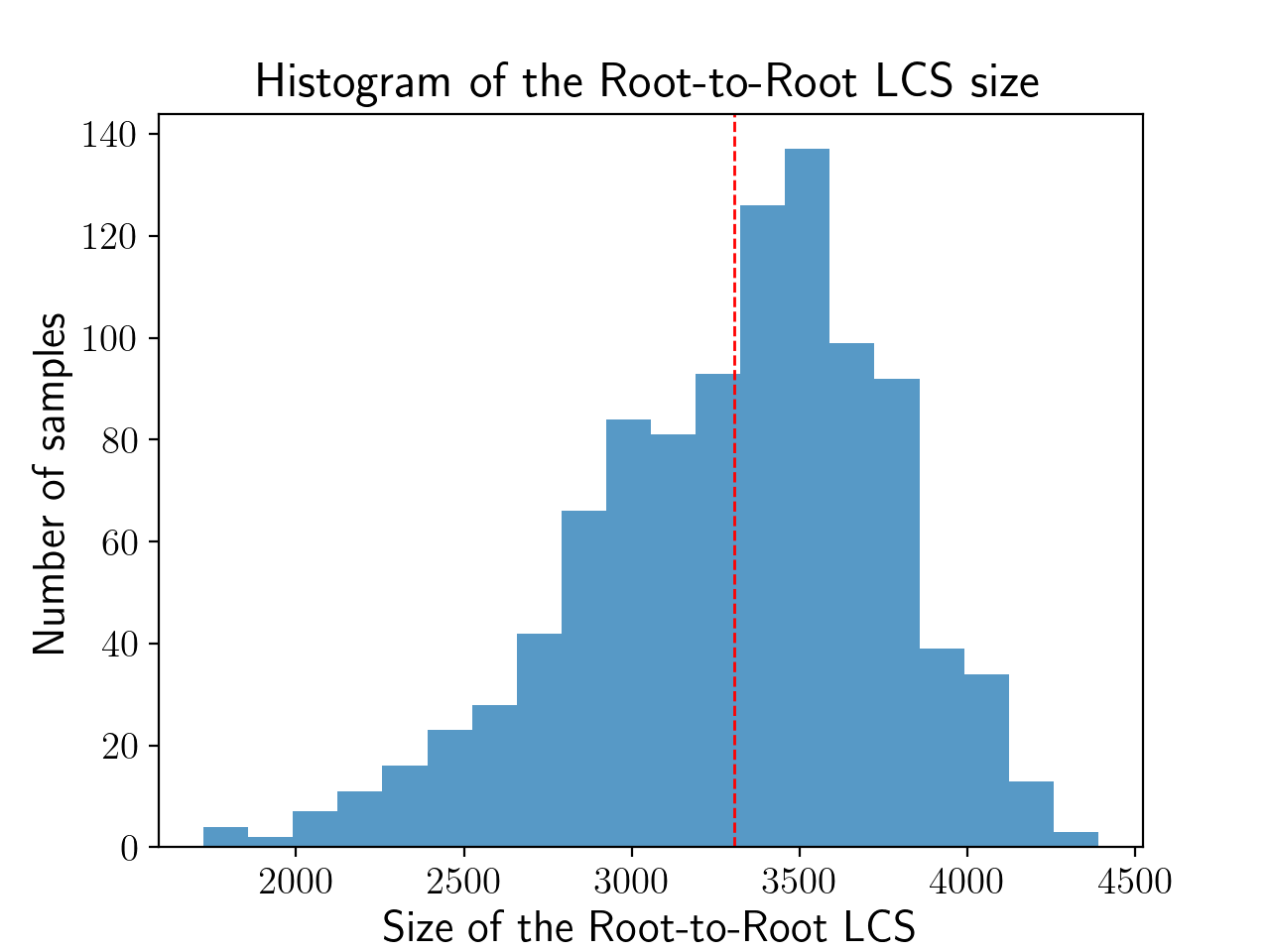}
    \caption{Empirical distribution of the \emph{Root-to-Root LCS} for two independent UA trees on $n = 10000$ vertices, generated by $1000$ independent trials. 
    The histogram was generated using 20 bins, and the dotted red line denotes the empirical mean of the \emph{Root-to-Root LCS}, which is $3304$.}
    \label{fig:histogram_normalized_expectation}
\end{figure}

\subsection{Related work}\label{sec:related}

Largest common substructures of two instances of a combinatorial structure have been studied in probabilistic combinatorics for several decades. 
Aldous~\cite{AldousOpenProblem_substructures} formulated a general setting and discussed several examples, including graphs, permutations, partial orders, and cladograms.  

The largest common subtree (also referred to as the maximum agreement subtree) of two independent uniformly random cladograms (leaf-labeled binary trees) has received particular attention~\cite{AldousOpenProblem_substructures,bryant2003size,bernstein2015bounds,misra2019bounds,bordewich2022maximum,aldous2022largest,pittel2023expected,khezeli2024improved,budzinski2024maximum}, motivated by applications in phylogenetics (e.g.,~\cite{de2007congruence}). 
Here Aldous~\cite{AldousOpenProblem_substructures} conjectured that the expected size of the largest common subtree is $n^{\gamma + o(1)}$ for some $\gamma < 1/2$ (where $n$ is the number of leaves), 
and Bryant, McKenzie, and Steel~\cite{bryant2003size} provided numerical evidence that $\gamma$ is close to $1/2$. 
Bryant, McKenzie, and Steel~\cite{bryant2003size} showed an upper bound of order $n^{1/2}$, which was only recently improved to $n^{1/2-\eps}$ for some $\eps > 0$ by Budzinski and S\'enizergues~\cite{budzinski2024maximum}. 
Bernstein~et~al.~\cite{bernstein2015bounds} proved a lower bound of order $n^{1/8}$, which was improved to roughly $n^{0.366}$ by Aldous~\cite{aldous2022largest}, and most recently to $n^{0.4464}$ by Khezeli~\cite{khezeli2024improved}. 
Several variations of this question have been explored, such as considering the Yule--Harding distribution instead of the uniform distribution~\cite{bryant2003size,bernstein2015bounds}, assuming that the two trees have the same shape~\cite{misra2019bounds}, and considering more general leaf-labeled random rooted trees arising from branching processes~\cite{pittel2023expected}.

Our work is similar in spirit to the works above, but there are some significant differences. 
For one, the structure of UA trees, and more generally randomly growing trees, is quite different from that of the trees studied above. 
Moreover, in all of the works above, the structures and substructures are \emph{labeled}, 
whereas in this paper we consider \emph{unlabeled} structures and substructures. 

Recently, Angel, Atamanchuk, Brandenberger, Donderwinkel, and Khanfir~\cite{angel2025largest} studied the largest common subtree of two independent unlabeled Bienaym\'e trees, showing that its size is of order $n^{1/2}$ under mild moment assumptions. 
The structure of Bienaym\'e trees also differs significantly from that of UA trees, and their largest common subtree is much smaller than that of UA~trees. 

In another line of work, Chatterjee and Diaconis~\cite{chatterjee2023isomorphisms} and Surya, Warnke, and Zhu~\cite{surya2025isomorphisms} studied the largest common induced subgraph of two independent dense Erd\H{o}s--R\'enyi random graphs (among other questions), 
proving rather precise results on its size (in particular, two-point concentration). 
Subsequently, 
Gamarnik, R\'acz, and Schoenbach~\cite{GRS26} studied the algorithmic question of efficiently finding a large common induced subgraph, providing evidence that this problem exhibits a computation-to-optimization gap.
While again similar in spirit, these works are also quite different from ours for several reasons. 
For one, dense graphs are very different from trees. 
Also, in the works above, the independence of edges in Erd\H{o}s--R\'enyi random graphs can be strongly leveraged for analysis, whereas there is significantly more structure in the setting of UA trees / randomly growing trees.

Overall, there is a growing literature on studying largest common subtrees/subgraphs of random trees/graphs, yet, as far as we know, our work is the first to study this question for uniform attachment trees and more broadly for randomly growing graphs.
We hope that our work, and the open problems above, spur further activity in this space.

\subsection{Outline}

The rest of the paper is organized as follows. 
In Section~\ref{sec:polyLB} we present a simple argument that shows that $X_{n}$ grows at least polynomially in $n$. 
In Section~\ref{sec:prelim} we collect some technical preliminaries in preparation for subsequent sections. In Section~\ref{sec:UHLB} we analyze 
the Ulam--Harris common subtree
of two UA trees 
and as a consequence we show that $X_{n} \geq n^{2(\sqrt{2}-1) - \eps}$ with high probability for any $\eps > 0$; in particular, $X_{n} \geq n^{0.828}$. 
In Section~\ref{sec:beyondUH} we then study further local optimizations to obtain a larger common subtree, culminating in a proof of Theorem~\ref{thm:main}. 
Section~\ref{sec:UB} is devoted to upper bound arguments, in particular the proofs of Theorems~\ref{thm:UB_near1} and~\ref{thm:UB_near0}. 
Finally, we prove results for general randomly growing tree models, such as variants of preferential attachment, in Section~\ref{sec:general}.


\section{A simple polynomial lower bound} \label{sec:polyLB} 

In this section we prove a simple polynomial lower bound on the size of the largest common subtree. 
The main idea is to argue recursively, on time intervals that are growing exponentially in length. 
The following claim is in expectation; 
a similar result holds with high probability, but we omit the details, since we will prove stronger high-probability statements in subsequent sections. 

\begin{claim}\label{claim:simple}
For every $n \geq 1$ we have that $\E \left[ X_{n} \right] \geq n^{\log_{2}(5/4)} \geq n^{0.32}$. 
\end{claim}

\begin{proof}
We will show that 
$\E \left[ X_{2^{k+1}} \right] 
\geq \frac{5}{4} \E \left[ X_{2^{k}} \right]$ 
for every $k \geq 1$. 
Since $X_{2} = 2$ 
(because both $T_{2}^{1}$ and $T_{2}^{2}$ are the unique tree on two vertices consisting of a single edge), the claim then follows by iterating this inequality and using monotonicity of $X_{n}$ in $n$: 
\begin{equation}\label{eq:recursion_simple}
\E \left[ X_{n} \right] 
\geq \E \left[ X_{2^{\lfloor \log_{2}(n) \rfloor}} \right] 
\geq \left( \frac{5}{4} \right)^{\lfloor \log_{2}(n) \rfloor - 1} \E \left[ X_{2} \right] 
\geq 2 \left( \frac{5}{4} \right)^{\log_{2}(n) - 2}
\geq \left( \frac{5}{4} \right)^{\log_{2}(n)} 
= n^{\log_{2}(5/4)}. 
\end{equation}

We now argue that 
$\E \left[ X_{2^{k+1}} \right] 
\geq \frac{5}{4} \E \left[ X_{2^{k}} \right]$ 
for every $k \geq 1$. 
We do so by conditioning on $T_{2^{k}}^{1}$ and~$T_{2^{k}}^{2}$. 
Let $S=(V,E)$ be a tree that is a subtree of both $T_{2^{k}}^{1}$ and $T_{2^{k}}^{2}$ with $|S|=|V|=X_{2^k}$.
Let $S^{1}$ and $S^{2}$ denote the copies of $S$ in 
$T_{2^{k}}^{1}$ and $T_{2^{k}}^{2}$, respectively. 
Let $v^{1}$ be a vertex of $S^{1}$ and let $v^{2}$ be the corresponding vertex of $S^{2}$. 

Let us consider the evolution of the first tree from time $2^{k}+1$ to time $2^{k+1}$. 
The probability that no new edge attaches to $v^{1}$ during this time interval is 
\[
\left(1-\frac{1}{2^k}\right) \left(1-\frac{1}{2^k+1}\right) \cdot \ldots \cdot  \left(1-\frac{1}{2^{k+1}-1}\right) 
= \frac{2^k-1}{2^k} \cdot \frac{2^k}{2^k+1} \cdot \ldots \cdot  \frac{2^{k+1}-2}{2^{k+1}-1} 
= \frac{2^k - 1}{2^{k+1}-1}
< \frac{1}{2}.
\]
Thus the probability that at least one new edge attaches to $v^{1}$ during this period is at least~$1/2$. 
Since the two tree growth processes have the same marginal law, the same also holds for $v^{2}$ in $S^{2}$. 
By the independence of the two growth processes, 
the probability that 
both of these events occur 
is at least $(1/2)^{2} = 1/4$. 

Define $\wt{V}$ as the set of vertices $v \in V$ for which at least one edge attaches to it from time $2^{k}+1$ to time $2^{k+1}$ in \emph{both} uniform attachment processes. 
The argument in the previous paragraph, together with linearity of expectation, yields that 
$\E \left[ |\wt{V}| \, \middle| \, T_{2^{k}}^{1}, T_{2^{k}}^{2} \right] 
\geq \frac{1}{4} X_{2^{k}}$. 
Now consider the tree that contains $S=(V,E)$ and where we attach an additional leaf to all vertices $v \in \wt{V}$. 
This tree is contained as a subtree in both $T_{2^{k+1}}^{1}$ and $T_{2^{k+1}}^{2}$ 
and has cardinality $|V|+|\wt{V}|$. 
This shows that 
\[
\E \left[ X_{2^{k+1}} \, \middle| \, T_{2^{k}}^{1}, T_{2^{k}}^{2} \right] 
\geq \E \left[ |V| + |\wt{V}| \, \middle| \, T_{2^{k}}^{1}, T_{2^{k}}^{2} \right] 
\geq \frac{5}{4} X_{2^{k}}.
\]
The claim thus follows by taking expectation. 
\end{proof}

We note that by optimizing the base of the exponent in the proof of Claim~\ref{claim:simple}, the lower bound can be improved to 
$n^{\max_{x > 1} \log_{x}(1+(1-1/x)^{2})} 
\in [n^{0.336},n^{0.337}]$, 
which is only marginally better.

\section{Technical preliminaries} \label{sec:prelim} 

In this section we collect some technical preliminaries that are necessary for the proofs in subsequent sections. 
Specifically, we discuss the Ulam--Harris tree, the associated vertex labeling, and the common subtree it induces between two randomly growing trees. 
We also discuss the embedding of UA trees into continuous-time branching processes. 

\subsection{The Ulam--Harris tree and Ulam--Harris vertex labels} \label{sec:UlamHarris}

The \emph{Ulam--Harris tree} $\cU_{\infty}$ is the infinite rooted tree defined as follows (see, e.g., Harris's book~\cite[Chapter~VI]{harris_book} or Janson's survey~\cite[Section~6]{janson2012simply}).
The root of $\cU_{\infty}$ is the empty string $\varnothing$ and the node set of $\cU_{\infty}$ is 
\[
V_{\infty} := \bigcup_{k \geq 0} \N^{k},
\]
with $\N^0 := \{\varnothing\}$, and where $\N^{k}$ is the set of all length-$k$ strings $(i_{1}, \ldots, i_{k})$ with $i_{1}, \ldots, i_{k} \in \N := \{ 1, 2, 3, \ldots \}$. 
We then connect $(i_{1}, \ldots, i_{k})$ and $(i_{1}, \ldots, i_{k}, i_{k+1})$ by an edge for every $k \geq 0$ and $i_{1}, \ldots, i_{k+1} \in \N$ to obtain the edge set of $\cU_{\infty}$. 
For any $k \geq 1$, the \emph{parent} of vertex $v = (i_{1}, \ldots, i_{k})$ is the vertex $(i_{1}, \ldots, i_{k-1})$. 
Moreover, the \emph{children} of $v = (i_{1}, \ldots, i_{k})$ are the vertices $(i_{1}, \ldots, i_{k}, 1), (i_{1}, \ldots, i_{k}, 2), (i_{1}, \ldots, i_{k}, 3), \ldots$ 
and having them in this order we can view $\cU_{\infty}$ as an infinite \emph{ordered} rooted tree.

The family $\cT$ of ordered rooted trees can be identified with the set of rooted subtrees $T$ of $\cU_{\infty}$ whose vertex set $V(T)$ satisfies the following three properties: 
\begin{enumerate}
\item[\namedlabel{prop:UH1}{(UH1)}] $\varnothing \in V(T)$;
\item[\namedlabel{prop:UH2}{(UH2)}] for every $k \geq 1$, if $(i_{1}, \ldots ,i_{k-1}, i_{k}) \in V(T)$, then $(i_{1}, \ldots, i_{k-1}, j) \in V(T)$ for all $1 \leq j \leq i_{k}$;
\item[\namedlabel{prop:UH3}{(UH3)}] for every $k \geq 1$, if $(i_{1}, \ldots, i_{k-1}, i_{k}) \in V(T)$, then $(i_{1}, \ldots, i_{k-1}) \in V(T)$.
\end{enumerate} 
Furthermore, let $\cT_{n} := \{ T \in \cT : |T| = n \}$ denote the set of all ordered rooted trees with $n$ vertices.

The Ulam--Harris tree is convenient for analysis because a randomly growing tree with $n$ vertices can be naturally viewed as a random element of $\cT_{n}$, that is, as a \emph{random subtree of the Ulam--Harris tree}. 
Specifically, 
the first vertex has Ulam--Harris label~$\varnothing$; 
the second vertex has Ulam--Harris label~$(1)$; 
if the third vertex attaches to the first vertex, then it has Ulam--Harris label~$(2)$, 
otherwise, if it attaches to the second vertex, then it has Ulam--Harris label~$(1,1)$; 
and so on: 
if the $n$-th vertex attaches to vertex $v$ with Ulam--Harris label $(i_{1},\ldots, i_{k})$  
and so far $(j-1)$ vertices have attached to $v$ throughout the process for some $j\geq 1$, 
then the $n$-th vertex gets Ulam--Harris label~$(i_{1} ,\ldots, i_{k}, j)$. 
See Figure~\ref{fig:ulam_harris_example} for an illustration.

There are two perspectives here: 
in addition to viewing a randomly growing tree $T_{n}$ with $n$ vertices as a random element of $\cT_{n}$, 
we can also view the Ulam--Harris labels as a \emph{random labeling} of the vertices of $T_{n}$. 
For a vertex $v \in V(T_{n})$, we write $\ell_{\UH}(v)$ for its Ulam--Harris label as defined above. 
Both of these perspectives will be useful in the following in order to find large common subtrees of two independent randomly growing trees.

For a vertex $v = (i_{1}, \ldots, i_{k}) \in V_{\infty}$, 
we write $|v|_{0} := k$ for its \emph{depth} (i.e., its distance from the root). 
We also call $|v|_{1} := i_{1} + \ldots + i_{k}$ the \emph{level} of node $v$. 
This is relevant when studying randomly growing trees, since $|v|_{1}$ many vertices have to be ``born'' before $v$ can be ``born''.

\subsection{The Ulam--Harris common subtree}\label{sec:UH_common_subtree}

Consider two randomly growing trees $T_{n}^{1}$ and $T_{n}^{2}$. 
Note that the vertices of $T_{n}^{1}$ and $T_{n}^{2}$ (typically) have different (random) Ulam--Harris labels; 
to emphasize this difference, we write $\ell_{\UH}^{1}$ and $\ell_{\UH}^{2}$ for the respective functions that map vertices to Ulam--Harris labels in the two trees. 
Now define 
\begin{equation}\label{eq:UH_matching}
S_{n}^{1} := \left\{ v \in V \left( T_{n}^{1} \right) : \text{there exists } w \in V\left( T_{n}^{2} \right) \text{ such that } \ell_{\UH}^{1}(v) = \ell_{\UH}^{2}(w) \right\}.
\end{equation}
In words, $S_{n}^{1}$ is the subset of vertices in $T_{n}^{1}$ such that there is a vertex in $T_{n}^{2}$ with the same Ulam--Harris label. 
Since Ulam--Harris labels are unique, if for $u\in V(T_n^1)$ there exists a corresponding $w\in V(T_n^2)$ in~\eqref{eq:UH_matching}, then it is unique.
Consequently, we may define 
\[
S_{n}^{2} := \left( \ell_{\UH}^{2} \right)^{-1} \left( \ell_{\UH}^{1} \left( S_{n}^{1} \right) \right),
\]
where the functions act elementwise. 

In words, the Ulam--Harris labels in the two trees induce a \emph{partial matching} between~$V(T_{n}^{1})$ and~$V(T_{n}^{2})$: 
namely, the subsets $S_{n}^{1} \subseteq V(T_{n}^{1})$ and $S_{n}^{2} \subseteq V(T_{n}^{2})$ are matched according to the bijection given by $\left( \ell_{\UH}^{2} \right)^{-1} \circ \ell_{\UH}^{1}$. 
We write 
\begin{equation*}\label{eq:Yn}
Y_{n} := \left| S_{n}^{1} \right| = \left| S_{n}^{2} \right|
\end{equation*}
for the size of this partial matching. 
Now observe that the construction implies that the Ulam--Harris labels $\ell_{\UH}^{1}(S_{n}^{1})$ satisfy the properties~\ref{prop:UH1}--~\ref{prop:UH3} in Section~\ref{sec:UlamHarris}, and so do the Ulam--Harris labels $\ell_{\UH}^{2}(S_{n}^{2})$. In particular, this implies that $S_{n}^{1}$ induces a subtree of $T_{n}^{1}$, 
also $S_{n}^{2}$ induces a subtree of $T_{n}^{2}$, 
and, moreover, these subtrees are \emph{isomorphic} (since they are both isomorphic to the same subtree of the Ulam--Harris tree). 
In other words, this gives a common subtree of $T_{n}^{1}$ and $T_{n}^{2}$, 
which we term the 
\emph{Ulam--Harris common subtree}. 
The size of this common subtree is $Y_{n}$, 
showing that 
\begin{equation}\label{eq:UH_LB}
X_{n} \geq Y_{n}. 
\end{equation}

In Section~\ref{sec:UHLB} we analyze $Y_{n}$ for UA trees in order to give a lower bound on $X_{n}$. 
This analysis is facilitated by embedding the UA trees into a continuous-time branching process, 
the background for which we introduce next.

\subsection{Embedding growing trees into continuous-time branching processes}\label{sec:embed}

In this section we provide an embedding of UA trees into continuous-time branching processes (CTBPs), in particular into a Yule process. This embedding allows us to study the common structures in two independent Yule processes, which provides some analytical advantages.

Let $\{\xi_v\}_{v\in V_\infty}$ be a collection of i.i.d.\ homogeneous Poisson processes on $\R_+:=[0,\infty)$ with~rate~$1$. We define the Yule process, written as $\{\bp(t):t\geq 0\}$, as a branching process started with one individual (denoted by $\varnothing$ and called the \emph{root}) at time $t=0$. Every individual $v\in V_\infty$ born into the process (including the root) produces offspring independently according to the point process $\xi_v$. In other words, each individual has a rate $1$ exponential clock, and each time the clock rings, it gives birth to a child (which then independently starts giving birth according to its own independent clock as well).

Let $|\bp(t)|$ denote the number of individuals in the branching process at time $t$. 
Define the stopping times $\{\tau_{n}\}_{n\in\N}$ 
as 
\be \label{eq:stoppingtimes}
\tau_n:=\inf\{t\geq 0: |\bp(t)|=n\}. 
\ee
In words, $\tau_{n}$ is when the $n$-th individual is born into the branching process. Note that $\tau_{1} = 0$. 

The Athreya--Karlin embedding (see~\cite{AthrKar68}) provides us with the following crucial result. 

\begin{lemma}[\cite{AthrKar68}]\label{thrm:embed}
Consider a sequence of UA trees $\{T_{n}\}_{n\in\N}$ and a Yule process $\{\bp(t): t\geq 0\}$ defined above. 
Seen as a sequence of randomly growing labeled trees, it holds that 
$$\left\{T_n:n\in\N \right\}\overset d=\left\{\bp(\tau_n): n\in\N \right\}.$$  
\end{lemma}

Moreover, we can precisely characterize the growth rate of this branching process.  

\begin{theorem}[Growth rate of a Yule process; Sec.~III.7, Theorem 1 in \cite{AthrNey72}]\label{thrm:malt}
Consider a Yule process $\left\{\bp(t): t\geq 0 \right\}$. 
There exists a random variable $W \sim \Exp(1)$ on the same probability space such that 
\be 
|\bp(t)|\e^{- t}\overset{\mathrm{a.s.}}{\longrightarrow}W. 
\ee 
\end{theorem}

\textbf{CTBP as a random subtree of the Ulam--Harris tree.}
By assigning labels to the individuals born in $\{\bp(t):t\geq 0\}$ according to the Ulam--Harris labeling, we can view $\bp(t)$ as a random subtree of the Ulam--Harris tree~$\cU_\infty$. Observe that this labeling of $\bp(t)$ satisfies properties~\ref{prop:UH1} through~\ref{prop:UH3}, as desired.

The embedding into continuous-time branching processes can be extended from UA to more general models of randomly growing trees, with results analogous to Theorems~\ref{thrm:embed} and~\ref{thrm:malt}. 
Since in this paper we mostly focus on UA trees, we do not discuss these extensions here. 
However, we will use such extensions for the analysis of superlinear PA trees later; see Section~\ref{sec:superlinear} for details. 

\textbf{Comparing two trees in continuous time.}
So far, we have only considered a single CTBP. In order to compare two independent UA trees, we will consider two independent CTBPs, which we denote by $A_{1} := \left\{ A_{1} (t) \right\}_{t\geq 0}$ and $A_{2} := \left\{ A_{2} (t) \right\}_{t\geq 0}$. We think of $A_1 (t)$ and $A_2 (t)$ as the subset of vertices in $\mathcal{U}_\infty$ that have been explored until time $t$. 

For fixed $t$, the trees $A_1(t)$ and $A_2(t)$ do not necessarily have the same size. Nonetheless, we still study the size of the largest common subtree of $A_1(t)$ and $A_2(t)$, denoted by $\MCST\left( A_1(t), A_2(t) \right)$. This is useful, as there is an immediate implication from the continuous model to the discrete model about largest common subtrees.

\begin{observation}\label{obs:conttodiscr}
    Let $f:\R_{\geq 0} \to \R_{\geq 0}$ 
    and suppose that
    \begin{align}\label{eq:obs_conttodiscr_aspmt}
        \MCST\left( A_1(t), A_2(t) \right) \geq f(\e^t) \quad \text{ w.h.p. as } t\to \infty. 
    \end{align}
    Then
    \begin{align}\label{eq:obs_conttodiscr}
        X_n = \MCST \left( T_n^1, T_n^2 \right) \geq f(n) \quad \text{ w.h.p. as } n \to \infty.
    \end{align}
\end{observation}

\begin{proof}
    Define $t := \log(n) $. By Theorem~\ref{thrm:malt}, there is a uniformly (in $n$) positive probability that
    \begin{equation}\label{eq:conditioning_L_n}
        \max \left\{ |A_{1} \left( \log(n) \right)| , \left| A_{2} \left( \log(n) \right) \right| \right\}
        \leq \e^{\log(n)} = n.
    \end{equation}
    Call this event $\mathcal{L}_n$. As $\mathcal{L}_n$ has uniform positive probability, $\MCST\left( A_1(\log(n)), A_2(\log(n)) \right)$ is still at least $f\left( \e^{\log(n)} \right) = f(n)$ with high probability as $n \to \infty$ after conditioning on $\mathcal{L}_n$ and assuming \eqref{eq:obs_conttodiscr_aspmt}. However, conditioned on $\mathcal{L}_{n}$, by a coupling argument we have that $\MCST\left( A_1(\log(n)), A_2(\log(n)) \right)$ is stochastically dominated by $\MCST \left( T_n^1, T_n^2 \right)$ (recall~\eqref{eq:conditioning_L_n}), which concludes the proof.
\end{proof}

\subsection{The Ulam--Harris common subtree in continuous time}\label{sec:UH_conts_time}

In Section~\ref{sec:UH_common_subtree} we defined (in discrete time) the Ulam--Harris common subtree of two randomly growing trees.  
Here, we define its continuous-time analogue. 
For a vertex $\sigma \in V_\infty$ and $i\in \{1,2\}$, define
\begin{equation}\label{def:ais}
    A_i (\sigma, t) \coloneqq \mathbbm{1}_{\left\{ \sigma \text{ is born by time $t$ in the $i$-th tree} \right\}},
\end{equation}
that is, $A_i (\sigma, t)$ is the indicator of the event that $\sigma$ was born by time $t$ in the $i$-th tree. 
Observe that, by construction, 
the size of the $i$-th tree at time $t$ is given by 
\[
|A_{i}(t)| = \sum_{\sigma \in V_{\infty}} A_{i}(\sigma, t).
\]
The Ulam--Harris common subtree of $A_{1}(t)$ and $A_{2}(t)$ is induced by the vertices $\sigma \in V_{\infty}$ who are present in both $A_{1}(t)$ and $A_{2}(t)$. 
This latter event occurs precisely when 
$A_{1}(\sigma,t) = A_{2}(\sigma,t) = 1$. 
Thus the size of the Ulam--Harris common subtree of $A_{1}(t)$ and $A_{2}(t)$ is given by 
\begin{equation}\label{eq:UHCS_def}
\UHCS(A_{1}(t), A_{2}(t)) 
:= 
\sum_{\sigma \in V_{\infty}} A_{1}(\sigma, t) A_{2}(\sigma, t).
\end{equation}
Thus the continuous-time analogue of the discrete-time inequality $X_{n} \geq Y_{n}$ (see~\eqref{eq:UH_LB} in Section~\ref{sec:UH_common_subtree}) is that 
\begin{equation}\label{eq:UH_LB_continuous_time}
\LCS(A_{1}(t), A_{2}(t)) 
\geq 
\UHCS(A_{1}(t), A_{2}(t)).
\end{equation}
Analogously to Observation~\ref{obs:conttodiscr} 
we have the following observation (with an analogous proof which we omit). 

\begin{observation}\label{obs:UH_conttodiscr}
    Let $f:\R_{\geq 0} \to \R_{\geq 0}$ 
    and suppose that
    \begin{align*}
        \UHCS(A_{1}(t), A_{2}(t)) \geq f(\e^t) \quad \text{ w.h.p. as } t\to \infty. 
    \end{align*}
    Then
    \begin{align*}
        Y_{n} \geq f(n) \quad \text{ w.h.p. as } n \to \infty.
    \end{align*}
\end{observation}

\section{Lower bound via the Ulam--Harris common subtree} \label{sec:UHLB} 

In this section we obtain a lower bound on $Y_{n}$, 
the size of the Ulam--Harris common subtree of two independent UA trees (see Sections~\ref{sec:UlamHarris} and~\ref{sec:UH_common_subtree} for definitions). 

\begin{theorem}\label{thm:UH_matching}
For every $\eps > 0$ we have that, with high probability, $Y_{n} \geq n^{2(\sqrt{2}-1)-\eps}$. 
In particular, with high probability, we have that $Y_{n} \geq n^{0.828}$. 
\end{theorem}

Due to the inequality $X_{n} \geq Y_{n}$ (see~\eqref{eq:UH_LB} in Section~\ref{sec:UH_common_subtree}), we have the following corollary. 

\begin{corollary}\label{thm:UHLB}
For every $\eps > 0$ we have that, with high probability, $X_{n} \geq n^{2(\sqrt{2}-1)-\eps}$. 
In particular, with high probability, we have that $X_{n} \geq n^{0.828}$. 
\end{corollary}

Throughout this section we will mainly work with the continuous-time trees $A_{1}$ and $A_{2}$ (defined in Sections~\ref{sec:embed} and~\ref{sec:UH_conts_time}). 
Theorem~\ref{thm:UH_matching} is directly implied by the following theorem (which is the continuous-time analogue of Theorem~\ref{thm:UH_matching}) and Observation~\ref{obs:UH_conttodiscr}. 

\begin{theorem}\label{thm:UH_matching_cts}
For every $\eps > 0$ and for all $t>0$ large enough, we have that 
\begin{equation}\label{eq:UH_expected}
\E \left[ \UHCS(A_{1}(t), A_{2}(t)) \right] 
\geq 
\exp \left( \left\{ 2 \left( \sqrt{2} - 1 \right) - \eps \right\} t \right).
\end{equation}
Moreover, for every $\eps > 0$, with high probability as $t\to \infty$, we have that 
\begin{equation}\label{eq:UH_highprob}
\UHCS(A_{1}(t), A_{2}(t)) 
\geq 
\exp \left( \left\{ 2 \left( \sqrt{2} - 1 \right) - \eps \right\} t \right).
\end{equation}
In particular, with high probability, we have that 
$\UHCS(A_{1}(t), A_{2}(t))
\geq \exp( 0.828 t )$. 
\end{theorem}
The rest of this section is devoted to the proof of Theorem~\ref{thm:UH_matching_cts}. 

\subsection{Lower bound on the expectation}\label{sec:UHLB_expectation}

We start by showing~\eqref{eq:UH_expected}. 
While ultimately we are interested in showing the high-probability lower bound~\eqref{eq:UH_highprob}, 
it is useful to first understand the expectation, 
as this highlights the main ideas, 
without the need to deal with some additional technical difficulties.

\begin{proof}[Proof of~\eqref{eq:UH_expected}]
By linearity of expectation and since $A_{1}$ and $A_{2}$ are i.i.d., we have that 
\begin{equation}\label{eq:UHCS_expectation}
\E \left[ \UHCS(A_{1}(t), A_{2}(t)) \right] 
= 
\sum_{\sigma \in V_{\infty}} \p \left( \sigma \text{ born by time } t \text{ in } A_{1} \right)^{2}.
\end{equation}
Note, in particular, that we only have to understand the evolution of a single Yule process $A_{1}$. 
To simplify notation, 
in the following we drop $A_{1}$ from the notation, as it is understood from context. 

Recall that the \emph{level} of a vertex 
$\sigma = (\sigma_{1}, \ldots, \sigma_{\ell}) \in V_{\infty}$ 
in the Ulam--Harris tree equals the sum of its indices: 
$|\sigma|_{1} = \sigma_{1} + \ldots + \sigma_{\ell}$. 
To obtain a lower bound on the expectation in~\eqref{eq:UHCS_expectation}, 
we restrict the sum to vertices on a particular level $\alpha t \in \N$ (where $\alpha$ will be chosen (i.e., optimized) later): 
\begin{equation}\label{eq:UHCS_expectation_LB}
\E \left[ \UHCS(A_{1}(t), A_{2}(t)) \right] 
\geq 
\sum_{\sigma \in V_{\infty}: |\sigma|_{1} = \alpha t} \p \left( \sigma \text{ born by time } t \right)^{2}.
\end{equation}
This is useful, since the probability that a vertex is born by time $t$ is the same for all vertices on a particular level. 
Specifically, we have that
\begin{equation}\label{eq:prob_sum_of_exp}
\p \left( \sigma \text{ born by time } t \right) 
= 
\p \left( \sum_{i=1}^{|\sigma|_{1}} E_{i} \leq t \right),
\end{equation}
where the $\{E_{i}\}_{i \geq 1}$ are i.i.d.\ rate $1$ exponential random variables, 
corresponding to appropriate edges in the Ulam--Harris tree. 
Indeed, for 
$\sigma = (\sigma_{1}, \ldots, \sigma_{\ell})$ 
to be born by time $t$, 
the root needs to give birth to its $\sigma_{1}$-th child, 
then $\sigma_{1}$ needs to give birth to its $\sigma_{2}$-th child, 
then $(\sigma_{1}, \sigma_{2})$ needs to give birth to its $\sigma_{3}$-th child, 
and so on, 
all before time $t$. 
Each birth corresponds to an independent rate~$1$ exponential random variable 
and there are $|\sigma|_{1}$ such random variables in total. 

Since the random variables in the sum are i.i.d., we use Cram\'er's theorem (e.g.,~\cite[Theorem~1.4]{hollander2000large}) 
in order to approximate large deviation probabilities of the form appearing in~\eqref{eq:prob_sum_of_exp}. 
The rate function of an exponential random variable with mean $1$ is given by 
\begin{equation}\label{eq:exp_rate_function}
\Lambda^{\star}_{\exp}(x)
:= \sup_{s\in \R}\left\{ sx-\log\left(\mathbb E \left[\e^{s E_1} \right] \right) \right\}
= x - 1 -\log x,
\end{equation}
for $x > 0$. 
Thus for $\alpha > 1$ we obtain that 
\be \label{eq:cramerbound}
\prob{\sum_{i=1}^{\alpha t}E_i\leq t} 
= \exp \left(-\Lambda_{\exp}^{\star}(1/\alpha)\alpha t + o(t) \right)
= \exp \left(-(1-\alpha+\alpha\log\alpha)t + o(t) \right).
\ee

Next, in order to compute the sum in~\eqref{eq:UHCS_expectation_LB}, we need to count the number of vertices at a particular level. 
For fixed $1 \leq \ell \leq N$, let us count the number of vertices 
at depth $\ell$ 
and at level $N$, 
that is, 
the number of vertices $\sigma$ 
with 
$|\sigma|_{0} = \ell$ 
and 
$|\sigma|_{1} = N$. 
Such vertices correspond exactly to compositions of $N$ with $\ell$ elements, 
and there are $\binom{N-1}{\ell-1}$ such compositions. 
(A composition of $N$ with $\ell$ elements is an ordered sequence of positive integers $a_{1}, \ldots, a_{\ell}$ such that $\sum_{i=1}^{\ell} a_{i} = N$.) 
Summing over the different values of $\ell$, 
the number of vertices $\sigma$ with $|\sigma|_1 = N$ is exactly 
\be \label{eq:numind}
\sum_{\ell=1}^{N}\binom{N-1}{\ell-1}=\sum_{\ell=0}^{N-1}\binom{N-1}{\ell}=2^{N-1}.
\ee 

Putting together~\eqref{eq:UHCS_expectation_LB},~\eqref{eq:prob_sum_of_exp},~\eqref{eq:cramerbound}, and~\eqref{eq:numind}, 
we obtain that 
\begin{align*}
\E \left[ \UHCS(A_{1}(t), A_{2}(t)) \right] 
&\geq 2^{\lfloor \alpha t \rfloor-1} \exp \left(-2(1-\alpha+\alpha\log\alpha)t + o(t) \right) \\
&= \exp \left( \left( - 2 \alpha \log\alpha + (2+\log 2) \alpha - 2 \right) t + o(t) \right)
\end{align*}
for every $\alpha > 1$. 
An elementary computation shows that the function 
$\alpha \mapsto - 2 \alpha \log\alpha + (2+\log 2) \alpha - 2$ 
is maximized at $\alpha = \sqrt{2}$. 
Plugging this value into the display above, we obtain that 
\[
\E \left[ \UHCS(A_{1}(t), A_{2}(t)) \right] 
\geq 
\exp \left( 2 \left( \sqrt{2}-1 \right) t + o(t) \right)
\]
and~\eqref{eq:UH_expected} follows. 
\end{proof}

\begin{remark}[Levels of a Yule process]\label{remark:levels} 
In the proof above we studied the set of vertices at a particular level $\alpha t$. 
There are roughly $2^{\alpha t}$ such vertices, which is an increasing function of $\alpha$. 
Also, the probability 
$\prob{\sum_{i=1}^{\alpha t}E_i\leq t}$ 
is a decreasing function of $\alpha$. 
In the calculations above, these two quantities both appear, 
which results in an optimization problem that presents a tradeoff between high and low values of $\alpha$. 
For the particular optimization problem above, the two effects balance out at $\alpha = \sqrt{2}$. 
We mention that there are other natural choices of $\alpha$ that may be considered: 
\begin{itemize}
\item $\alpha = 1$: 
At $\alpha = 1$, the probability 
$\prob{\sum_{i=1}^{\alpha t}E_i\leq t}$ 
is roughly $1/2$ by the central limit theorem, so a vertex $\sigma \in V_\infty$ with $|\sigma|_{1}=t$ has a uniform positive probability of being born by time~$t$. 
Thus the proof above with the choice $\alpha = 1$ 
results in the lower bound 
$\E \left[ \UHCS(A_{1}(t), A_{2}(t)) \right] 
= \Omega(2^{t})$ 
and thus 
in the lower bound 
$\E[Y_{n}] = \Omega(n^{\log 2}) \geq n^{0.693}$.

\item $\alpha = \sqrt{2}$: This choice maximizes 
$2^{\alpha t} \prob{\sum_{i=1}^{\alpha t}E_i\leq t}^{2}$, 
as shown in the proof above.

\item $\alpha = 2$: This choice maximizes 
$2^{\alpha t} \prob{\sum_{i=1}^{\alpha t}E_i\leq t}$ 
(\emph{without} squaring the probability); in words, $\alpha = 2$ maximizes the expected number of explored vertices up to time $t$ at a given level in a single Yule process. 
Since $2^{\alpha t} \prob{\sum_{i=1}^{\alpha t}E_i\leq t}=e^{t+o(t)}$ for $\alpha = 2$, 
it also holds that the bulk of the explored vertices 
in a Yule process 
up to time $t$ are around level $2t$. 

\end{itemize}
\end{remark}

\begin{remark}[Multiple trees / Yule processes]\label{remark:many_trees}
The argument above also works for three or more Yule processes. 
Indeed, let $T_{n}^{1}, T_{n}^{2}, \ldots, T_{n}^{k}$ be $k$ independent UA trees, and  
let $A_{1}, A_{2}, \ldots, A_{k}$ be the corresponding $k$ independent Yule processes. 
The largest common subtree (LCS) and the Ulam--Harris common subtree (UHCS) of $k$ trees / Yule processes can be defined in the natural analogous way. 
Then, the equation~\eqref{eq:UHCS_expectation} in the proof is simply replaced by 
\[
\E \left[ \UHCS(A_{1}(t), A_{2}(t), \ldots, A_{k}(t)) \right] 
= 
\sum_{\sigma \in V_{\infty}} \p \left( \sigma \text{ born by time } t \text{ in } A_{1} \right)^{k}.
\]
In other words, the square is replaced by a $k$-th power. 
The rest of the proof is analogous. 

By choosing $\alpha = 1$ in the proof (as in Remark~\ref{remark:levels}), 
this shows, for some positive constant $c_{k}$ that depends only on $k$, that 
$\E \left[ \UHCS(A_{1}(t), A_{2}(t), \ldots, A_{k}(t)) \right] 
\geq c_{k} 2^{t}$. 
For the discrete-time process, this gives the lower bound 
$\E \left[ \UHCS \left( T_{n}^{1}, T_{n}^{2}, \ldots, T_{n}^{k} \right) \right] 
\geq c_{k} n^{\log 2} 
\geq c_{k} n^{0.693}$. 

One can do better by optimizing over $\alpha$. The function 
$\alpha \mapsto \alpha \log 2 - 2(1-\alpha + \alpha \log \alpha)$
in the exponent is now replaced by the function 
$\alpha \mapsto \alpha \log 2 - k(1-\alpha + \alpha \log \alpha)$. 
This is maximized at $\alpha = 2^{1/k}$, with the maximum value being $k(2^{1/k}-1)$. 
(Note that, as $k \to \infty$, the maximizer converges to $\alpha = 1$ 
and the maximum value converges to $\log 2$, so for large $k$ this is not much better than the argument in the previous paragraph.) 
Thus we obtain that 
$\E \left[ \UHCS(A_{1}(t), A_{2}(t), \ldots, A_{k}(t)) \right] 
\geq \exp \left( k(2^{1/k}-1) t - o(t) \right)$ 
and also that 
$\E \left[ \UHCS \left( T_{n}^{1}, T_{n}^{2}, \ldots, T_{n}^{k} \right) \right] 
\geq n^{k(2^{1/k}-1) - o(1)}$. 
\end{remark}

\subsection{High probability estimates}\label{sec:UHLB_whp}

With the lower bound on the expectation in hand, 
we now turn to proving the high probability lower bound (\eqref{eq:UH_highprob} of Theorem~\ref{thm:UH_matching_cts}). 
The core of the proof is a second moment argument. 
However, the second moment bound that we prove (see~\eqref{eq:2 plus o moment condition} below) is weaker than what suffices for a standard second moment argument. 
To overcome this, we use the exponential growth property of a Yule process to effectively strengthen the second moment bound (see below for details).

As in Section~\ref{sec:UHLB_expectation}, 
we bound from below the size of the Ulam--Harris common subtree 
by considering only vertices at a particular level. 
Specifically, guided by the optimization done in Section~\ref{sec:UHLB_expectation},
we consider only vertices at level 
$\lfloor \sqrt{2} t \rfloor$: 
\begin{equation}\label{eq:UHCS_LB}
\UHCS(A_{1}(t), A_{2}(t)) 
=
\sum_{\sigma \in V_{\infty}} A_{1}(\sigma, t) A_{2}(\sigma, t) 
\geq 
\sum_{\sigma \in V_{\infty}: |\sigma|_{1} = \lfloor \sqrt{2} t \rfloor} A_{1}(\sigma, t) A_{2}(\sigma, t) =: S_{t}.
\end{equation}
In the following we analyze the sum $S_{t}$. Recall that in Section~\ref{sec:UHLB_expectation} we derived the following estimate on its expectation: 
\begin{equation}\label{eq:S_t_expectation}
\E[S_{t}] 
= \exp \left( 2 \left( \sqrt{2} - 1 \right) t + o(t) \right).
\end{equation}
We will prove the following bound on its second moment: 
\begin{equation}\label{eq:2 plus o moment condition}
\E \left[ S_{t}^{2} \right] = \E \left[ S_{t} \right]^{2+o(1)}.
\end{equation}
Before proving this bound, we first show how to obtain~\eqref{eq:UH_highprob}, assuming that~\eqref{eq:2 plus o moment condition} holds. Subsequently, we will verify that~\eqref{eq:2 plus o moment condition} indeed holds. 
\begin{proof}[Proof of~\eqref{eq:UH_highprob} in Theorem~\ref{thm:UH_matching_cts}, assuming that~\eqref{eq:2 plus o moment condition} holds] 
By the Paley--Zygmund inequality, we have that 
\begin{equation*}
\p \left( S_{t} > \E[S_{t}] / 2 \right) 
\geq \frac{\E[S_{t}]^{2}}{4 \E\left[S_{t}^{2}\right]}.
\end{equation*}
Plugging in~\eqref{eq:2 plus o moment condition} into the denominator
and using the fact that $\E[S_{t}]$ grows exponentially in $t$ (see~\eqref{eq:S_t_expectation}), 
we obtain that 
\begin{equation}\label{eq:prob_not_too_small}
\p \left( S_{t} > \E[S_{t}] / 2 \right) 
\geq \frac{\E[S_{t}]^{2}}{4\E[S_{t}]^{2+o(1)}}
= e^{-o(t)}.
\end{equation}

We now explore the two Yule processes $A_{1}$ and $A_{2}$, and the associated Ulam--Harris trees, up to time $t$. 
We do so in two phases: 
first up to time $\eps t$, 
and then from time $\eps t$ to $t$; 
here $\eps \in (0,1)$ is an arbitrary fixed value. 
Consider the first phase up to time $\eps t$, and 
write $\sigma^{(1)}, \ldots, \sigma^{(K)}$ for the vertices 
$\sigma \in V_{\infty}$ 
such that $|\sigma|_{1} = \lfloor \eps t / 2 \rfloor$ 
and $A_{1}(\sigma, \eps t) = A_{2}(\sigma, \eps t) = 1$. 
By a large deviation estimate (similar to~\eqref{eq:cramerbound}) and a union bound, 
there exists a constant $c = c(\eps) > 0$ such that 
$K \geq e^{ct}$ with high probability as $t \to \infty$. 

Now condition on $(A_{1}(\eps t), A_{2}(\eps t))$ and consider the second phase from time $\eps t$ to $t$. 
For every $i \in [K]$, there is a ``new'' Ulam--Harris tree below vertex $\sigma^{(i)}$, which is being explored by both $A_{1}$ and $A_{2}$ after time $\eps t$. 
Let $Y_{i}$ denote the size of the associated Ulam--Harris common subtree of $A_{1}$ and $A_{2}$ up to time $t$ that is below $\sigma^{(i)}$ and is explored after time $\eps t$. Observe that $Y_{i}$ has the same distribution as $\UHCS(A_{1}((1-\eps)t), A_{2}((1-\eps)t))$. By~\eqref{eq:UHCS_LB} we thus have that $Y_{i}$ has the same distribution as $S_{(1-\eps)t}$ and so,  by~\eqref{eq:prob_not_too_small}, we have that 
\begin{align*}
    \p \left( Y_{i} \geq \E \left[ S_{(1-\eps)t} \right] / 2 \right) = e^{-o(t)}.
\end{align*} 
Note also that, given  $(A_{1}(\eps t), A_{2}(\eps t))$, the random variables $\left\{ Y_{i} \right\}_{i \in [K]}$ are mutually independent, since the corresponding subtrees being explored are disjoint. 
Since $K \geq e^{ct}$ with high probability, 
it follows that, 
with high probability, 
there exists $i_{*} \in [K]$ such that 
$Y_{i_{*}} \geq \E \left[ S_{(1-\eps)t} \right] / 2$. 
Observe that, by construction, $\UHCS(A_{1}(t), A_{2}(t)) 
\geq Y_{i}$ for all $i\in [K]$. 
Furthermore, by~\eqref{eq:S_t_expectation}, 
we have that 
$\E \left[ S_{(1-\eps)t} \right] / 2 = \exp \left( 2(\sqrt{2}-1) (1-\eps) t + o(t) \right)$. 

Combining all of the above, we conclude that, with high probability,  
\[
\UHCS(A_{1}(t), A_{2}(t)) 
\geq Y_{i_{*}} 
\geq \E \left[ S_{(1-\eps)t} \right] / 2
= \exp \left( 2 \left( \sqrt{2} - 1 \right) (1-\eps) t + o(t) \right),
\]
as $t\to \infty$. Since $\eps \in (0,1)$ was arbitrary,  the conclusion follows. 
\end{proof}

To conclude the proof of Theorem~\ref{thm:UH_matching_cts}, it remains to prove that~\eqref{eq:2 plus o moment condition} holds. We begin by introducing some notation. 

\begin{definition}\label{def:ycp}
Let $\eta = (\eta_{1}, \ldots, \eta_{k}), \sigma = (\sigma_{1}, \ldots, \sigma_{\ell}) \in V_{\infty}$ be two vertices of the Ulam--Harris tree. 
We say that $\eta \preccurlyeq \sigma$ if the following three conditions hold: 
$k \leq \ell$, 
for all $i \in \{1, \ldots, k-1\}$ we have that $\eta_{i} = \sigma_{i}$, 
and $\eta_{k} \leq \sigma_{k}$. 
In words, if $\eta$ and $\sigma$ are distinct vertices, then $\eta \preccurlyeq \sigma$ if and only if $\eta$ needs to be born before $\sigma$ can be born. 
For two vertices $\sigma, \sigma^\prime \in V_{\infty}$, we define the \emph{youngest common predecessor} $\ycp(\sigma, \sigma^\prime)$ as the maximal element (with respect to the partial ordering $\preccurlyeq$) of the set
\begin{equation*}
    \left\{ \eta : \eta \preccurlyeq \sigma \right\} \cap 
    \left\{ \eta : \eta \preccurlyeq \sigma^\prime \right\}.
\end{equation*}
\end{definition}

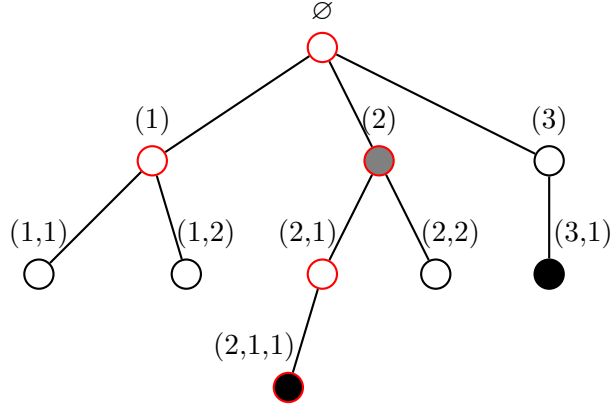
\begin{figure}
	\[\begin{tikzpicture}[scale = 1.5]
		\vertex[thick, red, label = $\varnothing$ ] (0) at (0,3) { };
		
		\vertex[thick, red, label = (1)] (1) at (-1.5,2) { };
		\vertex[thick, black, label = {(1,1)}] (11) at (-2.5,1) { };
		\vertex[thick, black, label = { \ \ \ \ (1,2)}] (12) at (-1.2,1) { };

		\vertex[thick, red, fill=gray, label = (2)] (2) at (0.5,2) { };
		\vertex[thick, red, label = {(2,1)\ \ \ \ }] (21) at (0,1) { };
		\vertex[thick, black, label = {\ \ \ (2,2)}] (22) at (1,1) { };
		\vertex[thick, red, fill=black, label = {(2,1,1)\ \ \ \ \ \ \ \ }] (211) at (-0.3,0) { };

		\vertex[thick, black, label = (3)] (3) at (2,2) { };
		\vertex[thick, black, fill=black, label = {\ \ \ \ \ \ \ (3,1)}] (31) at (2,1) { };
		
		\path[thick]
		(0) edge (1) (1) edge (11)
		(1) edge (12) (0) edge (2)
		(2) edge (21) (2) edge (22)
		(21) edge (211) (0) edge (3) (3) edge (31);
		
	\end{tikzpicture}\]
	\centering
 \caption{The black filled vertices (2,1,1) and (3,1) have the gray filled vertex (2) as their youngest common predecessor. The set $\{\eta : \eta \preccurlyeq (2,1,1)\}=\{\varnothing, (1), (2), (2,1), (2,1,1) \}$ consists of the vertices with the red boundary. Observe that, following the notation of~\eqref{eq:oplus} and~\eqref{eq:oplusprime}, one has $(2,1,1)=(2)\oplus(1,1)$ and $(3,1)=(2)\oplus^\prime (1,1)$.} \label{fig:tree_ycp}
\end{figure}

The concepts defined in Definition~\ref{def:ycp} are illustrated in Figure~\ref{fig:tree_ycp}. As another example, one has that $\ycp\left( (1,4,5), (1,7)  \right) = (1,4)$. Also, if $\sigma \preccurlyeq \sigma^\prime$, then $\ycp(\sigma, \sigma^\prime) = \sigma$.  Observe that for each vertex $\sigma \in V_{\infty}$ and each $r\in \{1,\ldots,|\sigma|_1\}$, there exists exactly one vertex $\eta \in V_\infty$ with $\eta \preccurlyeq \sigma$ and $|\eta|_1=r$; see Figure~\ref{fig:tree_ycp} again for an illustration. Further, $\varnothing \preccurlyeq \sigma$ for every $\sigma \in V_{\infty}$. Moreover, we will use the observation captured in the following lemma.

\begin{lemma}\label{lem:partial_order_expansion}
For every $n, k \geq 1$ and $\eta \in V_{\infty}$ such that $|\eta|_{1} = n$, there are exactly $2^{k}$ many vertices $\sigma \in V_{\infty}$ such that $|\sigma|_{1} = n+k$ and $\eta \preccurlyeq \sigma$. 
\end{lemma}
\begin{proof}
Let $\eta = (\eta_{1}, \ldots, \eta_{j})$, where $j \geq 1$ since $n \geq 1$. Let $\mu = (\mu_{1}, \ldots, \mu_{\ell}) \in V_{\infty}$ be a vertex such that $|\mu|_{1} = k$, and recall that there are $2^{k-1}$ such vertices (see~\eqref{eq:numind}). Then the vertex 
\begin{equation}\label{eq:oplus}    
\eta \oplus \mu \coloneqq (\eta_1,\ldots,\eta_j,\mu_1,\mu_2,\ldots,\mu_\ell) \in V_\infty
\end{equation}
satisfies $|\eta \oplus \mu|_{1} = n + k$ and $\eta \preccurlyeq \eta \oplus \mu$. 
Furthermore, the vertex 
\begin{equation}\label{eq:oplusprime}    
\eta \oplus' \mu \coloneqq (\eta_1,\ldots,\eta_{j-1},\eta_{j}+\mu_1,\mu_2,\ldots,\mu_\ell) \in V_\infty
\end{equation}
also satisfies $|\eta \oplus' \mu|_{1} = n+k$ and $\eta \preccurlyeq \eta \oplus' \mu$. All of the vertices in~\eqref{eq:oplus} and~\eqref{eq:oplusprime} are distinct, which gives $2 \times 2^{k-1} = 2^{k}$ vertices $\sigma \in V_{\infty}$ such that $|\sigma|_{1} = n+k$ and $\eta \preccurlyeq \sigma$. 
Finally, it follows from the definition that every vertex $\sigma \in V_{\infty}$ such that $|\sigma|_{1} = n+k$ and $\eta \preccurlyeq \sigma$ can be written either as in~\eqref{eq:oplus} or as in~\eqref{eq:oplusprime}. 
\end{proof}

With these definitions and observations, we are ready to proceed with the proof of the bound on the second moment~\eqref{eq:2 plus o moment condition}.

\begin{proof}[Proof of \eqref{eq:2 plus o moment condition}] 
Let $\{E_{i}\}_{i\in \N}$ be i.i.d.\ exponential random variables with mean $1$; these will be used throughout the proof. 
To abbreviate notation, we write $V_{\sqrt{2}t} := \left\{\sigma \in V_\infty : |\sigma|_1 = \lfloor \sqrt{2}t \rfloor \right\}$ for the set of vertices at level $\lfloor \sqrt{2}t \rfloor$. 
Recall that $\E \left[ A_{1}(\sigma, t) \right] = \p \left( A_1(\sigma,t)=1 \right)$ (i.e., the probability that $\sigma$ is born by time $t$) is the same for all vertices at the same level (see~\eqref{eq:prob_sum_of_exp}), so in particular the probability is the same for all~$\sigma \in V_{\sqrt{2}t}$. 

Now, rewriting $\E{[S_t^2]}$ by expanding the square of the sum, using linearity of expectation and the independence of $A_{1}(t)$ and $A_{2}(t)$, we get that 
\begin{equation}
    \E{[S_t^2]} = \sum_{\sigma \in V_{\sqrt{2}t}} \sum_{\sigma^\prime \in V_{\sqrt{2}t}}
\E \left[  A_1(\sigma,t) A_1(\sigma^\prime,t)  \right]^{2}.
\end{equation}
We also have that 
\begin{equation}
    \E{[S_t]}^2 = \left(  \sum_{\sigma \in V_{\sqrt{2}t}} \E[A_1(\sigma, t)]^2\right)^2 =  \left( \left| V_{\sqrt{2}t} \right| \E \left[  A_1(\mu,t)  \right]^2  \right)^2 ,
\end{equation}
for $\mu \in V_{\sqrt{2}t}$. Further, using the fact that $\E[S_{t}]$ is exponential in $t$ (see~\eqref{eq:S_t_expectation}), in order to prove~\eqref{eq:2 plus o moment condition}, it suffices to show the following inequality:
\begin{equation}\label{eq:correlation}
\sum_{\sigma \in V_{\sqrt{2}t}} \sum_{\sigma^\prime \in V_{\sqrt{2}t}}
\E \left[  A_1(\sigma,t) A_1(\sigma^\prime,t)  \right]^{2} 
\leq \left( \left| V_{\sqrt{2}t} \right| \E \left[  A_1(\mu,t)  \right]^2  \right)^2 \e^{o(t)},
\end{equation}
for $\mu \in V_{\sqrt{2}t}$. From this we can see that the key is to understand the quantity 
$\E \left[  A_1(\sigma,t) A_1(\sigma^\prime,t)  \right]$ for $\sigma, \sigma' \in V_{\sqrt{2}t}$.

Let $T_{\sigma}$ denote the time at which the vertex $\sigma \in V_{\infty}$ gets explored in the Yule process $A_{1}$, 
and recall that $T_{\sigma}$ has the same distribution as 
$\sum_{i=1}^{|\sigma|_{1}} E_{i}$. 
For two vertices $\sigma, \sigma' \in V_{\sqrt{2}t}$, 
the arrival times $T_{\sigma}$ and $T_{\sigma'}$ are correlated, 
since some of the exponential random variables in the respective sums are shared (i.e., they are the same). 
To control this correlation, we aim to separate out the independent exponential random variables. 
To this end, let 
$\eta := \ycp(\sigma, \sigma^\prime)$ be the youngest common predecessor of $\sigma$ and $\sigma^\prime$ (see Definition~\ref{def:ycp}). 
We can then write 
$T_{\sigma} = T_{\eta} + (T_{\sigma} - T_{\eta})$ 
and 
$T_{\sigma'} = T_{\eta} + (T_{\sigma'} - T_{\eta})$. 
This is useful, since the random variables 
$T_{\eta}$, $T_{\sigma} - T_{\eta}$, and $T_{\sigma'} - T_{\eta}$ 
are mutually independent. 

We then have the following inequality: 
\begin{equation}\label{eq:crucial_ineq}
\E \left[  A_1(\sigma,t) A_1(\sigma^\prime,t)  \right]
= \p \left( T_{\sigma} \leq t, T_{\sigma^{\prime}} \leq t \right) 
\leq \p \left( T_{\sigma} \leq t, T_{\sigma^{\prime}} - T_{\eta} \leq t \right) 
= \p \left( T_{\sigma} \leq t \right) \p \left( T_{\sigma^{\prime}} - T_{\eta} \leq t \right),
\end{equation}
where the last equality is due to independence. 
We note that the inequality in~\eqref{eq:crucial_ineq} is not optimal, but using this inequality leads to computations that are significantly more tractable, and it is sufficient for our purposes. 

Plugging~\eqref{eq:crucial_ineq} into the left hand side of~\eqref{eq:correlation} we thus obtain the following inequality: 
\begin{equation}\label{eq:crucial_ineq2}
\sum_{\sigma \in V_{\sqrt{2}t}} \sum_{\sigma^\prime \in V_{\sqrt{2}t}}
\E \left[  A_1(\sigma,t) A_1(\sigma^\prime,t)  \right]^2
\leq
\sum_{\sigma \in V_{\sqrt{2}t}} \p \left(  T_{\sigma} \leq t \right)^2 \sum_{\eta: \eta \preccurlyeq \sigma} \sum_{\substack{\sigma^\prime \in V_{\sqrt{2}t}: \\ \ycp(\sigma, \sigma^\prime) = \eta}}
     \p \left( T_{\sigma^\prime} - T_\eta \leq t   \right)^{2}.
\end{equation} 
Recall that for each vertex $\sigma \in V_{\sqrt{2}t}$ and each $r \in \{1,\ldots,\lfloor \sqrt{2} t \rfloor\}$, there exists exactly one vertex $\eta \in V_{\infty}$ for which $\eta \preccurlyeq \sigma$ and $|\eta|_1=r$. 
Also, 
while $\varnothing \preccurlyeq \sigma$, 
we cannot have that 
$\ycp(\sigma, \sigma^\prime) = \varnothing$, 
since the vertex $(1) \in V_{\infty}$ is a predecessor of both $\sigma$ and $\sigma^{\prime}$ (i.e., $(1) \preccurlyeq \sigma$ and $(1) \preccurlyeq \sigma^{\prime}$). 
By Lemma~\ref{lem:partial_order_expansion}, 
for fixed vertices $\sigma \in V_{\sqrt{2}t}$ and $\eta \in V_\infty$ with $\eta \preccurlyeq \sigma$ and $|\eta|_{1} = r \geq 1$, there are at most $2^{\lfloor \sqrt{2}t \rfloor - r}$ many vertices $\sigma^\prime \in V_{\sqrt{2}t}$ for which $\ycp(\sigma, \sigma^\prime) = \eta$. 
Observe also that, for such $\sigma'\in V_{\sqrt{2}t}$,
the difference in arrival times $T_{\sigma^\prime}-T_\eta$ 
has the same distribution as $\sum_{i=r+1}^{\lfloor \sqrt{2} t \rfloor} E_{i}$. Thus, for $\sigma \in V_{\sqrt{2}t}$ and $\eta \preccurlyeq \sigma$ with $|\eta|_1 = r$, we have that 
$$\sum_{\substack{\sigma^\prime \in V_{\sqrt{2}t}: \\ \ycp(\sigma, \sigma^\prime) = \eta}}
     \p \left( T_{\sigma^\prime} - T_\eta \leq t   \right)^{2} \leq  2^{\sqrt{2}t - r}
\p \left( \sum_{i=r+1}^{\lfloor \sqrt{2} t \rfloor } E_i  \leq t   \right)^{2}.$$
Together with the other observations of the above paragraph and~\eqref{eq:crucial_ineq2}, we can replace the sum over $\eta$ in~\eqref{eq:crucial_ineq2} by a sum over $r \in  \{1,\ldots,\lfloor \sqrt{2} t \rfloor\}$  and use the fact that $T_{\sigma}$ has the same distribution as $\sum_{i=1}^{\lfloor \sqrt{2} t \rfloor} E_{i}$ to get that 
\begin{equation}\label{eq:2nd_moment_bound}
\sum_{\sigma \in V_{\sqrt{2}t}} \sum_{\sigma^\prime \in V_{\sqrt{2}t}}
\E \left[  A_1(\sigma,t) A_1(\sigma^\prime,t)   \right]^2
\leq
\left|V_{\sqrt{2}t}\right| \p \left( \sum_{i=1}^{\lfloor \sqrt{2} t \rfloor} E_i \leq t \right)^2 \sum_{r=1}^{\lfloor \sqrt{2}t \rfloor} 2^{\sqrt{2}t - r}
\p \left( \sum_{i=r+1}^{\lfloor \sqrt{2} t \rfloor } E_i  \leq t   \right)^{2}.
\end{equation} 
Comparing the right hand side of~\eqref{eq:2nd_moment_bound} with the right hand side of~\eqref{eq:correlation}, 
what remains to be shown is the following bound:
\begin{equation}\label{eq:to show whp identity}
    \sum_{r=1}^{\lfloor \sqrt{2}t \rfloor}
    2^{\sqrt{2}t-r}
    \mathbb{P} \left( \sum_{i=r+1}^{\lfloor \sqrt{2}t \rfloor} E_i \leq t \right)^{2} \leq 
    2^{\sqrt{2} t}
    \p \left( \sum_{i=1}^{\lfloor \sqrt{2} t \rfloor} E_i \leq t \right)^2   \e^{o(t)},
\end{equation}
where we used that $|V_{\sqrt{2}t}| = 2^{\sqrt{2}t + \Theta(1)}$.

To this end, define the function 
    \begin{align*}
        \overline{\Lambda}(x) := \begin{cases}
        x-1-\log(x) &\text{ if } x \in (0,1],\\
        0 &\text{ if } x>1.
        \end{cases}
    \end{align*}
Note that $\overline{\Lambda}(x)$ equals the rate function $\Lambda^{\star}_{\exp}(x)$ of an exponential random variable with mean $1$ 
(see~\eqref{eq:exp_rate_function}) 
on the interval $(0,1]$ and is otherwise $0$. 
Using a standard large deviation tail bound for $r < \lfloor \sqrt{2}t \rfloor - t$ and using the trivial upper bound of $1$ for all other $r$, we have that 
\begin{align*}
\mathbb{P} \left( \sum_{i=r+1}^{\lfloor \sqrt{2}t \rfloor} E_i \leq t \right) 
&\leq \exp \left( - \left( \lfloor \sqrt{2}t \rfloor - r \right) \overline{\Lambda} \left( \frac{t}{\lfloor \sqrt{2}t \rfloor - r} \right) \right) \\
&= \exp \left( - \left( \sqrt{2}t - r \right) \overline{\Lambda} \left( \frac{t}{\sqrt{2}t - r} \right)  + \Theta(1) \right),
\end{align*}
where the second equality just removes the floors at the cost of a constant factor. 
Plugging this bound into the left hand side of~\eqref{eq:to show whp identity} we have that 
\begin{equation*}
\sum_{r=1}^{\lfloor \sqrt{2}t \rfloor}
    2^{\sqrt{2}t-r}
    \mathbb{P} \left( \sum_{i=r+1}^{\lfloor \sqrt{2}t \rfloor} E_i \leq t \right)^{2}
    \leq
    \sum_{r=1}^{\lfloor \sqrt{2}t \rfloor}
    \exp \left( \log(2) (\sqrt{2}t-r) - 2 (\sqrt{2}t-r) \overline{\Lambda}\left( \frac{t}{\sqrt{2}t - r} \right) + \Theta(1) \right).
    \end{equation*}
Assuming that the expression  
$\exp \left( \log(2) (\sqrt{2}t-r) - 2 (\sqrt{2}t-r) \overline{\Lambda}\left( \frac{t}{\sqrt{2}t - r} \right) \right)$, considered as a function of $r$ from $\left[0, \sqrt{2}t\right ]$ to $\R$,
is maximized at $r=0$, we can conclude that 
\begin{equation}\label{eq:sum_r_tail_bound}
        \sum_{r=1}^{\lfloor \sqrt{2}t \rfloor}
    2^{\sqrt{2}t-r}
    \mathbb{P} \left( \sum_{i=r+1}^{\lfloor \sqrt{2}t \rfloor} E_i \leq t \right)^{2}
    \leq \sqrt{2}t \exp \left( \log(2) \sqrt{2}t - 2 \sqrt{2}t \overline{\Lambda}\left( \frac{1}{\sqrt{2}} \right) + \Theta(1) \right).
    \end{equation}
By the large deviation bound~\eqref{eq:cramerbound}, the right hand side of~\eqref{eq:to show whp identity} bounds from above the right hand side of~\eqref{eq:sum_r_tail_bound} for an appropriate choice of the factor $e^{o(t)}$. 
    
Thus we are left to show that the maximizer of the function 
\[
r \mapsto    \exp \left( \log(2) (\sqrt{2}t-r) - 2 (\sqrt{2}t-r) \overline{\Lambda}\left( \frac{t}{\sqrt{2}t - r} \right) \right)
\] 
over $r \in \left[0,\sqrt{2}t\right]$ is $r = 0$.
Equivalently, we will show that the maximizer of the function 
    \begin{align*}
        g(\alpha) := \log(2) \alpha - 2 \alpha \overline{\Lambda}\left( \frac{1}{\alpha} \right) ,
    \end{align*}
    over the domain $\alpha \in \left[ 0, \sqrt{2} \right]$ is $\alpha = \sqrt{2}$. It is clear that the function $g(\alpha)$ is increasing for $\alpha \in [0,1]$, as the function $\overline{\Lambda}\left( \tfrac{1}{\alpha} \right) $ is equal to $0$ for $\alpha \in [0,1]$. For $\alpha > 1$ the function $g$ is still increasing for $|\alpha-1|$ small enough, and further,
    \begin{align*}
        \frac{\dd}{\dd \alpha}g(\alpha) = \frac{\dd}{\dd \alpha} \left\{ \log(2)\alpha -2 \alpha \overline{\Lambda} \left( \frac{1}{\alpha} \right) \right\}
        = \log(2) - 2 \log(\alpha),
    \end{align*}
    which equals $0$ exactly for $\alpha = \sqrt{2}$. We conclude that $g(\alpha)$ is increasing for $\alpha \in  \left[ 0, \sqrt{2} \right]$.
\end{proof}

\section{Going beyond 
Ulam--Harris: 
a lower bound of $n^{0.83}$}
\label{sec:beyondUH}

In this section we 
go beyond the Ulam--Harris common subtree 
considered in the previous section. This allows us to obtain an improved lower bound for the size of the largest common subtree. 
The key idea, described in detail in Section~\ref{sec:switching_subtrees}, is to start with the Ulam--Harris common subtree and then switch certain subtrees in the two trees to locally optimize the random fluctuations in which parts of the trees grow faster. 
Our main result in this section is the following theorem, which, combined with Observation~\ref{obs:conttodiscr}, directly implies Theorem~\ref{thm:main}.

\begin{theorem}\label{thm:ENH_matching_cts}
For all $t>0$ sufficiently large, we have that
\be\label{eq:ENH_expected}
    \E \left[ \MCST\left( A_1(t ), A_2(t) \right)\right] \geq \e^{0.833 t}.
\ee
Moreover, with high probability we have that
\be \label{eq:ENH_highprob}
     \MCST\left( A_1(t ), A_2(t) \right) \geq \e^{0.833 t}.
\ee
\end{theorem}
After describing the common subtree construction in Section~\ref{sec:switching_subtrees}, 
we prove the lower bound on the expectation~\eqref{eq:ENH_expected} in Section~\ref{sec:switching_subtrees_expectation}. 
We then prove the high probability estimate~\eqref{eq:ENH_highprob} in Section~\ref{subsec:High probab}, using a modified second moment argument similar to the one in Section~\ref{sec:UHLB_whp}. 
Some technical calculations are deferred to Section~\ref{subsec: technical statements}. 

\subsection{Locally improving the Ulam--Harris common subtree by switching subtrees}\label{sec:switching_subtrees}

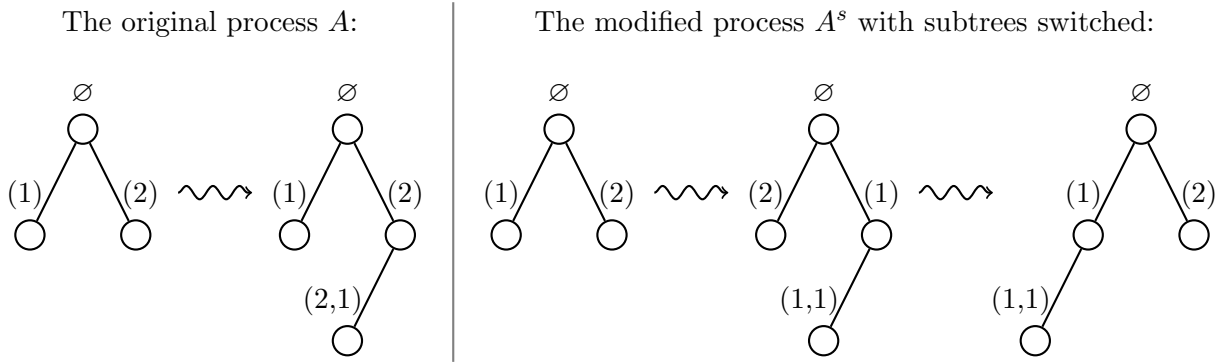
\begin{figure}
	\[\begin{tikzpicture}[scale = 1.4]

            \draw (-2.8,3) node {The original process $A$:};

            \vertex[thick, black, label = $\varnothing$ ] (0a) at (-4,2) {};
            \vertex[thick, black, label = {(1)\ \ } ] (1a) at (-4.5,1) {}; 
            \vertex[thick, black, label = {\ (2)} ] (2a) at (-3.5,1) {};

            \draw[->,thick,decorate,decoration=snake] (-3.1,1.4) -- (-2.4,1.4);

            \vertex[thick, black, label = $\varnothing$ ] (0b) at (-1.5,2) {};
            \vertex[thick, black, label = {(1)\ \ } ] (1b) at (-2,1) {}; 
            \vertex[thick, black, label = {\ (2)} ] (2b) at (-1,1) {};
            \vertex[thick, black, label = {(2,1)\ \ \ \ } ] (21b) at (-1.5,0) {};
 
            \draw[thick, gray] (-0.5,-0.2) -- (-0.5,3.2);

            \draw (3.2,3) node {The modified process $A^{s}$ with subtrees switched:};

            \vertex[thick, black, label = $\varnothing$ ] (0sa) at (0.5,2) {};
            \vertex[thick, black, label = {(1)\ \ } ] (1sa) at (0,1) {}; 
            \vertex[thick, black, label = {\ (2)} ] (2sa) at (1,1) {};

            \draw[->,thick,decorate,decoration=snake] (1.4,1.4) -- (2.1,1.4);

            \vertex[thick, black, label = $\varnothing$ ] (0sb) at (3,2) {};
            \vertex[thick, black, label = {\ (1)} ] (1sb) at (3.5,1) {}; 
            \vertex[thick, black, label = {(2)\ \ } ] (2sb) at (2.5,1) {};
            \vertex[thick, black, label = {(1,1)\ \ \ \ } ] (11sb) at (3,0) {};

            \draw[->,thick,decorate,decoration=snake] (3.9,1.4) -- (4.6,1.4);

            \vertex[thick, black, label = $\varnothing$ ] (0sc) at (6,2) {};
            \vertex[thick, black, label = {(1)\ \ } ] (1sc) at (5.5,1) {}; 
            \vertex[thick, black, label = {\ (2)} ] (2sc) at (6.5,1) {};
            \vertex[thick, black, label = {(1,1)\ \ \ \ } ] (11sc) at (5,0) {};

            \path[thick]
            (0a) edge (1a)
            (0a) edge (2a)
            (0b) edge (1b)
            (0b) edge (2b)
            (2b) edge (21b)
            (0sa) edge (1sa)
            (0sa) edge (2sa)
            (0sb) edge (1sb)
            (0sb) edge (2sb)
            (1sb) edge (11sb)
            (0sc) edge (1sc)
            (0sc) edge (2sc)
            (1sc) edge (11sc);  
		
	\end{tikzpicture}\]
	\centering
 \caption{An illustration of the \emph{switching subtrees} construction. Suppose that the original process $A$ grows as in the left part of the figure, with vertices $\varnothing$, (1), and (2) present, and vertex (2,1) then attaches to vertex (2). Then the modified process $A^{s}$ grows as in the right part of the figure (with the new vertex in $A^{s}$ attaching at the same time as in $A$): the labels of (1) and (2) get switched, and the new vertex attaches to (1) and gets label (1,1); it is also natural to then switch the two subtrees, in order to preserve the increasing ordering of vertex labels among the children of each vertex.} 
 \label{fig:switching_subtrees}
\end{figure}

We begin by describing the common subtree construction that is key to improving the lower bound and is the focus of this section. Using the same notation as in the previous sections, recall that for $\sigma = (\sigma_1,\ldots,\sigma_k)$ we have that $|\sigma|_1 = \sigma_1 + \ldots + \sigma_k$ and $|\sigma|_0 = k$. 
Recall that a Yule process $A = \{ A(t) \}_{t \geq 0}$ starts with a single vertex $\varnothing$ and then all vertices independently create offspring at rate $1$. 
Also, 
we view a Yule process as a continuous-time exploration of the Ulam--Harris tree, so each vertex gets an Ulam--Harris label when it is born, and $A(t)$ is the subset of the Ulam--Harris tree that has been explored until time $t$. 

Now, given a Yule process $A = \{ A(t) \}_{t \geq 0}$, we define a new process $A^{s} = \{ A^{s}(t) \}_{t \geq 0}$ by 
relabeling the vertices of $A$, as described below (and illustrated in Figure~\ref{fig:switching_subtrees}). 
This transformation can be equivalently understood as switching the order of certain subtrees---hence the superscript $s$ which stands for \textit{switching}. 

Before describing the \emph{switching subtrees} construction, we introduce some notation which will be helpful. 
For a vertex $\sigma \in V_{\infty}$, let $T_{\sigma}$, respectively $T_{\sigma}^s$, denote the time at which the vertex $\sigma$ gets explored in the Yule process $A$, respectively $A^s$; 
we set $T_{\varnothing} :=0$ and $T_{\varnothing}^{s} := 0$. 
The construction is such that for every fixed $T \geq 0$, we have that 
$\{ A^{s}(t) \}_{t \in [0,T]}$ is a function of $\{ A(t) \}_{t \in [0,T]}$. In particular, for all $t \geq 0$, $A^{s}(t)$ is a relabeling of $A(t)$, hence the two trees $A(t)$ and $A^{s}(t)$ are isomorphic. 
For all $t \geq 0$, we let $\pi_{t}: A(t) \to A^{s}(t)$ denote the isomorphism from $A(t)$ to $A^{s}(t)$.
Observe that the process $A$ only changes at the stopping times $\{ T_{\sigma} \}_{\sigma \in V_{\infty}}$; 
as we shall see, the same is true for $A^{s}$ and the process $\{ \pi_{t} \}_{t \geq 0}$, so it suffices to describe the construction at the stopping times $\{ T_{\sigma} \}_{\sigma \in V_{\infty}}$. 

We start the description of the \emph{switching subtrees} construction by describing the initial time evolution of the process. 
At time $t = 0$, $A(0)$ consists of just the root vertex $\varnothing$, and so does $A^{s}(0)$. 
The initial isomorphism $\pi_{0}$ is the identity, with $\pi_{0}(\varnothing) = \varnothing$; 
in fact, the construction fixes the root throughout time, so $\pi_{t}(\varnothing) = \varnothing$ for all $t \geq 0$. 
We know that (deterministically) the first (non-root) vertex to get explored by $A$ is the vertex $(1)$, which happens at time $T_{(1)}$. 
We set $T_{(1)}^{s} := T_{(1)}$, so the vertex $(1)$ gets explored by $A^{s}$ at the same time. 

Subsequently, the next vertex to get explored by $A$ is either $(1,1)$ or $(2)$, depending on whether $T_{(1,1)} < T_{(2)}$ or $T_{(1,1)} > T_{(2)}$ 
(note that almost surely no two vertices get explored at the same time); 
we distinguish two cases based on this. 
If $T_{(1,1)} < T_{(2)}$, 
then we set $T_{(1,1)}^{s} := T_{(1,1)}$ and $T_{(2)}^{s} := T_{(2)}$; 
in words, the vertices $(1,1)$ and $(2)$ get explored by $A^{s}$ at the same time as they do by $A$. 
Furthermore, 
we set 
$\pi_{t}((1)) := (1)$ for all $t \geq T_{(1)}$ 
and 
$\pi_{t}((2)) := (2)$ for all $t \geq T_{(2)}$; 
that is, in this case, the construction fixes the vertices $(1)$ and $(2)$ throughout time.

The more interesting case is when $T_{(1,1)} > T_{(2)}$, which we describe next. This 
will showcase~the key idea of the \emph{switching subtrees} construction, see Figure~\ref{fig:switching_subtrees} for an illustration. 
In this case we further distinguish between two cases, depending on whether $(1,1)$ or $(2,1)$ gets explored by $A$~earlier. 
\begin{itemize}
\item \textbf{Case 1: $T_{(1,1)} < T_{(2,1)}$.} In this case, $A^{s}$ is identical to $A$ in its next step. 
That is, we let $T_{(1,1)}^{s} := T_{(1,1)}$; in words, $(1,1)$ gets explored by $A^{s}$ at the same time as it does by $A$. 
We also let $T^{s}_{(2)} := T_{(2)}$. 
Furthermore, 
we set 
$\pi_{t}((1)) := (1)$ for all $t \geq T_{(1)}$ 
and 
$\pi_{t}((2)) := (2)$ for all $t \geq T_{(2)}$; 
that is, in this case, the construction fixes the vertices $(1)$ and $(2)$ throughout time---no switching occurs between $(1)$ and $(2)$.

\item \textbf{Case 2: $T_{(1,1)} > T_{(2,1)}$.} This is the interesting case: at time $T_{(2,1)}$, when vertex $(2,1)$ appears in $A$, in the process $A^{s}$ we let the vertex $(1,1)$ appear instead. 
We can think of this modified process as a switching/relabeling of the vertices $(1)$ and $(2)$ and their corresponding subtrees. 

More formally, in this case we set $T_{(1,1)}^{s} := T_{(2,1)}$; 
in words, $(1,1)$ gets explored by $A^{s}$ when $(2,1)$ gets explored by $A$. 
Furthermore, we set 
\begin{equation*}
    \pi_{t} ((1)) := 
    \begin{cases}
        (1) &\mbox{if } t \in [T_{(1)}, T_{(2,1)}), \\
        (2) &\mbox{if } t \in [T_{(2,1)}, \infty),
    \end{cases}
\end{equation*}
and 
\begin{equation*}
    \pi_{t} ((2)) := 
    \begin{cases}
        (2) &\mbox{if } t \in [T_{(2)}, T_{(2,1)}), \\
        (1) &\mbox{if } t \in [T_{(2,1)}, \infty).
    \end{cases}
\end{equation*}
In words, the vertices $(1)$ and $(2)$ get \emph{switched} in the process $A^{s}$ (compared to $A$) at time~$T_{(2,1)}$. 
This case is illustrated in Figure~\ref{fig:switching_subtrees}. 
Note that we also set 
$\pi_{T_{(2,1)}} \left( (2,1) \right) := (1,1)$.
\end{itemize}

Subsequently, the exploration of the tree by $A$ continues, with the exploration of the tree by~$A^{s}$ following analogously, preserving switches of subtrees that have already been made 
(and which are tracked by the isomorphisms $\{\pi_{t}\}_{t \geq 0}$); 
furthermore,
the subtree switching step (in situations analogous to the one described in Figure~\ref{fig:switching_subtrees}) is done inductively for all $k \in \N$ and all pairs of vertices 
$(\sigma_{1}, \ldots, \sigma_{k}), (\sigma_{1}, \ldots, \sigma_{k-1}, \sigma_{k} + 1) \in \N^{k}$ 
where $\sigma_{k}$ is \emph{odd}. 

Let us now describe more formally the general analogues of the cases described above. 
Let $k \in \N$ and $\sigma_{k}$ be \emph{odd}. 
The vertex $\sigma := (\sigma_{1}, \ldots, \sigma_{k})$ is the analogue of the vertex $(1)$ from the cases above. 
Let 
\[
\eta := \left( \eta_{1}, \ldots, \eta_{k} \right) := \pi_{T_{\sigma}} \left( \sigma \right)
\]
be the vertex explored by $A^{s}$ when $\sigma$ is explored by $A$ (note that $T_{\sigma} = T_{\eta}^{s}$ by definition).
At time $T_{\sigma}$ the vertices 
$(\sigma_{1}, \ldots, \sigma_{k-1}, \sigma_{k} + 1)$ 
and $(\sigma_{1}, \dots, \sigma_{k}, 1)$ 
have not yet been explored by $A$. 
Similarly, the construction satisfies, by induction, 
that at time $T_{\sigma} = T_{\eta}^{s}$, the vertices 
$(\eta_{1}, \ldots, \eta_{k-1}, \eta_{k} + 1)$ 
and $(\eta_{1}, \dots, \eta_{k}, 1)$ 
have not yet been explored by $A^{s}$. 
In all cases considered below, we also set 
$T^{s}_{(\eta_{1}, \ldots, \eta_{k-1}, \eta_{k} + 1)} := T_{(\sigma_{1}, \ldots, \sigma_{k-1}, \sigma_{k} + 1)}$.

First, consider whether 
$(\sigma_{1}, \ldots, \sigma_{k-1}, \sigma_{k} + 1)$ 
or 
$(\sigma_{1}, \dots, \sigma_{k}, 1)$ 
is explored by $A$ earlier. 
If 
$T_{(\sigma_{1}, \dots, \sigma_{k}, 1)} < T_{(\sigma_{1}, \ldots, \sigma_{k-1}, \sigma_{k} + 1)}$, 
then we set 
$T^{s}_{(\eta_{1}, \dots, \eta_{k}, 1)} := T_{(\sigma_{1}, \dots, \sigma_{k}, 1)}$. 
Furthermore, in this case we also set 
$\pi_{t} \left( \sigma \right) := \eta$ 
for all $t \geq T_{\sigma}$ 
and $\pi_{t} \left( (\sigma_{1}, \ldots, \sigma_{k-1}, \sigma_{k} + 1) \right) := (\eta_{1}, \ldots, \eta_{k-1}, \eta_{k} + 1)$ 
for all $t \geq T_{(\sigma_{1}, \ldots, \sigma_{k-1}, \sigma_{k} + 1)}$; 
that is, in this case the two corresponding subtrees are not switched. 

Now consider the case when $T_{(\sigma_{1}, \dots, \sigma_{k}, 1)} > T_{(\sigma_{1}, \ldots, \sigma_{k-1}, \sigma_{k} + 1)}$; 
in this case, we further distinguish between two cases, 
depending on whether 
$(\sigma_{1}, \dots, \sigma_{k}, 1)$ 
or 
$\left(\sigma_{1}, \dots, \sigma_{k-1}, \sigma_{k} + 1, 1 \right)$
gets explored by $A$ earlier. 
The general \textbf{Case 1} and \textbf{Case 2} below are analogous to \textbf{Case 1} and \textbf{Case 2} above.
\begin{itemize}
\item \textbf{Case 1: $T_{(\sigma_{1}, \ldots, \sigma_{k}, 1)} < T_{(\sigma_{1}, \ldots, \sigma_{k-1}, \sigma_{k} + 1, 1)}$.} 
In this case, we let 
$T^{s}_{(\eta_{1}, \dots, \eta_{k}, 1)} := T_{(\sigma_{1}, \dots, \sigma_{k}, 1)}$. 
Also, we set 
$\pi_{t} \left( \sigma \right) := \eta$ 
for all $t \geq T_{\sigma}$ 
and $\pi_{t} \left( (\sigma_{1}, \ldots, \sigma_{k-1}, \sigma_{k} + 1) \right) := \left(\eta_{1}, \ldots, \eta_{k-1}, \eta_{k} + 1 \right)$ 
for all $t \geq T_{(\sigma_{1}, \ldots, \sigma_{k-1}, \sigma_{k} + 1)}$; 
that is, in this case the two corresponding subtrees are not switched.

\item \textbf{Case 2: $T_{(\sigma_{1}, \ldots, \sigma_{k}, 1)} > T_{(\sigma_{1}, \ldots, \sigma_{k-1}, \sigma_{k} + 1, 1)}$.} 
This is the interesting case, in which switching occurs. 
In this case, we set 
$T^{s}_{(\eta_{1}, \ldots, \eta_{k}, 1)} := T_{(\sigma_{1}, \ldots, \sigma_{k-1}, \sigma_{k} + 1, 1)}$, 
so $(\eta_{1}, \ldots, \eta_{k}, 1)$ gets explored by~$A^{s}$ 
when $(\sigma_{1}, \ldots, \sigma_{k-1}, \sigma_{k} + 1, 1)$ gets explored by $A$. 
Furthermore, we set 
\begin{equation*}
    \pi_{t} (\sigma) := 
    \begin{cases}
        \eta &\mbox{if } t \in [T_{\sigma}, T_{(\sigma_{1}, \ldots, \sigma_{k-1}, \sigma_{k} + 1, 1)}), \\
        \left( \eta_{1}, \ldots, \eta_{k-1}, \eta_{k} + 1 \right) &\mbox{if } t \in [T_{(\sigma_{1}, \ldots, \sigma_{k-1}, \sigma_{k} + 1, 1)}, \infty),
    \end{cases}
\end{equation*}
and 
\begin{equation*}
    \pi_{t} ((\sigma_{1}, \ldots, \sigma_{k-1}, \sigma_{k} + 1)) := 
    \begin{cases}
        \left( \eta_{1}, \ldots, \eta_{k-1}, \eta_{k} + 1 \right) &\mbox{if } t \in [T_{(\sigma_{1}, \ldots, \sigma_{k-1}, \sigma_{k} + 1)}, T_{(\sigma_{1}, \ldots, \sigma_{k-1}, \sigma_{k} + 1, 1)}), \\
        \eta &\mbox{if } t \in [T_{(\sigma_{1}, \ldots, \sigma_{k-1}, \sigma_{k} + 1, 1)}, \infty).
    \end{cases}
\end{equation*}
Also, we set 
$\pi_{T_{(\sigma_{1}, \ldots, \sigma_{k-1}, \sigma_{k} + 1, 1)}}((\sigma_{1}, \ldots, \sigma_{k-1}, \sigma_{k} + 1, 1)) := (\eta_{1}, \ldots, \eta_{k}, 1)$. 
\end{itemize}

An algorithm describing this procedure is given in Algorithm~\ref{alg:switching_subtrees} below. 
We denote the resulting process as $A^{s} = \{ A^{s}(t) \}_{t \geq 0}$, 
and we again view this as a process on trees. 
Observe that each switching step is an isomorphism on trees, 
and since the transformation from $A(t)$ to $A^{s}(t)$ is a composition of these switching steps, 
it follows that 
the trees $A(t)$ and $A^{s}(t)$ are isomorphic for every $t \geq 0$. 
In fact, there is an isomorphism that maps the root to the root, since each switching step keeps the root in place.

In the construction above, $A^{s}$ is described as a relabeling of $A$. 
However, in addition, it is also natural to subsequently reorder (i.e., \emph{switch}) the subtrees of $A^{s}(t)$ in a way that the children of each vertex are in increasing order. 
This is illustrated in the last step of Figure~\ref{fig:switching_subtrees}. 
Note that after the reordering of subtrees in this way, the resulting rooted, labeled tree is also \emph{ordered}, and hence can be viewed as a subtree of the Ulam--Harris tree; this perspective will be useful in what follows. 
With a slight abuse of notation, 
we will also denote by $A^{s}(t)$ the tree after reordering the subtrees. 
Clearly, reordering the subtrees does not change the isomorphism class of the tree. 


\begin{breakablealgorithm}
\caption{Vertex labels via switching subtrees}
\label{alg:switching_subtrees}
\begin{algorithmic}[1]
\Require{A rooted tree $T$ equipped with arrival times for each vertex and a vertex $v$}
\Ensure{Label $\pi(v)$ of $v$}
\State If $v$ is the root, return $\pi(v)= \varnothing$.
\State Let $\mathsf{par}(v)$ be the parent of $v$, and let $\cC$ be the set of children of $\mathsf{par}(v)$.
\State Let $m$ be the number of vertices in $\cC$ which arrive before $v$.
\State If $m$ is even:
\begin{itemize}
\item Let $u$ be the first vertex in $\cC$ born after $v$. If there does not exist such $u$, set $\pi(v) \leftarrow \pi( \mathsf{par}(v)) \oplus (m + 1)$ and skip the remaining steps.
\item If $v$ has no children and $u$ has at least one child, or if the first child of $u$ is born before the first child of $v$, then set 
$\pi(v) \leftarrow \pi( \mathsf{par}(v)) \oplus (m + 2)$.
\item Else, set $\pi(v) \leftarrow \pi( \mathsf{par}(v)) \oplus (m+1)$.
\end{itemize}
\State If $m$ is odd:
\begin{itemize}
\item Let $u$ be the last vertex in $\cC$ born before $v$.
\item If $u$ has no children and $v$ has at least one child, or if the first child of $v$ is born before the first child of $u$, then set 
$\pi(v) \leftarrow \pi( \mathsf{par}(v)) \oplus (m)$.
\item Else, set $\pi(v) \leftarrow \pi( \mathsf{par}(v)) \oplus (m+1)$.
\end{itemize}
\State Return $\pi(v)$.
\end{algorithmic}
\end{breakablealgorithm}

Now recall that we are interested in two independent Yule processes, 
$A_{1} = \{ A_{1}(t) \}_{t \geq 0}$ and $A_{2} = \{ A_{2}(t) \}_{t \geq 0}$. 
We can then define the switched processes 
$A_{1}^{s} = \{ A_{1}^{s}(t) \}_{t \geq 0}$ and $A_{2}^{s} = \{ A_{2}^{s}(t) \}_{t \geq 0}$
as above. 
Since $A_{1}(t)$ and $A_{1}^{s}(t)$ are isomorphic, 
and so are $A_{2}(t)$ and $A_{2}^{s}(t)$, 
it follows that (the size of) the largest common subtree of $A_{1}(t)$ and $A_{2}(t)$ 
is the same as that of $A_{1}^{s}(t)$ and $A_{2}^{s}(t)$. 
To estimate the latter, we can again turn to the Ulam--Harris common subtree---but note, importantly, that the Ulam--Harris common subtree of $A_{1}^{s}(t)$ and $A_{2}^{s}(t)$  
is different from the Ulam--Harris common subtree of $A_{1}(t)$ and $A_{2}(t)$, 
since the switched processes are relabelings of the original processes. 
In fact, the switching process is defined in a way so that the Ulam--Harris common subtree of the switched trees is typically larger than the Ulam--Harris common subtree of the original trees. 
Summarizing this paragraph, we have that
\begin{equation}\label{eq:switched_UH_LB}
\LCS(A_{1}(t), A_{2}(t)) 
= \LCS(A_{1}^{s}(t), A_{2}^{s}(t)) 
\geq \UHCS(A_{1}^{s}(t), A_{2}^{s}(t)) 
= \sum_{\sigma \in V_{\infty}} A_{1}^{s}(\sigma, t) A_{2}^{s}(\sigma,t),
\end{equation}
where $A_{i}^{s}(\sigma,t)$ is the indicator variable of the event that $\sigma$ is born by time $t$ in $A_{i}^{s}$, for $i \in \{1,2\}$. 

To analyze the right hand side of~\eqref{eq:switched_UH_LB}, it will be useful to give an equivalent description of the switching process in terms of arrival times, as follows. 
First, consider the Yule process $A$, and recall that for a vertex $\sigma \in V_{\infty}$, 
its arrival time is denoted by $T_{\sigma}$. 
For a vertex $\sigma \in V_{\infty} \setminus \{\varnothing\}$, 
we write $W_{\sigma}$ for the time that elapses between the birth of the direct predecessor of $\sigma$ and the birth of  $\sigma$. That is, if $\sigma = (\sigma_{1}, \ldots, \sigma_{k})$ and $\sigma_{k} = 1$, 
then $W_{\sigma} := T_{\sigma} - T_{(\sigma_{1}, \ldots, \sigma_{k-1})}$, and if $\sigma = (\sigma_{1}, \ldots, \sigma_{k})$ and $\sigma_{k} \geq 2$, 
then $W_{\sigma} := T_{\sigma} - T_{(\sigma_{1}, \ldots, \sigma_{k} - 1)}$. 
In addition, define $W_{\varnothing} := 0$ for notational convenience. 
With this notation we have that 
\begin{equation}\label{eq:arrival_time_sum_of_iid}
T_{\sigma} = \sum_{\eta: \eta \preccurlyeq \sigma} W_{\eta},
\end{equation}
where recall that 
$\{W_{\eta}\}_{\eta: \eta \preccurlyeq \sigma, \eta \neq \varnothing}$ 
are i.i.d.\ exponential random variables with mean $1$.

We now reformulate the switching process  $A^{s}$ in terms of arrival times. 
Again we view $A^{s} = \{ A^{s}(t) \}_{t \geq 0}$ as a random tree growing in continuous time, 
and in particular for every $t \geq 0$ we view $A^{s}(t)$ as a subtree of the Ulam--Harris tree. 
Each active vertex produces an offspring at rate $1$, independent of all other vertices, with a slight twist as described next. 
Let $\sigma = (\sigma_{1}, \ldots, \sigma_{k})$ with $\sigma_{k}$ odd. 
Suppose that $(\sigma_{1}, \ldots, \sigma_{k-1}, \sigma_{k} + 1)$ appears in $A^{s}$ before $(\sigma_{1}, \ldots, \sigma_{k},1)$. 
Then, we wait until either $(\sigma_{1}, \ldots, \sigma_{k})$ or $(\sigma_{1}, \ldots, \sigma_{k-1}, \sigma_{k}+1)$ creates a new child, 
and at the time when this occurs, no matter which of the two vertices birthed the new child, 
we give the new child label $(\sigma_{1}, \ldots, \sigma_{k},1)$ and attach it to $\sigma$. 
This is equivalent to the switching construction that is described above and illustrated in Figure~\ref{fig:switching_subtrees}. 
After such a switching step, vertices continue to produce offspring at rate~$1$, and such switching steps continue to happen whenever the situation arises. 

This different perspective changes the distribution of arrival times compared to~\eqref{eq:arrival_time_sum_of_iid}, as follows.  
Suppose that $\sigma = (\sigma_{1}, \ldots, \sigma_{k})$ appears at time $T_{\sigma}^{s}$ in $A^{s}$. 
If $\sigma_{k}$ is odd, 
then the vertex $(\sigma_{1}, \ldots, \sigma_{k},1)$ appears at time 
\[
T_{\sigma}^{s} + \min \left\{ W_{(\sigma_{1}, \ldots, \sigma_{k},1)}, W_{(\sigma_{1}, \ldots, \sigma_{k-1}, \sigma_{k}+1)} + W_{(\sigma_{1}, \ldots, \sigma_{k-1}, \sigma_{k}+1,1)} \right\}
\]
in $A^{s}$. 
On the other hand, if $\sigma_{k}$ is even, 
then the vertex $(\sigma_{1}, \ldots, \sigma_{k},1)$ appears in $A^{s}$ at time 
\[
T_{(\sigma_{1}, \ldots, \sigma_{k-1}, \sigma_{k}-1)}^{s} + \max \left\{ W_{(\sigma_{1}, \ldots, \sigma_{k-1}, \sigma_{k}-1,1)}, W_{\sigma} + W_{(\sigma_{1}, \ldots, \sigma_{k},1)} \right\}.
\]
Inductively, a general vertex 
$\sigma = (\sigma_{1}, \ldots, \sigma_{k})$ thus appears in $A^{s}$ at time 
\begin{align}
    &  
    \sum_{\substack{i=1: \\ \sigma_{i} \text{ odd}}}^{k-1} \min \left\{ W_{(\sigma_1,\ldots,\sigma_i,1)} , W_{(\sigma_1,\ldots, \sigma_{i-1}, \sigma_i + 1 )} + W_{(\sigma_1,\ldots, \sigma_{i-1}, \sigma_i + 1, 1)} \right\}\label{eq:Tmin}\\
    &
    +
    \sum_{\substack{i=1: \\ \sigma_{i} \text{ even}}}^{k-1} \max \left\{ W_{(\sigma_1,\ldots, \sigma_{i-1}, \sigma_{i} - 1,1)} , W_{(\sigma_1,\ldots,\sigma_i  )} + W_{(\sigma_1,\ldots,\sigma_i, 1)} \right\}\label{eq:Tmax}\\
    & + W_{(1)} 
    +
    \sum_{\substack{i=1: \\ \sigma_{i} \text{ odd}}}^{k-1} \sum_{j=2}^{\sigma_i} W_{(\sigma_1,\ldots,\sigma_{i-1},j)}
    +
    \sum_{\substack{i=1: \\ \sigma_{i} \text{ even}}}^{k-1} \sum_{j=2}^{\sigma_i-1} W_{(\sigma_1,\ldots,\sigma_{i-1},j)}
    +
    \sum_{j=2}^{\sigma_k} W_{(\sigma_1,\ldots,\sigma_{k-1},j)}.\label{eq:Texp}
\end{align}
We write $T^{\min}(\sigma)$, $T^{\max}(\sigma)$, and $T^{\exp}(\sigma)$ for the first, second, and third line of this expression. 
Later, when we study two switched processes $A_{1}^{s}$ and $A_{2}^{s}$, 
we add subscripts $1$ and $2$ to this notation to denote these expressions in the respective trees. 

Note that all the random variables $\{W_{\eta}\}$ appearing in the expressions above are independent exponentially distributed random variables with mean $1$. 
Assume that, for some integers $j$, $\ell$, and~$n$, a vertex $(\sigma_1,\ldots,\sigma_k)$ is such that
\begin{align*}
    &\#\{i\in \{1,\ldots,k-1\} : \sigma_i \text{ odd} \} = j, \\
    &\#\{i\in \{1,\ldots,k-1\} : \sigma_i \text{ even} \} = \ell, \text{ and } \\ 
    &|\sigma|_1 = n.
\end{align*}
In particular, this implies that $j+2\ell \leq n$. Then the arrival time of this vertex is distributed as 
\begin{equation}\label{eq:arrival_time_complicated_sum}
    \sum_{i=1}^{j} \min\{U_i,Y_i+Z_i\} + \sum_{i=1}^{\ell} \max\{U_i^\prime,Y_i^\prime+Z_i^\prime\} + 
    \sum_{i=1}^{n-j-2\ell} U_i^{\prime \prime}
\end{equation}
where all the random variables are i.i.d.\ exponential random variables with mean $1$. Compared with~\eqref{eq:arrival_time_sum_of_iid}, 
we see that the expression in~\eqref{eq:arrival_time_complicated_sum} 
is sometimes stochastically dominated by the expression in~\eqref{eq:arrival_time_sum_of_iid} (e.g., if $\ell = 0$) 
and sometimes the expression in~\eqref{eq:arrival_time_complicated_sum} stochastically dominates the expression in~\eqref{eq:arrival_time_sum_of_iid} (e.g., if $j = 0$). 
Vertices of the former type are beneficial for our estimates, while vertices of the latter type hurt our estimates; 
in other words, there is no free lunch. 
The key to the estimates to come is showing that the benefits of the construction outweigh the contributions coming from vertices with larger passage times.

\subsection{Lower bound on the expectation}\label{sec:switching_subtrees_expectation}

\begin{proof}[Proof of~\eqref{eq:ENH_expected} in Theorem~\ref{thm:ENH_matching_cts}] 
The structure of this proof is similar to the proof of~\eqref{eq:UH_expected} in Section~\ref{sec:UHLB_expectation}, 
but the large deviation calculation (and subsequent optimization) is different and more involved. 
Taking expectation in~\eqref{eq:switched_UH_LB}, using linearity of expectation, and since $A_{1}^{s}$ and $A_{2}^{s}$ are i.i.d., 
we have
\begin{equation}\label{eq:switched_expectation}
\E \left[ \LCS(A_{1}(t), A_{2}(t)) \right] 
\geq \sum_{\sigma \in V_{\infty}} \p \left( \sigma \text{ born by time } t \text{ in } A_{1}^{s} \right)^{2}.
\end{equation}
Note that we only have to understand the evolution of a single switched process $A_{1}^{s}$. Let $V_{N} := \left\{ \sigma \in V_{\infty} : |\sigma|_{1} = N \right\}$ denote the set of vertices at level $N$. 
Recall that we can bound the right hand side in~\eqref{eq:switched_expectation} from below by considering the sum only over $V_{N}$, where $N$ will subsequently be chosen carefully. 
Furthermore, we will only consider a subset of $V_{N}$, as defined next. 
Let $N \in \N$ be an integer multiple of~$6$ and define 
\begin{equation} \label{eq:VNtyp}
    V_N^{\mathtt{typ}} := 
    \left\{ \sigma\in V_N:  |\sigma|_0 =  \frac{N}{2} + 1 , 
     \,\#\left\{ i \in \left[N/2\right]  : \sigma_i \text{ even} \right\}=  \frac{N}{6} ,
     \#\left\{ i \in \left[N/2\right]  : \sigma_i \text{ odd} \right\}=   \frac{N}{3} 
    \right\}.
\end{equation}
Note that for $\sigma \in V_N^{\mathtt{typ}}$ there is no parity constraint on the last coordinate $\sigma_{N/2 +1}$. 
The reason why we define the set $V_N^{\mathtt{typ}}$ like this is because a typical element of $V_N$ is in (or `close to') the set $V_N^{\mathtt{typ}}$, and we can more precisely control the probability with which a vertex $\sigma\in V_N^{\mathtt{typ}}$ is born into the tree before time $t$. Thus we will use the lower bound 
\begin{equation}\label{eq:switched_expectation_LB}
\E \left[ \LCS(A_{1}(t), A_{2}(t)) \right] 
\geq \sum_{\sigma \in V_N^{\mathtt{typ}}} \p \left( \sigma \text{ born by time } t \text{ in } A_{1}^{s} \right)^{2}
\end{equation}
for an appropriately chosen $N$.

As discussed in Section~\ref{sec:switching_subtrees}, more precisely in~\eqref{eq:arrival_time_complicated_sum}, the arrival time of a vertex $\sigma \in V_N^{\mathtt{typ}}$ in the switched process $A_{1}^{s}$ is distributed as
\begin{equation*}
    \sum_{i=1}^{N/3} \min \{ U_i,Y_i+Z_i \} 
    + \sum_{i=1}^{N/6} \max \{ U_i^\prime,Y_i^\prime+Z_i^\prime \}
    + \sum_{i=1}^{N/3} U_i^{\prime \prime},
\end{equation*}
where all of the random variables above are i.i.d.\ exponential random variables with mean $1$. 
Let $\Lambda_{\exp}^{\star}$, $\Lambda_{\min}^{\star}$, and $\Lambda_{\max}^{\star}$ denote the large deviation rate functions of random variables of the form $U$, $\min\{U,Y+Z\}$, and $\max\{U,Y+Z\}$, respectively. 
The former is given in~\eqref{eq:exp_rate_function}, 
while the latter two functions are computed in Lemma~\ref{lemma:ratefunc} in Section~\ref{subsec: technical statements}. 
Note that $\E[U] = 1$, 
$\E[ \min\{U,Y+Z\} ] = 3/4$, and 
$\E[ \max\{U,Y+Z\} ] = 9/4$ (see also Remark~\ref{remark:minmax_expectation}). 
Thus, for any fixed 
$\alpha\in(0,1)$, 
$\beta\in(0,3/4)$, 
$\gamma\in(0,9/4)$, and $z := (\alpha/3+\beta/3+\gamma/6)^{-1}$, 
we have that 
\begin{align}
    &\p \Bigg(\sum_{i=1}^\frac{N}{3} U_i^{\prime \prime}  +
    \sum_{i=1}^\frac{N}{3} \min \left\{ U_i', Y_i' + Z_i' \right\} + 
    \sum_{i=1}^\frac{N}{6} \max \left\{ U_i, Y_i + Z_i \right\} \leq \frac{N}{z} \Bigg)\\
    &\ \geq{}  \p\! \left( \sum_{i=1}^\frac{N}{3} U_i^{\prime \prime} \leq \frac{\alpha N}{3} \right)
    \p\! \left( 
    \sum_{i=1}^\frac{N}{3} \min \left\{ U_i', Y_i' + Z_i' \right\} \leq \frac{\beta N}{3} \right)
     \p\!\left(  
    \sum_{i=1}^\frac{N}{6} \max \left\{ U_i, Y_i + Z_i \right\} \leq \frac{\gamma N}{6} \right) \ \ \label{eq:largedev}
    \\
    &\ = \exp\left( - \frac{N}{3} \Lambda_{\exp}^{\star} \left(\alpha\right) \right)
    \exp\left( - \frac{N}{3} \Lambda_{\min}^{\star} \left(\beta\right) \right)\exp\left( - \frac{N}{6} \Lambda_{\max}^{\star} \left(\gamma\right) \right)
     \exp\left( o(N)\right),\label{eq:largedev_LB_abc}
\end{align}
where the first inequality is due to a union bound and independence, 
and the final line follows 
by Cram\'er's theorem. 
We now choose
\begin{equation}\label{eq:constants}
    \alpha := 0.72,
    \quad  \beta := 0.64,
    \quad  \gamma := 1.85,
    \quad z := \left( \frac{\alpha}{3} + \frac{\beta}{3} + \frac{\gamma}{6} \right)^{-1} = 1.3129\ldots 
\end{equation}
Plugging these values into~\eqref{eq:largedev_LB_abc}, 
and numerically evaluating the large deviation rate functions (in~\eqref{eq:exp_rate_function} and Lemma~\ref{lemma:ratefunc}), 
we obtain for every vertex $\sigma\in V_N^{\mathtt{typ}}$ that 
\begin{equation*}
    \p \left( \sigma \text{ born by time } N/z \text{ in } A_{1}^{s} \right) 
    \geq \e^{ - 0.0292  N - o(N)}.
\end{equation*}

We are now ready to plug this estimate back into~\eqref{eq:switched_expectation_LB}. 
First, recall from~\eqref{eq:numind} that $|V_{N}|=2^{N-1}$. 
We will show in Lemma~\ref{lem:descendant_count_in_VN_discrete} that 
$|V_N^{\mathtt{typ}}|=2^{N-o(N)}$, 
so indeed we do not lose much by considering only vertices in 
$V_N^{\mathtt{typ}}$. 
Now setting $N := 6 \lfloor \tfrac{tz}{6} \rfloor$, 
we obtain that 
\begin{equation*}
\E \left[ \LCS(A_{1}(t), A_{2}(t)) \right] 
\geq \sum_{\sigma \in V_N^{\mathtt{typ}}} \p \left( \sigma \text{ born by time } t \text{ in } A_{1}^{s} \right)^{2} 
\geq 2^{N-o(N)} \e^{ - 2 \times 0.0292  N - o(N)},
\end{equation*}
which is bounded from below by $\e^{0.83336 t}$ 
for all $t$ large enough. This concludes the proof of~\eqref{eq:ENH_expected}. 
\end{proof}

\subsection{High probability estimates}\label{subsec:High probab}

In this section we provide a proof for~\eqref{eq:ENH_highprob} in Theorem~\ref{thm:ENH_matching_cts}, that is, that $\MCST\left( A^s_1(t), A^s_2(t) \right) \geq \e^{0.833 t}$ with high probability. To do so, we verify that a modified version of the second moment estimate~\eqref{eq:2 plus o moment condition} holds for our construction. This is sufficient by the same arguments as before. 

Define the binary entropy function as 
\be \label{eq:binentropy}
    H(x):= - x\log(x)-(1-x)\log(1-x)
\ee 
for $x \in (0,1)$, where $\log$ denotes the natural logarithm (in order to be consistent with the rest of the paper); in addition, let $H(0) := 0$ and $H(1) := 0$.  
Note that for any fixed $x \in (0,1)$, assuming $x N \in \N$, we have the standard estimate 
\begin{equation}\label{eq:binom_coeff_entropy_asymp}
    \binom{N}{xN} 
    \sim 
    \frac{1}{\sqrt{2\pi x(1-x) N}} \exp \left( H(x) N \right),
\end{equation}
which follows from Stirling's formula (and where we use the standard asymptotic notation that $a_{N} \sim b_{N}$ if $\lim_{N \to \infty} a_{N} / b_{N} = 1$). In particular, we have the following lower and upper bounds: 
\begin{equation}\label{eq:binom_coeff_entropy_bounds}
    \exp \left( H(x) N - \eps_{N}\right)
    \leq 
    \binom{N}{xN} 
    \leq
    \exp \left( H(x) N \right),
\end{equation}
for some $\eps_{N}$ satisfying $\eps_{N} = \Theta \left( \log N \right)$. 
We will use these estimates throughout the following sections. Another useful property about the binary entropy function $H$ is that for each fixed $y>0$, the function
\begin{align}\label{eq:increasing weighted bin entrop}
    x \mapsto (x+y)H\left( \frac{y}{x+y} \right) = \lim_{N\to \infty} \frac{1}{N} \log \left( \binom{(x+y)N}{y N} \right)
\end{align}
is increasing in $x \in (0,\infty)$.
We start with a technical lemma, which is the analogue of Lemma~\ref{lem:partial_order_expansion}.

\begin{lemma}\label{lem:descendant_count_in_VN_discrete}
Fix $\eta \in V_{\infty}$. 
Let $k := |\eta|_{0}$, 
$\ell := |\eta|_{1}$, 
and 
$r := \left| \left\{ i \in [k] : \eta_{i} \text{ is odd} \right\} \right|$. 
Let $N$ be an integer multiple of $6$, 
which satisfies that
$1 \leq k \leq N/2$, 
$r \leq N/3$, 
and $k-r \leq N/6$. 
If, in addition, $N$ is such that
$\ell \geq N/3 + (r + 2(k-r))$, 
then 
$\left| \left\{ \sigma \in V_{N}^{\mathtt{typ}} : \eta \preccurlyeq \sigma \right\} \right| = 0$.
On the other hand, if $\ell < N/3 + (r + 2(k-r))$, 
then 
\begin{multline}\label{eq:counting_lemma_bound}
\left| \left\{ \sigma \in V_{N}^{\mathtt{typ}} : \eta \preccurlyeq \sigma \right\} \right| \\
\leq 
4 N^{4} \exp \left( \left( \frac{N}{2} - k \right) H \left( \frac{N/3-r}{N/2-k} \right) 
+ \left( \frac{2N}{3}  -  \frac{\ell + r}{2}  \right) H \left( \frac{N/2 - k}{\frac{2N}{3} -  \frac{\ell + r}{2}} \right)  \right).
\end{multline}
Note that in the edge case when $k = N/2$, 
the expression $0 \times H(0/0)$ in~\eqref{eq:counting_lemma_bound} should be interpreted as $0$, and hence the overall bound in this case is $4N^{4}$.

Furthermore, we have that 
\begin{equation}\label{eq:VNtyp_size}
\left| V_{N}^{\mathtt{typ}} \right| = 2^{N-o(N)}. 
\end{equation}
\end{lemma}
\begin{proof} 
Observe that the conditions
$1 \leq k \leq N/2$, 
$r \leq N/3$, 
$k-r \leq N/6$, 
and 
$\ell < N/3 + (r + 2(k-r))$ 
imply that 
\begin{equation}\label{eq:k_r_ineqs}
0 \leq N/3 - r 
\leq N/2 - k 
< 2N/3 - (\ell+r)/2.
\end{equation}
Thus the expression in~\eqref{eq:counting_lemma_bound} makes sense (i.e., the denominators of the fractions are positive and the arguments of $H$ are in $[0,1]$) whenever~$k < N/2$. 

We first deal with the edge case that $k = N/2$. Let $\eta = ( \eta_{1}, \ldots, \eta_{N/2} )$. 
Recall from \eqref{eq:VNtyp} that if $\sigma \in V_{N}^{\mathtt{typ}}$, 
then $|\sigma|_{0} = N/2 + 1$. 
If $\sigma = (\sigma_{1}, \ldots, \sigma_{N/2+1})$ satisfies $\eta \preccurlyeq \sigma$, 
then we must have $\sigma_{i} = \eta_{i}$ for all $i \in [N/2-1]$. 
Thus only the last two coordinates of $\sigma$ are undetermined, both of which take values in $[N]$. 
So
$\left| \left\{ \sigma \in V_{N}^{\mathtt{typ}} : \eta \preccurlyeq \sigma \right\} \right| 
\leq N^{2}$, which proves~\eqref{eq:counting_lemma_bound} in this special case. 

Furthermore, if $k = N/2$, then by~\eqref{eq:k_r_ineqs} 
we have that $r = N/3$, 
and so 
$N/3 + (r+2(k-r)) = N$. Since every $\sigma \in V_{N}^{\mathtt{typ}}$ has at least one more coordinate than $\eta$, if $\eta \preccurlyeq \sigma$, then $N = |\sigma|_{1} > |\eta|_{1} = \ell$. 
Thus if $k= N/2$ and $\ell \geq N/3 + (r+2(k-r)) = N$, then 
$\left| \left\{ \sigma \in V_{N}^{\mathtt{typ}} : \eta \preccurlyeq \sigma \right\} \right| = 0$. 

For the remainder of the proof we can assume that $k < N/2$. The arguments that follow will make use of two inequalities which we now state. 
The first inequality 
bounds the number of sequences in $\{1,2\}^{n}$ with exactly $k$ odd coordinates: 
\begin{align}\label{eq:odd_even_count}
    \left| \left\{ x=(x_1,\ldots,x_n) \in \{1,2\}^n : \left| \left\{ i \in [n] : x_i = 1 \right\} \right|=k \right\} \right| = \binom{n}{k} \leq \exp \left( n H\left( \frac{k}{n} \right) \right),
\end{align}
where the inequality directly follows from \eqref{eq:binom_coeff_entropy_bounds}. 
The second inequality controls the number of sequences with all even coordinates and whose $\ell_{1}$ norm is bounded by some integer $k$:
\begin{align}
    \notag \left| \left\{ x \in (2\N_0)^{n} : |x|_1 \leq k \right\} \right| 
    &= \left| \left\{ x \in (\N_0)^{n} : |x|_1 \leq k/2 \right\} \right|
    \leq k \left| \left\{ x \in (\N_0)^{n} : |x|_1 = \lfloor k/2 \rfloor \right\} \right| \\
    &= k \binom{n + \lfloor k/2 \rfloor -1}{ \lfloor k/2 \rfloor} 
    < k \binom{n + \lfloor k/2 \rfloor }{ \lfloor k/2 \rfloor}. \label{eq:useful_ineq_two}
\end{align}
Here, the first equality follows by the bijection that divides each coordinate by $2$, 
and the first inequality follows from the fact that 
$\left| \left\{ x \in (\N_0)^{n} : |x|_1 = m \right\} \right|$
is nondecreasing in $m$. 
The next equality follows by a `stars and bars' argument for counting combinations, and in the final step we remove the $-1$ for notational convenience.

Recall that for $\sigma \in V_{N}^{\mathtt{typ}}$ we have that $|\sigma|_{0} = \frac{N}{2} + 1$ and $|\sigma|_{1} = N$. 
By the analysis in the proof of Lemma~\ref{lem:partial_order_expansion} (and also recalling the notation $\oplus$ and $\oplus'$ introduced in \eqref{eq:oplus} and \eqref{eq:oplusprime}), 
in order for $\eta \preccurlyeq \sigma$ and the constraints in the statement of Lemma~\ref{lem:descendant_count_in_VN_discrete} to be satisfied, there are two possibilities for the structure of~$\sigma$: 
\begin{itemize}
\item[] \textbf{(Case 1)} either $\sigma = \eta \oplus \mu$ for some $\mu \in V_{\infty}$ with 
$\left| \mu \right|_{0} = N/2 - k + 1$ 
and $\left| \mu \right|_{1} = N - \ell$;
\item[] \textbf{(Case 2)} or $\sigma = \eta \oplus' \mu$ for some $\mu \in V_{\infty}$ with 
$\left| \mu \right|_{0} = N/2 - k + 2$ 
and $\left| \mu \right|_{1} = N - \ell$.
\end{itemize}
In both cases, in order for $\sigma \in V_{N}^{\mathtt{typ}}$ to hold, some additional parity constraints have to be satisfied.
\textbf{Case 1.} 
In the first case, the parity constraints imply that  
$\mu = (\mu_{1}, \mu_{2}, \ldots, \mu_{N/2 - k + 1})$ 
has to satisfy that 
$|\{ i \in [N/2 - k] : \mu_{i} \text{ is odd} \}| = N/3 - r$ 
(and thus also 
$|\{ i \in [N/2 - k] : \mu_{i} \text{ is even} \}| = N/6 - (k-r)$). 
So, in particular, there exists a unique decomposition of $\mu$ into 
\[
\mu=\mu^{(1)} + \mu^{(2)} + \mu^{(3)},
\]
where the sum is taken coordinate-wise and the three sequences $\mu^{(1)}$, $\mu^{(2)}, \mu^{(3)} \in (\N_{0})^{N/2 -k +1}$ satisfy the following constraints.

\begin{itemize}
\item First, the sequence $\mu^{(1)}$ takes care of the parity constraints, 
satisfying 
\begin{itemize}
    \item that 
$\mu^{(1)}_{i} \in \{1,2\}$ for all $i \in [N/2-k]$, 
    \item that 
$|\{ i \in [N/2 - k] : \mu_{i}^{(1)} \text{ is odd} \}| = N/3 - r$, 
    \item and that 
$\mu^{(1)}_{N/2-k+1} = 0$.
\end{itemize}

\item Second, the sequence $\mu^{(2)}$ (essentially) takes care of the $\ell_{1}$ norm constraint, 
satisfying 
\begin{itemize}
    \item that 
$\mu^{(2)} \in (2 \N_0)^{N/2-k+1}$, 
    \item that 
$\mu^{(2)}_{N/2-k+1} = 0$, 
    \item and that 
\begin{align*}
|\mu^{(2)}|_{1} 
&\leq N - 1 - \ell - \left[ (N/3 - r) + 2  (N/6 - (k-r)) \right] \\
&= N/3 - 1 - \ell + (r + 2(k-r)),
\end{align*}
where we used that 
$|\sigma|_{1} = N$, 
that $|\eta|_{1} = \ell$, 
that $\mu_{N/2 - k + 1} \geq 1$, 
and that $|\mu^{(1)}|_{1} = (N/3 - r) + 2(N/6 - (k-r))$.
\end{itemize}

\item Finally, the sequence $\mu^{(3)}$ takes care of the last coordinate, 
satisfying that 
$\mu^{(3)}_i = 0$ for all $i \in [N/2 - k]$, 
and that 
$\mu^{(3)}_{N/2 - k + 1} \in [N]$. 
\end{itemize}
Let $A$, $B$, and $C$ be the corresponding spaces of the sequences $\mu^{(1)}$, $\mu^{(2)}$, and $\mu^{(3)}$, respectively. 
Note that $B$ is nonempty only if 
$\ell < N/3 + r + 2(k-r)$, 
so in the remainder of this case we assume that 
$\ell < N/3 + r + 2(k-r)$. 
With this decomposition in hand, we then have that 
\begin{equation*}
    \left| \left\{\mu : \eta \oplus \mu \in V_{N}^{\mathtt{typ}} \right\} \right| \leq |A|\cdot |B| \cdot |C| .  
\end{equation*}
Thus, we need to bound the size of the sets $A$, $B$, and $C$ from above. 
For $A$, we can directly use~\eqref{eq:odd_even_count} and obtain that 
\[
|A| 
= \binom{N/2 - k}{N/3 - r}
\leq \exp \left( \left( \frac{N}{2} - k \right) H \left( \frac{N/3-r}{N/2-k} \right) \right).
\]
For $B$, first note that $\ell$ and $r$ have the same parity, and so $-\ell + r + 2(k-r)$ is even. 
Using~\eqref{eq:useful_ineq_two} we thus obtain that
\begin{align*}
    |B| &\leq \left| \left\{ x \in (2\N_{0})^{N/2 - k} : |x|_1 \leq N/3 - 1 - \ell + (r + 2(k-r)) \right\} \right|
    \\
    &
    =
    \left| \left\{ x \in (\N_{0})^{N/2 - k} : |x|_1  \leq N/6+\frac{- \ell + (r + 2(k-r))}{2} - 1 \right\} \right|
    \\
    &
    \leq N \binom{2N/3 - (\ell + r)/2}{N/2 - k}
    \leq 
    N
    \exp \left( \left( \frac{2N}{3} - \frac{\ell + r}{2} \right) H \left( \frac{N/2 - k}{\frac{2N}{3} - \frac{\ell + r}{2}} \right) \right).
\end{align*} 
Finally, we have that $|C| \leq N$. 
Overall, we thus obtain the bound 
\begin{multline*}
\left| \left\{\mu :  \eta \oplus \mu \in V_{N}^{\mathtt{typ}} \right\} \right| 
\leq |A|\cdot |B| \cdot |C|  \\
\leq N^{2} \exp \left( \left( \frac{N}{2} - k \right) H \left( \frac{N/3-r}{N/2-k} \right) 
+ \left( \frac{2N}{3}  -  \frac{\ell + r}{2}  \right) H \left( \frac{N/2 - k}{\frac{2N}{3} -  \frac{\ell + r}{2}} \right)  \right).
\end{multline*}

\textbf{Case 2.} 
In the second case, 
the parity constraints imply that
$\mu = (\mu_{0}, \mu_{1}, \mu_{2}, \ldots, \mu_{N/2 - k + 1})$ 
has to satisfy that 
$|\{ i \in [N/2 - k] : \mu_{i} \text{ is odd} \}| \in \{N/3 - r - 1, N/3 - r, N/3 - r + 1 \} \cap \N_{0}$. 
(Note that here we start indexing the coordinates of $\mu$ from $0$ for notational convenience.) 
In particular, 
there exists a unique decomposition of $\mu$ into 
\[
\mu = \mu^{(1)} + \mu^{(2)} + \mu^{(3)},
\]
where the sequences $\mu^{(1)}$, $\mu^{(2)}$, and $\mu^{(3)} \in (\N_0)^{N/2-k+2}$  satisfy the following constraints.
\begin{itemize}
\item First, the sequence $\mu^{(1)}$ takes care of the parity constraints, satisfying 
\begin{itemize}
\item that $\mu^{(1)}_{i} \in \{1,2\}$ for all $i \in [N/2-k]$, 
\item that 
$|\{ i \in [N/2 - k] : \mu^{(1)}_{i} \text{ is odd} \}| 
\in \{ N/3 - r - 1, N/3 - r, N/3 - r + 1 \} \cap \N_{0}$,
\item and that 
$\mu^{(1)}_{0} = \mu^{(1)}_{N/2-k+1} = 0$.
\end{itemize}

\item Second, the sequence $\mu^{(2)}$ (essentially) takes care of the $\ell_{1}$ norm constraint, satisfying 
\begin{itemize}
    \item that 
$\mu^{(2)} \in (2 \N_{0})^{N/2-k+2}$, 
    \item that 
$\mu^{(2)}_{0} = \mu^{(2)}_{N/2-k+1} = 0$, 
    \item and that 
\begin{align*}
|\mu^{(2)}|_{1} 
&\leq 
N - 2 - \ell - [(N/3 - r) + 2(N/6 - (k-r)) - 1] \\
&= N/3 - 1 - \ell + (r + 2(k-r)),
\end{align*}
where we used that 
$|\sigma|_{1} = N$, 
that $|\eta|_{1} = \ell$, 
that $\mu_{0}, \mu_{N/2 - k + 1} \geq 1$, 
and that $|\mu^{(1)}|_{1} \geq (N/3 - r) + 2(N/6 - (k-r)) - 1$. 
\end{itemize}

\item Finally, the sequence $\mu^{(3)}$ takes care of the first and last coordinates, satisfying that 
$\mu^{(3)}_i = 0$ for all $i \in [N/2 - k]$, 
and that 
$\mu^{(3)}_{0}, \mu^{(3)}_{N/2 - k + 1} \in [N]$. 
\end{itemize}
Let $A$, $B$, and $C$ be the corresponding spaces of $\mu^{(1)}$, $\mu^{(2)}$, and $\mu^{(3)}$, respectively. 
Note that $B$ is nonempty only if 
$\ell < N/3 + r + 2(k-r)$, 
so in the remainder of this case we assume that 
$\ell < N/3 + r + 2(k-r)$. 
Then 
\begin{equation*}
    \left| \left\{\mu :  \eta \oplus' \mu \in V_{N}^{\mathtt{typ}} \right\} \right| \leq |A|\cdot |B| \cdot |C| . 
\end{equation*}
Thus, we need to bound the size of the sets $A$, $B$, and $C$ from above. 
For $A$, 
using the fact that 
$1/n \leq \binom{n}{k} / \binom{n}{k+1} \leq n$ 
for any $0 \leq k \leq n$, 
together with~\eqref{eq:odd_even_count}, we obtain that 
\begin{align*}
    |A| 
    &= \binom{N/2-k}{N/3-r-1} + \binom{N/2-k}{N/3-r} + \binom{N/2-k}{N/3-r+1} \\
    &\leq 3N \binom{N/2-k}{N/3-r} 
    \leq 3N \exp \left( \left( \frac{N}{2} - k \right) H \left( \frac{N/3-r}{N/2-k} \right) \right).
\end{align*}
The calculation above assumed $r < N/3$ 
and the same bound holds for $r = N/3$. 
For $B$, 
note that both the number of potentially nonzero coordinates and the $\ell_{1}$ constraint are the same as for the set $B$ from Case 1, so the same bound applies: 
\[
|B| 
\leq 
N
\exp \left( \left( \frac{2N}{3} - \frac{\ell + r}{2} \right) H \left( \frac{N/2 - k}{\frac{2N}{3} - \frac{\ell + r}{2}} \right) \right).
\] 
Finally, we have that $|C| \leq N^{2}$. 
Overall, we thus obtain the bound 
\begin{multline*}
\left| \left\{\mu :  \eta \oplus' \mu \in V_{N}^{\mathtt{typ}} \right\} \right| 
\leq |A|\cdot |B| \cdot |C|  \\ 
\leq 3 N^{4} \exp \left( \left( \frac{N}{2} - k \right) H \left( \frac{N/3-r}{N/2-k} \right) 
+ \left( \frac{2N}{3}  -  \frac{\ell + r}{2}  \right) H \left( \frac{N/2 - k}{\frac{2N}{3} -  \frac{\ell + r}{2}} \right)  \right).
\end{multline*}
Combining this with the bound on 
$\left| \left\{\mu :  \eta \oplus \mu \in V_{N}^{\mathtt{typ}} \right\} \right|$ 
from above, 
we have that~\eqref{eq:counting_lemma_bound} holds.


Finally, we turn to proving~\eqref{eq:VNtyp_size}.  The upper bound $\left| V_{N}^{\mathtt{typ}} \right| \leq \left| V_{N} \right| = 2^{N-1}$ is immediate, thus it suffices to show $\left| V_{N}^{\mathtt{typ}} \right| \geq 2^{N - o(N)}$. Note that 
$(1) \preccurlyeq \sigma$ for every $\sigma \in V_{\infty} \setminus \{\varnothing\}$, 
so 
$\left| V_{N}^{\mathtt{typ}} \right| 
= \left| \left\{ \sigma \in V_{N}^{\mathtt{typ}} : (1) \preccurlyeq \sigma \right\} \right|$. 
Taking $\eta = (1)$, we have that $k= \ell = r =1$, and, restricting attention to Case 1, we have that
\begin{align*}
    \left| V_{N}^{\mathtt{typ}} \right| & \geq \left|\{\mu: (1)  \oplus \mu \in V_N^{\mathtt{typ}}\} \right|.
\end{align*}
To obtain the desired lower bound, we construct a subfamily of valid sequences $\mu$ by following the unique decomposition $\mu = \mu^{(1)}+\mu^{(2)} + \mu^{(3)}$ from Case 1, with tighter constraints. Specifically, we keep $\mu^{(1)}$ as in Case 1, we fix $|\mu^{(2)}|_1 = N/3 -2$ and fix $\mu^{(3)}$ to be the sequence with $\mu_{N/2}^{(3)} = 2$ and zero for all other coordinates. As argued above, for such $\mu$ we have that $(1) \oplus \mu \in  V_N^{\mathtt{typ}}$.  As before, let $A$, $B$, and $C$ be the corresponding spaces of $\mu^{(1)}$, $\mu^{(2)}$, and $\mu^{(3)}$, respectively. 
Then $|A| = \binom{N/2 -1}{N/3 -1}$ as in Case 1, 
$|B| = |\{ x \in (\N_0)^{N/2-1}: |x|_1 = \frac{1}{2}(N/3 -2)\}| = \binom{(N/2-1) +  \frac{1}{2}(N/3 -2) - 1}{ \frac{1}{2}(N/3 -2)}$, 
and $|C|=1$. 
Using the lower bound of~\eqref{eq:binom_coeff_entropy_bounds} and the identity 
$\frac{1}{2} H \left( \frac{2}{3} \right) + \frac{2}{3} H \left( \frac{1}{4} \right) = \log 2$, 
it follows that $\left| V_{N}^{\mathtt{typ}} \right| \geq |A| \cdot |B| \geq 2^{N - o(N)}$; we omit the details, as they are similar to the above.
\end{proof}

Define the following truncated versions of the large deviation rate functions: 
    \be\ba\label{eq:barfunc}
        & \overline{\Lambda}_{\exp}(x) := \begin{cases} \Lambda_{\exp}^{\star} (x) & \mbox{if } x \leq 1,\\
        0 & \mbox{otherwise},
        \end{cases}\\
        & \overline{\Lambda}_{\min}(x) := \begin{cases} \Lambda_{\min}^{\star} (x) & \mbox{if } x \leq 3/4,\\
        0 & \mbox{otherwise},
        \end{cases}\\
        & \overline{\Lambda}_{\max}(x) := \begin{cases} \Lambda_{\max}^{\star} (x) & \mbox{if }x \leq 9/4,\\
        0 & \mbox{otherwise.}
        \end{cases}\\
    \ea\ee
    We also allow $x=+\infty$ as an argument for all three of those functions.
We are now ready to prove the second moment bound.
\begin{proof}[Proof of~\eqref{eq:ENH_highprob} in Theorem~\ref{thm:ENH_matching_cts}] 
Recall the choices of constants 
$\alpha$, $\beta$, $\gamma$, and $z$ from~\eqref{eq:constants}, 
and set $N := 6 \lfloor \tfrac{tz}{6} \rfloor$. 
By the arguments in Section~\ref{sec:switching_subtrees} and~\ref{sec:switching_subtrees_expectation}, we have the lower bound 
\begin{equation}\label{eq:VNtyp_LB}
\LCS \left( A_{1}(t), A_{2}(t) \right) 
\geq 
\sum_{\sigma \in V_{N}^{\mathtt{typ}}} A_{1}^{s}(\sigma, t) A_{2}^{s}(\sigma, t).
\end{equation}
Recall the definitions of 
$T_{1}^{\min}(\sigma)$, $T_{1}^{\max}(\sigma)$, and $T_{1}^{\exp}(\sigma)$ from~\eqref{eq:Tmin},~\eqref{eq:Tmax} and~\eqref{eq:Texp}
in Section~\ref{sec:switching_subtrees}. 
We say that a vertex 
$\sigma \in V_{N}^{\mathtt{typ}}$ is \emph{$1$-good} 
if 
$T_{1}^{\exp}(\sigma) \leq \frac{\alpha N}{3}$, 
$T_{1}^{\min}(\sigma) \leq \frac{\beta N}{3}$, 
and $T_{1}^{\max}(\sigma) \leq \frac{\gamma N}{6}$. 
In particular, as observed in Section~\ref{sec:switching_subtrees_expectation}, 
if $\sigma$ is $1$-good, then $\sigma$ arrives before time $t$ in the first switched process $A_{1}^{s}$; 
in other words, if $\sigma$ is $1$-good, then $A_{1}^{s}(\sigma, t) = 1$. 
Analogously, we say that 
$\sigma \in V_{N}^{\mathtt{typ}}$ is \emph{$2$-good} 
if 
$T_{2}^{\exp}(\sigma) \leq \frac{\alpha N}{3}$, 
$T_{2}^{\min}(\sigma) \leq \frac{\beta N}{3}$, 
and $T_{2}^{\max}(\sigma) \leq \frac{\gamma N}{6}$. 
Furthermore, we call a vertex $\sigma$ \emph{good} 
if it is both $1$-good and $2$-good. 
Since $\sigma$ being good implies that 
$A_{1}^{s}(\sigma, t) A_{2}^{s}(\sigma, t) = 1$, 
the lower bound in~\eqref{eq:VNtyp_LB} further implies the following lower bound:
\begin{equation}\label{eq:VNtyp_good_LB}
\LCS \left( A_{1}(t), A_{2}(t) \right) 
\geq 
\sum_{\sigma \in V_{N}^{\mathtt{typ}}} 
\mathbbm{1}_{\left\{ \sigma \text{ is good} \right\}}.
\end{equation}

This is the lower bound that we will analyze. The first moment of this lower bound was analyzed in Section~\ref{sec:switching_subtrees_expectation}, 
and here we will analyze its second moment. 
Recall from Section~\ref{sec:switching_subtrees_expectation} that 
$\p \left( \sigma \text{ is 1-good} \right)$ is the same for all $\sigma \in V_{N}^{\mathtt{typ}}$, 
and, from~\eqref{eq:largedev_LB_abc}, in particular it satisfies 
    \begin{equation}\label{eq:1goodprob}
        \p \left( \sigma \text{ is 1-good} \right)= e^{-\phi_N + o(N)},
    \end{equation}
with 
    \begin{equation*}
        \phi_N := \frac{N}{3} \Lambda_{\exp}^{\star} (\alpha) +\frac{N}{3} \Lambda_{\min}^{\star} (\beta)+\frac{N}{6} \Lambda_{\max}^{\star} (\gamma),
    \end{equation*}
and where we recall that $\Lambda^\star_{\mathrm{exp}}$ denotes the large deviation rate function of an exponential random variable with mean $1$, and recall also $\Lambda_{\min}^{\star}$ and $\Lambda_{\max}^{\star}$ from Lemma~\ref{lemma:ratefunc}. 
Furthermore, due to the independence of the two processes, we have that 
$\p \left( \sigma \text{ is good} \right) 
= \p \left( \sigma \text{ is 1-good} \right)^{2}$. 

We now turn to the second moment of the right hand side of~\eqref{eq:VNtyp_good_LB}. 
By expanding the square of the sum, 
using linearity of expectation, 
the independence of $A_{1}^{s}$ and $A_{2}^{s}$ (which, in particular, implies that $\mathbb{P} \left( \sigma, \sigma^\prime \text{ are good}  \right)
= \mathbb{P} \left( \sigma, \sigma^\prime \text{ are 1-good}  \right)^{2}$), 
we see that our goal is to show the following inequality (which is the analogue of~\eqref{eq:correlation}):
    \begin{equation}\label{eq:2nd_moment_beyondUH}
        \sum_{\sigma \in V_N^{\mathtt{typ}}} \sum_{\sigma^\prime \in V_N^{\mathtt{typ}}} \mathbb{P} \left( \sigma, \sigma^\prime \text{ are 1-good}  \right)^{2}
        \leq 
        \left( \left| V_N^{\mathtt{typ}} \right| \mathbb{P} \left( \sigma \text{ is 1-good}  \right)^{2} \right)^2 \e^{o(N)},
    \end{equation}
where $\sigma \in V_N^{\mathtt{typ}}$ on the right hand side of~\eqref{eq:2nd_moment_beyondUH}. 
Once~\eqref{eq:2nd_moment_beyondUH} is proven, the proof of~\eqref{eq:ENH_highprob} follows from a second moment argument, together with the use of the exponential growth property of a Yule process. Since this argument is identical to the one in Section~\ref{sec:UHLB_whp}, we omit the details. 
The rest of the proof is devoted to proving that~\eqref{eq:2nd_moment_beyondUH} holds.

In order to prove~\eqref{eq:2nd_moment_beyondUH}, the key is to understand 
$\mathbb{P} \left( \sigma, \sigma^\prime \text{ are 1-good}  \right)$. 
Let $T_{\sigma}^{s}$ denote the arrival time of $\sigma$ in $A_{1}^{s}$ 
and recall that $T_{\sigma}^{s} = T_{1}^{\min}(\sigma) + T_{1}^{\max}(\sigma) + T_{1}^{\exp}(\sigma)$; 
similarly, let 
$T_{\sigma'}^{s}$ denote the arrival time of $\sigma'$ in~$A_{1}^{s}$. 
The arrival times $T_{\sigma}^{s}$ and $T_{\sigma'}^{s}$ are correlated; 
to control this correlation, we aim to separate out independent parts. In the following, we will typically assume that $\sigma \neq \sigma^\prime$, although all results formally also hold for $\sigma = \sigma^\prime$.

As in Section~\ref{sec:UHLB_whp}, we will use the youngest common predecessor 
$\eta := \ycp(\sigma, \sigma')$ 
of $\sigma$ and $\sigma'$ 
in order to find appropriate independence, 
which facilitates the analysis. 
First note that 
$T_{\sigma}^{s} = T_{\eta}^{s} + (T_{\sigma}^{s} - T_{\eta}^{s})$ 
and 
$T_{\sigma'}^{s} = T_{\eta}^{s} + (T_{\sigma'}^{s} - T_{\eta}^{s})$. 
In the analysis of the original process in Section~\ref{sec:UHLB_whp}, 
the triple of random variables 
$\{ T_{\eta}, T_{\sigma} - T_{\eta}, T_{\sigma'} - T_{\eta} \}$
were mutually independent. 
However, this is no longer true in the switched process $A_{1}^{s}$. 
To see this through a simple example, consider $\sigma = (1,1)$ and $\sigma' = (2,1)$, in which case we have that $\eta = (1)$. In the switched process $A_{1}^{s}$, we have that 
$T_{\sigma}^{s} \leq T_{\sigma'}^{s}$ 
\emph{by construction} 
(indeed, this is the key point of the switched process), 
which implies also that 
$T_{\sigma}^{s} - T_{\eta}^{s} \leq T_{\sigma'}^{s} - T_{\eta}^{s}$, 
and thus these random variables are not independent.

Nonetheless, the triple of random variables 
$\{ T_{\eta}^{s}, T_{\sigma}^{s} - T_{\eta}^{s}, T_{\sigma'}^{s} - T_{\eta}^{s} \}$
are still \emph{approximately} independent in the switched process $A_{1}^{s}$, in an appropriate sense, when $|\sigma|_{1} = |\sigma'|_{1}$ is large. 
More precisely, we will slightly modify the definitions of these random variables to find independence, while not losing much quantitatively. 
To this end, 
write $\eta = \left( \eta_{1}, \ldots, \eta_{k} \right)$, 
where $k \geq 1$, 
and assume, without loss of generality, that 
$\sigma_{k} = \eta_{k}$ and $\sigma'_{k} > \eta_{k}$. 
Define $\wt{\eta}$ to be the unique vertex out of 
$\left( \eta_{1}, \ldots, \eta_{k-1}, \eta_{k} + 1, 1 \right)$ 
and 
$\left( \eta_{1}, \ldots, \eta_{k-1}, \eta_{k} + 2 \right)$ 
that satisfies 
$\wt{\eta} \preccurlyeq \sigma'$.  
In the special case where $\sigma' = \left( \eta_{1}, \ldots, \eta_{k-1}, \eta_{k} + 1 \right)$, 
set $\wt{\eta} := \sigma'$. 
The key observation is that $T_{\sigma}^{s}$ and $T_{\sigma'}^{s} - T_{\wt{\eta}}^{s}$ are independent in the switched process~$A_{1}^{s}$. 
The key point behind this is that, while switching creates dependencies among the passage times, these dependencies are all \emph{local} (see Figure~\ref{fig:switching_subtrees} for an illustration), so the minor modification of $\eta$ to $\wt{\eta}$ circumvents the dependency issue.

The random variable $T_{\sigma'}^{s} - T_{\wt{\eta}}^{s}$ is a sum of three different kinds of random variables (of the form $U$, $\min\{ U, Y + Z \}$, and $\max\{ U, Y + Z \}$, where $U, Y, Z$ are i.i.d.\ exponential with mean $1$), and we aim to separate these out. To this end, we distinguish three cases. 

First, suppose that 
$\sigma_{k} = \eta_{k} < \eta_{k} + 1 = \sigma'_{k}$ 
and 
$\wt{\eta} = (\eta_{1}, \ldots, \eta_{k-1}, \eta_{k} + 1, 1)$. 
Define 
$T_{1}^{\min} \left( \wt{\eta}, \sigma' \right) := T_{1}^{\min} \left( \sigma' \right) - T_{1}^{\min} \left( \wt{\eta} \right)$, 
$T_{1}^{\max} \left( \wt{\eta}, \sigma' \right) := T_{1}^{\max} \left( \sigma' \right) - T_{1}^{\max} \left( \wt{\eta} \right)$, 
and 
$T_{1}^{\exp} \left( \wt{\eta}, \sigma' \right) := T_{1}^{\exp} \left( \sigma' \right) - T_{1}^{\exp} \left( \wt{\eta} \right)$. 
In this case we have that 
\begin{align*}
    T_{1}^{\min}(\wt{\eta}, \sigma') 
    &=  
    \sum_{\substack{i=k+1: \\ \sigma'_{i} \text{ odd}}}^{N/2} \min \left\{ W_{(\sigma'_1,\ldots,\sigma'_i,1)} , W_{(\sigma'_1,\ldots,\sigma'_i + 1 )} + W_{(\sigma'_1,\ldots,\sigma'_i + 1, 1)} \right\}, \\
    T_{1}^{\max}(\wt{\eta}, \sigma')
    &=
    \sum_{\substack{i=k+1: \\ \sigma'_{i} \text{ even}}}^{N/2} \max \left\{ W_{(\sigma'_1,\ldots,\sigma'_{i}-1,1)} , W_{(\sigma'_1,\ldots,\sigma'_i  )} + W_{(\sigma'_1,\ldots,\sigma'_i, 1)} \right\}, \\
    T_{1}^{\exp}(\wt{\eta}, \sigma')
    &=
    \sum_{\substack{i=k+1: \\ \sigma'_{i} \text{ odd}}}^{N/2} \sum_{j=2}^{\sigma'_i} W_{(\sigma'_1,\ldots,\sigma'_{i-1},j)}
    +
    \sum_{\substack{i=k+1: \\ \sigma'_{i} \text{ even}}}^{N/2} \sum_{j=2}^{\sigma'_i-1} W_{(\sigma'_1,\ldots,\sigma'_{i-1},j)}
    +
    \sum_{j=2}^{\sigma'_{N/2+1}} W_{(\sigma'_1,\ldots,\sigma'_{N/2},j)}, 
\end{align*}
and note that 
$T_{\sigma'}^{s} - T_{\wt{\eta}}^{s} 
= T_{1}^{\min} \left( \wt{\eta}, \sigma' \right) + T_{1}^{\max} \left( \wt{\eta}, \sigma' \right) + T_{1}^{\exp} \left( \wt{\eta}, \sigma' \right)$, 
with 
$T_{1}^{\min} \left( \wt{\eta}, \sigma' \right)$, 
$T_{1}^{\max} \left( \wt{\eta}, \sigma' \right)$, 
and 
$T_{1}^{\exp} \left( \wt{\eta}, \sigma' \right)$ 
being mutually independent of each other 
and also of 
$T_{1}^{\min} \left( \sigma \right)$, 
$T_{1}^{\max} \left( \sigma \right)$, 
and 
$T_{1}^{\exp} \left( \sigma \right)$.  

Next, suppose 
either that 
$\sigma_{k} = \eta_{k} < \eta_{k} + 2 < \sigma'_{k}$, 
or that 
$\sigma_{k} = \eta_{k} < \eta_{k} + 2 = \sigma'_{k}$ 
and $\sigma'_{k}$ is odd. 
In both cases we have that  
$\wt{\eta} = (\eta_{1}, \ldots, \eta_{k-1}, \eta_{k} + 2)$. 
Define again 
$T_{1}^{\min} \left( \wt{\eta}, \sigma' \right) := T_{1}^{\min} \left( \sigma' \right) - T_{1}^{\min} \left( \wt{\eta} \right)$, 
$T_{1}^{\max} \left( \wt{\eta}, \sigma' \right) := T_{1}^{\max} \left( \sigma' \right) - T_{1}^{\max} \left( \wt{\eta} \right)$, 
and 
$T_{1}^{\exp} \left( \wt{\eta}, \sigma' \right) := T_{1}^{\exp} \left( \sigma' \right) - T_{1}^{\exp} \left( \wt{\eta} \right)$. 
In this case we have 
\begin{align*}
    T_{1}^{\min}(\wt{\eta}, \sigma') 
    &= 
    \sum_{\substack{i=k: \\ \sigma'_{i} \text{ odd}}}^{N/2} \min \left\{ W_{(\sigma'_1,\ldots,\sigma'_i,1)} , W_{(\sigma'_1,\ldots,\sigma'_i + 1 )} + W_{(\sigma'_1,\ldots,\sigma'_i + 1, 1)} \right\}, \\
    T_{1}^{\max}(\wt{\eta}, \sigma')
    &=
    \sum_{\substack{i=k: \\ \sigma'_{i} \text{ even}}}^{N/2} \max \left\{ W_{(\sigma'_1,\ldots,\sigma'_{i}-1,1)} , W_{(\sigma'_1,\ldots,\sigma'_i  )} + W_{(\sigma'_1,\ldots,\sigma'_i, 1)} \right\}, \\
    T_{1}^{\exp}(\wt{\eta}, \sigma')
    &=
    \sum_{j = \eta_{k}+3}^{\sigma'_{k}-1} W_{(\sigma'_{1}, \ldots, \sigma'_{k-1},j)} 
    + W_{(\sigma'_{1}, \ldots, \sigma'_{k-1},\sigma'_{k})} \mathbbm{1}_{\{ \eta_{k} + 2 < \sigma'_{k} \text{ and } \sigma'_{k} \text{ is odd} \}} \\
    &\quad + \sum_{\substack{i=k+1: \\ \sigma'_{i} \text{ odd}}}^{N/2} \sum_{j=2}^{\sigma'_i} W_{(\sigma'_1,\ldots,\sigma'_{i-1},j)}
    +
    \sum_{\substack{i=k+1: \\ \sigma'_{i} \text{ even}}}^{N/2} \sum_{j=2}^{\sigma'_i-1} W_{(\sigma'_1,\ldots,\sigma'_{i-1},j)} 
    +
    \sum_{j=2}^{\sigma'_{N/2+1}} W_{(\sigma'_1,\ldots,\sigma'_{N/2},j)}.
\end{align*}
Again we have 
$T_{\sigma'}^{s} - T_{\wt{\eta}}^{s} 
= T_{1}^{\min} \left( \wt{\eta}, \sigma' \right) + T_{1}^{\max} \left( \wt{\eta}, \sigma' \right) + T_{1}^{\exp} \left( \wt{\eta}, \sigma' \right)$, 
with 
$T_{1}^{\min} \left( \wt{\eta}, \sigma' \right)$, 
$T_{1}^{\max} \left( \wt{\eta}, \sigma' \right)$, 
and 
$T_{1}^{\exp} \left( \wt{\eta}, \sigma' \right)$ 
being mutually independent of each other 
and also of 
$T_{1}^{\min} \left( \sigma \right)$, 
$T_{1}^{\max} \left( \sigma \right)$, 
and 
$T_{1}^{\exp} \left( \sigma \right)$. 

Finally, 
suppose that 
$\sigma_{k} = \eta_{k} < \eta_{k} +2 = \sigma'_{k}$ 
and 
$\sigma'_{k}$ is even, 
and so 
$\wt{\eta} = \left(\eta_{1}, \ldots, \eta_{k-1}, \eta_{k} + 2 \right)$. 
In this case we have that 
\begin{align*}
T_{\sigma'}^{s} - T_{\wt{\eta}}^{s}
&= \sum_{\substack{i=k: \\ \sigma'_{i} \text{ odd}}}^{N/2} \min \left\{ W_{(\sigma'_1,\ldots,\sigma'_i,1)} , W_{(\sigma'_1,\ldots,\sigma'_i + 1 )} + W_{(\sigma'_1,\ldots,\sigma'_i + 1, 1)} \right\} \\
&\quad + \sum_{\substack{i=k: \\ \sigma'_{i} \text{ even}}}^{N/2} \max \left\{ W_{(\sigma'_1,\ldots,\sigma'_{i}-1,1)} , W_{(\sigma'_1,\ldots,\sigma'_i  )} + W_{(\sigma'_1,\ldots,\sigma'_i, 1)} \right\} \\
&\quad - W_{(\sigma'_{1}, \ldots, \sigma'_{k})} \\
&\quad + \sum_{\substack{i=k+1: \\ \sigma'_{i} \text{ odd}}}^{N/2} \sum_{j=2}^{\sigma'_i} W_{(\sigma'_1,\ldots,\sigma'_{i-1},j)}
    +
    \sum_{\substack{i=k+1: \\ \sigma'_{i} \text{ even}}}^{N/2} \sum_{j=2}^{\sigma'_i-1} W_{(\sigma'_1,\ldots,\sigma'_{i-1},j)} 
    +
    \sum_{j=2}^{\sigma'_{N/2+1}} W_{(\sigma'_1,\ldots,\sigma'_{N/2},j)}.
\end{align*}
The negative term $- W_{(\sigma'_{1}, \ldots, \sigma'_{k})}$ is somewhat inconvenient for the subsequent analysis, so we simply drop it and define the relevant terms as follows: 
\begin{align*}
    T_{1}^{\min}(\wt{\eta}, \sigma') 
    &:= 
    \sum_{\substack{i=k: \\ \sigma'_{i} \text{ odd}}}^{N/2} \min \left\{ W_{(\sigma'_1,\ldots,\sigma'_i,1)} , W_{(\sigma'_1,\ldots,\sigma'_i + 1 )} + W_{(\sigma'_1,\ldots,\sigma'_i + 1, 1)} \right\}, \\
    T_{1}^{\max}(\wt{\eta}, \sigma')
    &:=
    \sum_{\substack{i=k: \\ \sigma'_{i} \text{ even}}}^{N/2} \max \left\{ W_{(\sigma'_1,\ldots,\sigma'_{i}-1,1)} , W_{(\sigma'_1,\ldots,\sigma'_i  )} + W_{(\sigma'_1,\ldots,\sigma'_i, 1)} \right\}, \\
    T_{1}^{\exp}(\wt{\eta}, \sigma')
    &:=
    \sum_{\substack{i=k+1: \\ \sigma'_{i} \text{ odd}}}^{N/2} \sum_{j=2}^{\sigma'_i} W_{(\sigma'_1,\ldots,\sigma'_{i-1},j)}
    +
    \sum_{\substack{i=k+1: \\ \sigma'_{i} \text{ even}}}^{N/2} \sum_{j=2}^{\sigma'_i-1} W_{(\sigma'_1,\ldots,\sigma'_{i-1},j)} 
    +
    \sum_{j=2}^{\sigma'_{N/2+1}} W_{(\sigma'_1,\ldots,\sigma'_{N/2},j)}.
\end{align*}
Then we still have that 
$T_{1}^{\min}(\wt{\eta}, \sigma') \leq T_{1}^{\min}(\sigma')$, 
$T_{1}^{\max}(\wt{\eta}, \sigma') \leq T_{1}^{\max}(\sigma')$, 
and 
$T_{1}^{\exp}(\wt{\eta}, \sigma') \leq T_{1}^{\exp}(\sigma')$ 
(which are the relevant inequalities in the bounds to follow), 
and also 
$T_{1}^{\min} \left( \wt{\eta}, \sigma' \right)$, 
$T_{1}^{\max} \left( \wt{\eta}, \sigma' \right)$, 
and 
$T_{1}^{\exp} \left( \wt{\eta}, \sigma' \right)$ 
are mutually independent of each other 
and also of 
$T_{1}^{\min} \left( \sigma \right)$, 
$T_{1}^{\max} \left( \sigma \right)$, 
and 
$T_{1}^{\exp} \left( \sigma \right)$.

Recall that for every $\sigma \in V_N^{\mathtt{typ}}$ and every $\ell \in [N]$ there exists exactly one $\eta \in V_{\infty}$ for which $\eta \preccurlyeq \sigma$ and $\left| \eta \right|_{1} = \ell$. 
Recall also that we write $k = \left| \eta \right|_{0}$, and since $\sigma \in V_{N}^{\mathtt{typ}}$, we have that $k \in [N/2 + 1]$. 
Now let 
$r := \left| \left\{ i \in [k] : \eta_{i} \text{ is odd} \right\} \right|$, 
and hence we have 
$k-r = \left| \left\{ i \in [k] : \eta_{i} \text{ is even} \right\} \right|$. 
Recall also that, since $\sigma' \in V_{N}^{\mathtt{typ}}$, we have that 
$\left| \left\{ i \in [N/2] : \sigma'_{i} \text{ is odd} \right\} \right| = N/3$ 
and 
$\left| \left\{ i \in [N/2] : \sigma'_{i} \text{ is even} \right\} \right| = N/6$.
It is readily checked that, in all of the cases above, we have the following: 
\begin{itemize}
\item $T_{1}^{\min}(\wt{\eta}, \sigma')$ consists of at least $N/3 - r - 1$ summands; 




\item $T_{1}^{\max}(\wt{\eta}, \sigma')$ consists of at least $N/6 - (k-r) - 1$ summands; 
and 


\item $T_{1}^{\exp}(\wt{\eta}, \sigma')$ consists of at least $N/3 - (\ell - (2k-r)) - 4$ summands. 


\end{itemize}

Let us now return to the second moment estimate~\eqref{eq:2nd_moment_beyondUH} that we want to prove. 
Analogously to~\eqref{eq:crucial_ineq2}, we can first write 
\begin{equation}\label{eq:doublesum_to_eta}
\sum_{\sigma \in V_N^{\mathtt{typ}}} \sum_{\sigma^\prime \in V_N^{\mathtt{typ}}} \mathbb{P} \left( \sigma, \sigma^\prime \text{ are 1-good}  \right)^{2} 
= 
\sum_{\sigma \in V_N^{\mathtt{typ}}} \sum_{\eta: \eta \preccurlyeq \sigma} \sum_{\substack{\sigma^\prime \in V_N^{\mathtt{typ}} : \\ \ycp(\sigma, \sigma^\prime) = \eta}} \mathbb{P} \left( \sigma, \sigma^\prime \text{ are 1-good}  \right)^{2}.
\end{equation}
For $\sigma, \sigma' \in V_{N}^{\mathtt{typ}}$ with $\ycp(\sigma, \sigma') = \eta$ such that $\sigma_{k} = \eta_{k} < \sigma_{k}'$, we have (analogously to~\eqref{eq:crucial_ineq}) that 
\begin{align}
&\p \left( \sigma, \sigma' \text{ are 1-good} \right) \\ 
&= 
\p \left( T_{1}^{\exp} \left( \sigma \right) \leq \frac{\alpha N}{3}, T_{1}^{\min} \left( \sigma \right) \leq \frac{\beta N}{3}, T_{1}^{\max} \left( \sigma \right) \leq \frac{\gamma N}{6}, \right. \\
&\qquad \ \ \left.
T_{1}^{\exp} \left( \sigma' \right) \leq \frac{\alpha N}{3}, T_{1}^{\min} \left( \sigma' \right) \leq \frac{\beta N}{3}, T_{1}^{\max} \left( \sigma' \right) \leq \frac{\gamma N}{6} \right) \\ 
&\leq 
\p \left( T_{1}^{\exp} \left( \sigma \right) \leq \frac{\alpha N}{3}, T_{1}^{\min} \left( \sigma \right) \leq \frac{\beta N}{3}, T_{1}^{\max} \left( \sigma \right) \leq \frac{\gamma N}{6}, \right. \\
&\qquad \ \ \left.
T_{1}^{\exp} \left( \wt{\eta}, \sigma' \right) \leq \frac{\alpha N}{3}, T_{1}^{\min} \left( \wt{\eta}, \sigma' \right) \leq \frac{\beta N}{3}, T_{1}^{\max} \left( \wt{\eta}, \sigma' \right) \leq \frac{\gamma N}{6} \right) \\ 
&= 
\p \left( T_{1}^{\exp} \left( \sigma \right) \leq \frac{\alpha N}{3}, T_{1}^{\min} \left( \sigma \right) \leq \frac{\beta N}{3}, T_{1}^{\max} \left( \sigma \right) \leq \frac{\gamma N}{6} \right) \\
&\quad 
\times \p \left(
T_{1}^{\exp} \left( \wt{\eta}, \sigma' \right) \leq \frac{\alpha N}{3}, T_{1}^{\min} \left( \wt{\eta}, \sigma' \right) \leq \frac{\beta N}{3}, T_{1}^{\max} \left( \wt{\eta}, \sigma' \right) \leq \frac{\gamma N}{6} \right) \\ 
&= \p \left( \sigma \text{ is 1-good} \right) 
\p \left( T_{1}^{\exp} \left( \wt{\eta}, \sigma' \right) \leq \frac{\alpha N}{3} \right) 
\p \left( T_{1}^{\min} \left( \wt{\eta}, \sigma' \right) \leq \frac{\beta N}{3} \right) 
\p \left( T_{1}^{\max} \left( \wt{\eta}, \sigma' \right) \leq \frac{\gamma N}{6} \right), \quad \ \ \ \label{eq:factored_bound}
\end{align}
where the last two equalities use the mutual independence of the random variables 
$T_{1}^{\exp} \left( \sigma \right)$, 
$T_{1}^{\min} \left( \sigma \right)$, 
$T_{1}^{\max} \left( \sigma \right)$, 
$T_{1}^{\exp} \left( \wt{\eta}, \sigma' \right)$, 
$T_{1}^{\min} \left( \wt{\eta}, \sigma' \right)$, 
and 
$T_{1}^{\max} \left( \wt{\eta}, \sigma' \right)$. 
Note that the definition of $\wt{\eta}$ is not symmetric in $\sigma$ and $\sigma'$, 
so when $\sigma, \sigma' \in V_{N}^{\mathtt{typ}}$ are such that 
$\ycp(\sigma, \sigma') = \eta$ satisfies $\sigma_{k}' = \eta_{k} < \sigma_{k}$, 
then the display above holds with the roles of $\sigma$ and $\sigma'$ switched. 
Since 
$\p \left( \sigma \text{ is 1-good} \right) 
= \p \left( \sigma' \text{ is 1-good} \right)$
for all $\sigma, \sigma' \in V_{N}^{\mathtt{typ}}$, 
the first factor in~\eqref{eq:factored_bound} can always be taken to be 
$\p \left( \sigma \text{ is 1-good} \right)$. 
For the other factors in~\eqref{eq:factored_bound}, we shall see that the order of magnitude of these probabilities depends only on $\eta$ (up to lower order factors), so the switched roles of $\sigma$ and $\sigma'$ does not play a significant role.

Let us now turn to bounding the latter three factors in~\eqref{eq:factored_bound}, 
assuming that $\sigma_{k} = \eta_{k} < \sigma_{k}'$, 
and also letting 
$\ell := | \eta |_{1}$ 
and 
$r := |\{ i \in [k] : \eta_{i} \text{ is odd}\}|$. 
By the discussion above, 
$T_{1}^{\exp} \left( \wt{\eta}, \sigma' \right)$
is the sum of i.i.d.\ exponential random variables with mean $1$, 
with at least $N/3 - (\ell - (2k - r)) - 4$ summands. 
Thus, if $\{ U_{i} \}_{i \geq 1}$ are i.i.d.\ $\Exp(1)$ random variables, then 
\begin{multline*}
\p \left( T_{1}^{\exp} \left( \wt{\eta}, \sigma' \right) \leq \frac{\alpha N}{3} \right) 
\leq 
\p \left( \sum_{i=1}^{N/3 - (\ell - (2k - r)) - 4} U_{i} \leq \frac{\alpha N}{3} \right) \\
\leq 
\exp \left( - \overline{\Lambda}_{\exp} \left(  \alpha \cdot \frac{N/3}{N/3 - (\ell - (2k - r)) - 4} \right) \cdot \left( \frac{N}{3} - (\ell - (2k - r)) - 4 \right) \right),
\end{multline*}
where the second inequality follows by Cram\'er's theorem 
whenever 
$N/3 - (\ell - (2k - r)) - 4 
\geq \alpha N/3$, 
and by a trivial bound otherwise 
(recall that $\overline{\Lambda}_{\exp}(x) = \Lambda_{\exp}^{\star}(x)$ whenever $x \leq 1$, 
and that $\overline{\Lambda}_{\exp}(x) = 0$ otherwise). 
Similarly, we have that 
\begin{align*}
\p \left( T_{1}^{\min} \left( \wt{\eta}, \sigma' \right) \leq \frac{\beta N}{3} \right) 
&\leq 
\exp \left( - \overline{\Lambda}_{\min} \left( \beta \cdot \frac{N/3}{N/3 - r - 1} \right) \cdot \left( \frac{N}{3} - r - 1 \right) \right), \\
\p \left( T_{1}^{\max} \left( \wt{\eta}, \sigma' \right) \leq \frac{\gamma N}{6} \right) 
&\leq \exp \left( - \overline{\Lambda}_{\max} \left( \gamma \cdot \frac{N/6}{N/6 - (k-r) - 1} \right) \cdot \left( \frac{N}{6} - (k-r) - 1 \right) \right).
\end{align*}
In the expressions above, we can remove the additive constants of $-1$ and $-4$ by replacing them with a finite multiplicative constant in front of the expression. 
Altogether, we have thus shown that there exists a constant $C < \infty$ such that 
\begin{multline}
\p \left( \sigma, \sigma' \text{ are 1-good} \right) 
\leq C \p \left( \sigma \text{ is 1-good} \right) \\
\begin{aligned}
&\times \exp \left( - \overline{\Lambda}_{\exp} \left(  \alpha \cdot \frac{N/3}{N/3 - (\ell - (2k - r))} \right) \cdot \left( \frac{N}{3} - (\ell - (2k - r)) \right) \right) \\
&\times \exp \left( - \overline{\Lambda}_{\min} \left( \beta \cdot \frac{N/3}{N/3 - r} \right) \cdot \left( \frac{N}{3} - r \right) \right) \\
&\times \exp \left( - \overline{\Lambda}_{\max} \left( \gamma \cdot \frac{N/6}{N/6 - (k-r)} \right) \cdot \left( \frac{N}{6} - (k-r) \right) \right).\label{eq:function_G}
\end{aligned}
\end{multline}
In the above calculation, and in the rest of the paper, we use the convention $\frac{a}{0}=+\infty$ for every $a>0$. Remember also that we defined $\overline{\Lambda}_{\exp}(+\infty) = \overline{\Lambda}_{\min}(+\infty) = \overline{\Lambda}_{\max}(+\infty) = 0$.
Note that in the computations above we assumed that 
$\sigma_{k} = \eta_{k} < \sigma'_{k}$, 
but the exact same bound also holds if instead 
$\sigma'_{k} = \eta_{k} < \sigma_{k}$ (as the same arguments apply, with the roles of $\sigma$ and $\sigma'$ switched). Note that the bound above depends on $\eta$ only through $k$, $\ell$, and $r$. Combining this with~\eqref{eq:1goodprob} yields that 
\[
    \p \left( \sigma, \sigma' \text{ are 1-good} \right) 
\leq C \cdot  \exp \left(-(\phi_N + \varphi_N(k,\ell,r)) +o(N) \right),
\]
where we define the function $ \varphi_N(k,\ell,r)$  to be the sum of the terms  in the three exponential factors of~\eqref{eq:function_G}, that is, 
\begin{align}
     \varphi_N(k,\ell,r) := &  \overline{\Lambda}_{\exp} \left(  \alpha \cdot \frac{N/3}{N/3 - (\ell - (2k - r))} \right) \cdot \left( \frac{N}{3} - (\ell - (2k - r)) \right)\\
     & + \overline{\Lambda}_{\min} \left( \beta \cdot \frac{N/3}{N/3 - r} \right) \cdot \left( \frac{N}{3} - r \right)
     + \overline{\Lambda}_{\max} \left( \gamma \cdot \frac{N/6}{N/6 - (k-r)} \right) \cdot \left( \frac{N}{6} - (k-r) \right). 
\end{align}

Plugging this into~\eqref{eq:doublesum_to_eta}, we obtain that
\begin{multline}\label{eq:intermediate_step}
      \sum_{\sigma \in V_N^{\mathtt{typ}}} \sum_{\sigma^\prime \in V_N^{\mathtt{typ}}} \mathbb{P} \left( \sigma, \sigma^\prime \text{ are 1-good}  \right)^{2} \\
        \leq    C^{2}  e^{-2\phi_N+o(N)}
        \sum_{\sigma \in V_N^{\mathtt{typ}}} \sum_{\eta: \eta \preccurlyeq \sigma}   e^{- 2\varphi_N(k,\ell,r)}
        \left|\{\sigma' \in V_N^{\mathtt{typ}}: \ycp(\sigma, \sigma') = \eta\}\right|,
\end{multline}
where we recall that $k = |\eta|_0$, $\ell = |\eta|_1$, and $r = |\{i\in [k]: \eta_i \text{ is odd}\}|$. 
Recall also that we cannot have that $\ycp(\sigma, \sigma') = \varnothing$ (since $(1) \in V_{\infty}$ is a predecessor of both $\sigma$ and $\sigma'$), so the inner sum in~\eqref{eq:intermediate_step} goes over all $\eta$ such that $\eta \preccurlyeq \sigma$ and $\eta \neq \varnothing$. 
To bound the quantity in~\eqref{eq:intermediate_step}, we use the observation that $\{\sigma^\prime \in V_N^{\mathtt{typ}}: \ycp(\sigma, \sigma') = \eta\} \subseteq \{\sigma' \in V_N^{\mathtt{typ}}: \eta \preccurlyeq \sigma' \}$. Further, define the set 
\begin{align*}
     \cD := \big\{ (k,\ell,r) \in \N_0^3: 1\leq k\leq N/2, \quad r\leq N/3, \quad 0 \leq k-r\leq N/6,\quad  2k-r\leq \ell < N/3 + 2k-r\big\}.
\end{align*}
By a slight abuse of notation, write $\eta \in \cD$ if $(k,\ell, r)\in \cD$. Note that the first three constraints in $\cD$ are automatically satisfied for any $\eta \neq \varnothing$ with $ \eta \preccurlyeq \sigma $ for some $\sigma \in V_{N}^{\mathtt{typ}}$, by the definition of~$V_{N}^{\mathtt{typ}}$. 
Furthermore, the lower bound on $\ell$ in the definition of $\cD$ follows from the definitions of $k$, $\ell$, and $r$.
Finally, the upper bound on $\ell$ in the definition of $\cD$ allows us to apply Lemma~\ref{lem:descendant_count_in_VN_discrete}.
By Lemma~\ref{lem:descendant_count_in_VN_discrete},  
for $\eta \notin \cD$ 
we have that 
$  \left| \left\{ \sigma \in V_{N}^{\mathtt{typ}} : \eta \preccurlyeq \sigma \right\} \right| = 0$, 
and for $\eta \in \cD$ we have that
\begin{align*}
     \left| \left\{ \sigma \in V_{N}^{\mathtt{typ}} : \eta \preccurlyeq \sigma \right\} \right| 
     \leq 
     4 N^{4} \exp \left(\psi_N(k,\ell,r) \right),
\end{align*} 
where 
\begin{align*}
   \psi_N(k, \ell, r)
   :=  \left( \frac{N}{2} - k \right) H \left( \frac{N/3-r}{N/2-k} \right) 
+ \left( \frac{2N}{3}  -  \frac{\ell + r}{2}  \right) H \left( \frac{N/2 - k}{\frac{2N}{3} -  \frac{\ell + r}{2}} \right) .
\end{align*}
Applying this bound to the right-hand side of~\eqref{eq:intermediate_step} gives that 
\begin{multline*}
      \sum_{\sigma \in V_N^{\mathtt{typ}}} \sum_{\sigma^\prime \in V_N^{\mathtt{typ}}} \mathbb{P} \left( \sigma, \sigma^\prime \text{ are 1-good}  \right)^{2} \\
\begin{aligned}
    &\leq   4 C^{2}  N^4 e^{-2\phi_N+o(N)}
        \sum_{\sigma \in V_N^{\mathtt{typ}}} \sum_{\eta: \eta \preccurlyeq \sigma, \eta \in \cD}   \exp\left({- 2\varphi_N(k,\ell,r)} + \psi_N(k,\ell,r)\right)\\
    &\leq   4 C^{2}  N^4 e^{-2\phi_N+o(N)}
        \sum_{\sigma \in V_N^{\mathtt{typ}}} \left|\left\{\eta: \eta \preccurlyeq \sigma, \eta \neq \varnothing \right\} \right| \cdot  \max_{(k,\ell,r)\in \cD}  \exp\left({- 2\varphi_N(k,\ell,r)} + \psi_N(k,\ell,r)\right) \\
    &= 4 C^{2}  N^5 \left| V_N^{\mathtt{typ}} \right| \cdot e^{-2\phi_N+o(N)} \max_{(k,\ell,r)\in \cD}  \exp\left({- 2\varphi_N(k,\ell,r)} + \psi_N(k,\ell,r)\right),
\end{aligned}
\end{multline*}
where in the last line we used that 
$\left|\left\{\eta: \eta \preccurlyeq \sigma, \eta \neq \varnothing \right\} \right| = N$ 
for every $\sigma \in V_N^{\mathtt{typ}}$. 
Recall that we aim to show~\eqref{eq:2nd_moment_beyondUH}, 
and that 
$\left| V_{N}^{\mathtt{typ}} \right| = 2^{N-o(N)}$ 
and 
$\mathbb{P} \left( \sigma \text{ is 1-good}  \right) = \exp( - \phi_{N} + o(N))$. 
Consequently, by the inequality in the display above, 
it suffices to show that 
\begin{equation}\label{eq:intermediate_step_two}
    \max_{(k,\ell,r)\in \cD} \exp\left({- 2\varphi_N(k,\ell,r)} + \psi_N(k,\ell,r)\right) 
    \leq \exp\left(N\log 2 - 2\phi_N + o(N)\right),
\end{equation}
where the constants and the $N^5$ factor is absorbed into the $e^{o(N)}$. Taking the logarithm, this is equivalent to showing that
\begin{align}\label{eq:maximization_problem}
     \max_{(k,\ell,r)\in \cD} \left\{ - 2\varphi_N(k,\ell,r) + \psi_N(k,\ell,r)\right\} & \leq N\log 2 - 2\phi_N + o(N).
\end{align}

We now pass to a continuous maximization problem. Substituting $u := k/N$, $v:= \ell/N$ and $w:=r/N$, the triple $(u,v,w)$ takes values in 
\begin{align*}
    \cB_N := \left\{(u,v,w) \in \frac{1}{N}\N_0^3: \quad u\leq \frac{1}{2}, \quad w\leq \frac{1}{3},\quad  0 \leq u-w \leq \frac{1}{6}, \quad 2u-w \leq v \leq \frac{1}{3} +2u - w\right\}.
\end{align*}
Letting $N$ tend to infinity, the set $\cB_{N}$ converges to 
\begin{equation}\label{eq:B}
    \cB := \left\{(u,v,w) \in \R_{+}^{3}: \quad u\leq \frac{1}{2}, \quad w\leq \frac{1}{3},\quad  0\leq u-w \leq \frac{1}{6}, \quad 2u - w \leq v \leq \frac{1}{3} +2u - w\right\}.
\end{equation}
Now define 
\begin{align}
    \varphi(u,v,w) & := \frac{1}{N}\varphi_N(uN, vN, wN) \label{eq:scaled_varphi} \\
    & = \overline{\Lambda}_{\exp} \left(\frac{\alpha /3 }{1/3 - (v - (2u - w))} \right) \cdot \left( 1/3 - (v - (2u - w)) \right)\\
     & \quad + \overline{\Lambda}_{\min} \left(  \frac{\beta/3}{1/3 - w} \right) \cdot \left( 1/3 - w \right)
     + \overline{\Lambda}_{\max} \left(\frac{\gamma /6}{1/6 - (u-w)} \right) \cdot \left( 1/6 - (u-w) \right),
\end{align}
and 
\begin{align}
    \psi(u,v,w) & := \frac{1}{N} \psi_N(uN,vN,wN) \label{eq:scaled_psi} \\
    & = (1/2 - u)H\left(\frac{1/3 - w}{1/2 - u}\right) + (2/3 - (v+w)/2)H\left(\frac{1/2 - u}{2/3 - (v+w)/2}\right).
\end{align}
Observe that $\varphi$ and $\psi$ do not depend on $N$, and further, that 
\begin{align*}
    \phi := \frac{1}{N}\phi_N = \frac{1}{3} \Lambda_{\exp}^{\star} (\alpha) + \frac{1}{3} \Lambda_{\min}^{\star} (\beta)+\frac{1}{6} \Lambda_{\max}^{\star} (\gamma)
\end{align*}
does not depend on $N$ either, where we recall $\phi_N$ from~\eqref{eq:1goodprob}. We are now ready to move from the discrete maximization problem in~\eqref{eq:maximization_problem} to a continuous setting. By dividing both sides of~\eqref{eq:maximization_problem} by $N$ and letting $N\to \infty$, in order to show~\eqref{eq:maximization_problem} it therefore suffices to show that 
\begin{align}\label{eq:continuous_maximization}
    \sup_{(u,v,w) \in \cB} \left\{ \psi(u,v,w) - 2 \varphi(u,v,w) \right\} \leq \log 2 - 2\phi . 
\end{align}
Note that 
$\varphi(0,0,0) = \phi$ 
and 
$\psi(0,0,0) = \frac{1}{2}H(\frac{2}{3}) + \frac{2}{3}H(\frac{3}{4}) = \log 2$, 
so equality is attained in~\eqref{eq:continuous_maximization} by taking $u=v=w=0$. 
The inequality~\eqref{eq:continuous_maximization} follows directly 
by adding the two inequalities in Lemma~\ref{lem:inequalities_new} (and noting again that $\frac{1}{2}H(\frac{2}{3}) + \frac{2}{3}H(\frac{3}{4}) = \log 2$).
\end{proof}

\subsection{Technical statements} \label{subsec: technical statements}

We prove several technical results here that are used in the previous parts of Section~\ref{sec:beyondUH}.

\begin{lemma}\label{lemma:mgfs}
    Let $U,V,$ and $W$ be i.i.d.\ exponential random variables with mean $1$. Then, for $\lambda<2$ we have that 
    \begin{align}
    \E \left[ \exp \left(  \lambda \min \left\{ U + V, W \right\} \right) \right] &= \frac{2}{ (2-\lambda )^2} + \frac{1}{2 -\lambda },
    \intertext{and for $\lambda<1$ we have that}
    \E \left[ \exp \left(  \lambda \max \left\{ U + V, W \right\} \right) \right] &= \frac{4-3\lambda }{(1-\lambda )^2 (2-\lambda )^2}.
    \end{align}
\end{lemma}

\begin{proof}
Let us write $g(t):=\e^{-t}\mathbbm 1_{\{t>0\}}$ for the density of the random variables $U$, $V$, and $W$. By the convolution formula for densities of random variables, the density of $U+V$ is given by
\begin{equation*}
    g_{U+V}(t) = \int_{\R} g(s) g(t-s)\,\dd s = t \e^{-t} \mathbbm{1}_{\{t>0\}}.
\end{equation*}
Thus for all $a>0$ we get that 
\begin{equation*}
    \p \left( U+V > a \right) = \int_{a}^\infty t \e^{-t} \dd t = (a+1)\e^{-a}.
\end{equation*}
Hence, by independence of $U+V$ and $W$, for all $a>0$ we have that 
\begin{equation*}
    \p\left( \min\left\{U+V,W \right\} > a \right) = \p \left( U+V > a \right) \p \left( W > a \right)  = (a+1) \e^{-2a}.
\end{equation*}
As a result, the probability density function of $\min\{U+V,W\}$ is given by
\begin{equation*}
    g_{\min}(t) =
    (2t+1)\e^{-2t}
    \mathbbm{1}_{\{t>0\}} .
\end{equation*}
Integrating against the probability density function, we finally get, for any $\lambda < 2$, that 
\begin{align*}
    \E \left[ \exp\left( \lambda \min\{U+V,W\} \right) \right]= \int_0^\infty  (2t+1)\e^{-2t} \e^{\lambda t}\, \dd t=\frac{2}{(2-\lambda)^2}+ \frac{1}{2-\lambda}.
\end{align*}
Next, let us turn to the maximum. Here, for any $a>0$ we have that 
\begin{equation*}
    \p \left( \max\{U+V,W\} \leq a \right) =
    \p \left( U+V\leq a \right)
    \p \left( W \leq a \right)
    =
    \left( 1-(a+1)\e^{-a} \right) \left( 1-\e^{-a} \right),
\end{equation*}
and thus we get that the probability density function of $\max\{U+V,W\}$ is 
\begin{equation*}
    g_{\max}(t) =  \left((1+t)\e^{-t}-(1+2t)\e^{-2t} \right)\indicator_{\{t>0\}}.
\end{equation*}
Integrating against this density, we get, for any $\lambda <1$, that 
\begin{align*}
    \E \left[ \exp\left( \lambda \max\{U+V,W\} \right) \right]
    = 
    \int_0^\infty  \left((1+t)\e^{-t}-(1+2t)\e^{-2t} \right) \e^{\lambda t} \,\dd t= \frac{4-3\lambda}{(1-\lambda)^2(2-\lambda)^2},
\end{align*}
which concludes the proof.
\end{proof}

Next, we determine the large deviation rate functions of the random variables $\min\{U+V,W\}$ and $\max\{U+V,W\}$, which we denote by $\Lambda^{\star}_{\min}$ and $\Lambda_{\max}^{\star}$, respectively. 

\begin{lemma}\label{lemma:ratefunc}
    Let $U$, $V$, and $W$ be i.i.d.\ exponential random variables with mean $1$, and let $\Lambda^{\star}_{\min}$ (respectively, $\Lambda^{\star}_{\max}$) denote the large deviation rate function of the random variable $\min\{U+V,W\}$ (respectively, $\max\left\{U+V,W \right\}$). Then, for $a\in (0,3/4]$, we have that 
    \begin{align}
    \Lambda^{\star}_{\min}(a) &=  a x_{-}(a) - \log \left( \frac{2}{(2-x_{-}(a))^2} + \frac{1}{2-x_{-}(a)} \right), \label{eq:lambdamin}
    \intertext{and for $a\in(0,9/4]$, we have that }
    \Lambda_{\max}^{\star} \left(a\right)  &= a x_{\max}(a) - \log \left(  \frac{-3x_{\max}(a) + 4}{(1-x_{\max}(a))^2 (2-x_{\max}(a))^2} \right), \label{eq:lambdamax}
    \end{align}    
    where $x_{-}(a)$ and $x_{\max}(a)$ are defined in~\eqref{eq:xmin} and~\eqref{eq:xmax}, respectively.
\end{lemma}

\begin{remark}\label{remark:minmax_expectation}
    We are only interested in the domains $(0,3/4]$ and $(0,9/4]$ for the functions $\Lambda^{\star}_{\min}$ and $\Lambda^{\star}_{\max}$, respectively, since $\E[\min\{U+V,W\}]=3/4$ and $\E[\max\{U+V,W\}]=9/4$.
\end{remark}

\begin{proof}[Proof of Lemma~\ref{lemma:ratefunc}] 
By Cram\'er's theorem, the large deviation rate function is the Legendre transform of the logarithmic moment generating function. In Lemma~\ref{lemma:mgfs} we computed the relevant moment generating functions, so what remains is to compute the associated Legendre transform. 
We first determine $\Lambda^{\star}_{\min}$. Using Lemma~\ref{lemma:mgfs}, we have that 
\begin{equation}\label{eq:LDP Cramer min}
    \Lambda^\star_{\min}(a) = \sup_{x < 0} \left\{  ax - \log \left( \frac{2}{(2-x)^2} + \frac{1}{2-x} \right) \right\}. 
\end{equation}
To find the value of $x$ that attains the supremum, we compute the derivative with respect to $x$:
\begin{equation}
    \frac{\dd}{\dd x} \left\{ax - \log \left( \frac{2}{(2-x)^2} + \frac{1}{2-x} \right) \right\} =
    \frac{a(x^2-6x+8)+x-6}{(x-4)(x-2)},  
\end{equation}
and set this equal to $0$. 
The resulting equation has the two solutions
\be \label{eq:xmin}
    x_{\pm}(a) = \frac{(6a-1)\pm \sqrt{ (6a-1)^2 - 4a(8a-6) }}{2a} .
\ee
For $a< 3/4$, the solution $x_{-} = x_{-}(a)$ is indeed negative and thus we get that the supremum in~\eqref{eq:LDP Cramer min} is realized at $x=x_{-}$. 
In particular, 
for $a<3/4$ we have that~\eqref{eq:lambdamin} holds, 
and this also holds for $a=3/4$ 
(since 
$\Lambda^{\star}_{\min}(3/4) = 0$).

Next, let us turn to $\Lambda^{\star}_{\max}$. Here, again using Lemma~\ref{lemma:mgfs}, we have for $a \leq 9/4$ that 
\begin{equation}\label{eq:LDP Cramer max}
    \Lambda^{\star}_{\max}(a) = \sup_{x < 0} \left\{ ax - \log \left( \frac{-3x+4}{(1-x)^2(2-x)^2}  \right) 
 \right\}.
\end{equation}
Taking the derivative with respect to $x$  yields
\begin{align*}
    \frac{\dd}{\dd x} \left\{  ax - \log \left( \frac{-3x+4}{(1-x)^2(2-x)^2}  \right) \right\}
    =
    \frac{a(3x^3 - 13 x^2 + 18x -8) +9x^2 -25x + 18}{(x-2)(x-1)(3x-4)}.
\end{align*}
This implies that we want to solve the cubic equation
\begin{align}\label{eq:cubic}
    3ax^3+(9-13a)x^2+(18a-25)x+(18-8a) = 0
\end{align}
for $a \in (0,9/4]$. 
Note that the discriminant of this cubic equation is negative for all $a \in (0,9/4]$, 
so there is a unique real solution. 
Furthermore, 
since both the constant term and the coefficient of the cubic term in~\eqref{eq:cubic} 
are positive 
for $a \in (0, 9/4)$, 
it follows that the unique real solution of~\eqref{eq:cubic} is negative for $a \in (0, 9/4)$; 
in addition, for $a = 9/4$ we have that $18-8a = 0$, 
so the unique real solution is $0$. 
For $a \in (0,9/4)$, the unique real solution $x_{\max}(a)$ 
can be computed using Cardano's formula
and 
is given by
\be \ba\label{eq:xmax}
x_{\max}(a)={}&\frac{13a-9}{9a}\\
&+\Bigg(\frac{ 20 a^3 - 351 a^2 + 243 a - 1458}{1458a^3}+\sqrt{\frac{-4 a^4 - 36 a^3 + 323 a^2 - 426 a + 1863}{8748 a^4}}\Bigg)^{1/3}\\
&+\Bigg(\frac{20 a^3 - 351 a^2 + 243 a - 1458}{1458 a^3}-\sqrt{\frac{-4 a^4 - 36 a^3 + 323 a^2 - 426 a + 1863}{8748 a^4}}\Bigg)^{1/3}.
\ea \ee 
\invisible{
\begin{align*}
        & x_{\max}(a)  = \frac{13a-9}{9a}+  \Bigg( -\frac{-9\cdot 3a(9-13a)(18a-25)+27\cdot 9a^2(18-8a) + 2(9-13a)^3}{1458a^3} \\
        &+ \sqrt{ \left(\frac{-9\cdot 3a(9-13a)(18a-25)+27\cdot 9a^2(18-8a) + 2(9-13a)^3}{1458a^3}\right)^2 + \left( \frac{27a(18a-25)-3(9-13a)^2}{243a^2} \right)^3}\Bigg)^{1/3} \\
        &
        + \Bigg( -\frac{-9\cdot 3a(9-13a)(18a-25)+27\cdot 9a^2(18-8a) + 2(9-13a)^3}{1458a^3}  \\
        &- \sqrt{ \left(\frac{-9\cdot 3a(9-13a)(18a-25)+27\cdot 9a^2(18-8a) + 2(9-13a)^3}{1458a^3}\right)^2 + \left( \frac{27a(18a-25)-3(9-13a)^2}{243a^2} \right)^3} \Bigg)^{1/3}
\end{align*}
}
This implies that for $a \in (0, 9/4]$ the mapping $x\mapsto ax - \log \left( \frac{-3x+4}{(1-x)^2(2-x)^2}  \right)$ is maximal at $x_{\max}(a)$ over all non-positive values of $x$. 
This shows~\eqref{eq:lambdamax}, 
which concludes the proof.
\end{proof}

The following lemma proves two inequalities, which together imply~\eqref{eq:continuous_maximization}, which is needed to conclude the second moment bound in Section~\ref{subsec:High probab}. For this lemma, we also use the convention that $0H(0/0) = 0$.

\begin{lemma}\label{lem:inequalities_new}
    Recall the functions  $\Lambda_{\exp}^{\star}$, $\Lambda_{\min}^{\star}$, $\Lambda_{\max}^{\star}$, and $\overline \Lambda_{\mathrm{exp}}$, $\overline \Lambda_{\min}$, $\overline \Lambda_{\max}$ from \eqref{eq:exp_rate_function}, Lemma~\ref{lemma:ratefunc}, and \eqref{eq:barfunc}, respectively. 
    Recall also the choice of constants $\alpha$, $\beta$, and $\gamma$ from~\eqref{eq:constants}. 
    Finally, recall the set $\cB$ defined in~\eqref{eq:B}. 
    For all $(u,v,w) \in \cB$, we have that
    \begin{multline}\label{eq:inequalities_new 1}
         \left( \frac{2}{3} - \frac{v+w}{2} \right) H\left(\frac{1/2 - u}{2/3 - (v+w)/2}\right)  -  2 \overline{\Lambda}_{\exp} \left(\frac{\alpha /3 }{1/3 - (v - (2u - w))} \right) \cdot \left( \frac{1}{3} - (v - (2u - w)) \right)\\
         \leq \frac{2}{3}H(3/4) - \frac{2}{3} \Lambda_{\exp}^{\star} (\alpha),
    \end{multline}
    and
    \begin{multline}\label{eq:inequalities_new 2}
         \left( \frac{1}{2} - u \right) H\left(\frac{1/3 - w}{1/2 - u}\right) - 2\overline{\Lambda}_{\min} \left(  \frac{\beta/3}{1/3 - w} \right) \cdot \left( \frac{1}{3} - w \right)
     - 2\overline{\Lambda}_{\max} \left(\frac{\gamma /6}{1/6 - (u-w)} \right) \cdot \left( \frac{1}{6} - (u-w) \right)\\
         \leq \frac{1}{2}H(2/3) - \frac{2}{3} \Lambda_{\min}^{\star} (\beta)-\frac{1}{3} \Lambda_{\max}^{\star} (\gamma).
    \end{multline}
\end{lemma}

\begin{remark}
Observe that~\eqref{eq:inequalities_new 1} involves only the functions $\Lambda_{\exp}^{\star}$ and $\overline{\Lambda}_{\exp}$, and the constant $\alpha$, 
while~\eqref{eq:inequalities_new 2} involves only 
$\Lambda_{\min}^{\star}$, 
$\Lambda_{\max}^{\star}$, 
$\overline{\Lambda}_{\min}$, 
$\overline{\Lambda}_{\max}$, 
and the constants $\beta$ and $\gamma$. 
The proof of~\eqref{eq:inequalities_new 1} 
depends only mildly on the specific value of $\alpha$ chosen in~\eqref{eq:constants}; 
in fact, it can be modified to show  that~\eqref{eq:inequalities_new 1} holds for all $\alpha \in (0.71,1)$. 
On the other hand, due to the fact that the functions 
$\Lambda_{\min}^{\star}$ and 
$\Lambda_{\max}^{\star}$ 
are more involved, 
the proof of~\eqref{eq:inequalities_new 2} is specific to the choices of $\beta$ and $\gamma$ in~\eqref{eq:constants}. 
\end{remark}

\begin{proof}[Proof of Lemma~\ref{lem:inequalities_new}]
We first show \eqref{eq:inequalities_new 1}. Observe that using $H(x) = H(1-x)$, one can rewrite
\begin{align*}
      \left(\frac{2}{3} - \frac{v+w}{2} \right)
      H\left(\frac{1/2 - u}{2/3 - (v+w)/2}\right)  & = \left(\frac{2}{3} - \frac{v+w}{2} \right)
      H\left(\frac{2/3 - (v+w)/2 - (1/2-u)}{2/3 - (v+w)/2}\right)
      \\
      &
      =
      \left(\frac{2}{3} - \frac{v+w}{2} \right)
      H\left(\frac{1/6 - ((v+w)/2 - u) }{2/3 - (v+w)/2}\right).
\end{align*}   
Setting $t := v+w-2u \in [0,1/3]$, and observing that $\frac{v+w}{2} = \frac{t}{2}+ u$, we can bound this term from above by
\begin{align*}
    \left(\frac{2}{3} - \frac{v+w}{2} \right) 
    H\left(\frac{1/6 - ((v+w)/2 - u) }{2/3 - (v+w)/2}\right) &
    =
    \left(\frac{2}{3} - \left(\frac{t}{2} + u\right) \right)
    H\left(\frac{1/6 - \frac{t}{2} }{2/3 - \left(\frac{t}{2} + u\right)}\right)
    \\
    &
    \leq
    \left(\frac{2}{3} - \frac{t}{2} \right) H\left(\frac{1/6 - \frac{t}{2} }{2/3 - \frac{t}{2}}\right)
    =
    \left(\frac{2}{3} - \frac{t}{2} \right)
    H\left(\frac{1/2}{2/3 - \frac{t}{2}}\right),
\end{align*}
where in the inequality we used that 
for each fixed $y>0$, the function
$x \mapsto \left(x+y \right)H\left( y/(x+y) \right)$
is increasing in $x \in (0,\infty)$.
Thus, it suffices to show that
\begin{equation*}
    \left( \frac{2}{3} - \frac{t}{2} \right) H\left(\frac{1/2}{2/3 - \frac{t}{2}}\right) -  2 \overline{\Lambda}_{\exp} \left(\frac{\alpha /3 }{1/3 - t} \right) \cdot \left( \frac{1}{3} - t \right)
    \leq
    \frac{2}{3}H(3/4) - \frac{2}{3} \Lambda_{\exp}^{\star} (\alpha)
\end{equation*}
for all $t \in  \left[0,\frac{1}{3}\right]$, 
which is equivalent to showing that
\begin{equation}\label{equation:reference1}
    \frac{1}{2} \left( \frac{4}{3} - t \right) H\left(\frac{1}{4/3 - t}\right) -  2 \overline{\Lambda}_{\exp} \left(\frac{\alpha}{1 - 3t} \right) \cdot \left( \frac{1}{3} - t \right)
    \leq
    \frac{2}{3}H(3/4) - \frac{2}{3} \Lambda_{\exp}^{\star} (\alpha)
\end{equation}
for all $t\in \left[0,\frac{1}{3}\right]$. 
Note that the right-hand side of this inequality is exactly the expression of the left-hand side for $t=0$. Thus, it is sufficient to show that the function $t\mapsto Q(t)$ defined by
\begin{equation*}
    Q(t) \coloneqq \frac{1}{2} \left(\frac{4}{3} - t \right) H\left(\frac{1}{4/3 - t}\right) -  2 \overline{\Lambda}_{\exp} \left(\frac{\alpha}{1 - 3t} \right) \cdot \left( \frac{1}{3} - t \right)
\end{equation*}
is decreasing on $\left[0,\frac{1}{3}\right]$. 
This can be readily seen in Figure~\ref{fig:qfunction}; we next provide a formal proof.
\begin{figure}[t!]
    \centering
    \includegraphics[width=0.5\textwidth]{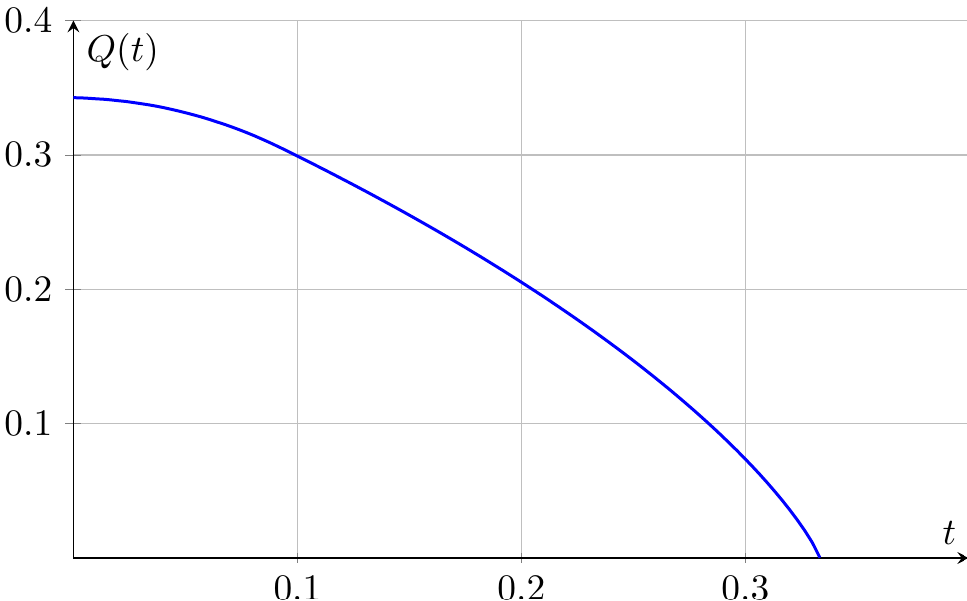}
    \caption{The function $Q$ in blue. Note that the function $t \mapsto Q(t)$ is decreasing for $t\leq \frac{1}{3}$.
    }
    \label{fig:qfunction}
\end{figure}

First, note that $\overline{\Lambda}_{\exp} \left(\frac{\alpha}{1 - 3t} \right) = 0$ for $\frac{\alpha}{1 - 3t} > 1$, that is, for $t > \frac{1-\alpha}{3}=\frac{7}{75}$. Thus, for $t > \frac{7}{75}$ we see that
\begin{align*}
    Q(t) = \frac{1}{2} \left( \frac{4}{3} - t \right) 
    H\left(\frac{1}{4/3 - t}\right), 
\end{align*}
which is decreasing in $t$. 

Next, we turn to $t\leq \frac{7}{75}$, in which case the function $Q$ simplifies to
\begin{align*}
    Q(t) = \frac{1}{2}\left( \frac{4}{3} - t \right) 
    H\left(\frac{1}{4/3 - t}\right) -  2 \left( \frac{1}{3} - t \right) \Lambda^{\star}_{\exp} \left(\frac{\alpha/3}{1/3 - t} \right) .
\end{align*}
The first and second derivatives of the functions $H(x)=-x \log(x) - (1-x) \log(1-x)$ and $\Lambda_{\exp}^\star(x)=x-1-\log(x)$ are given by
\begin{align*}
    \frac{\dd}{\dd x} H(x) &= - \log(x) + \log(1-x) , 
    &\quad
    \frac{\dd^2}{\dd x^2} H(x) &= -\frac{1}{x} - \frac{1}{1-x}, 
    \\
    \frac{\dd}{\dd x} \Lambda_{\exp}^\star(x) &= 1 - \frac{1}{x} , 
    &\quad
    \frac{\dd^2}{\dd^2 x} \Lambda_{\exp}^\star(x) &= \frac{1}{x^2}.
\end{align*}
Writing $H^\prime, H^{\prime \prime}$ and $\Lambda_{\exp}^{\star\prime}, \Lambda_{\exp}^{\star\prime\prime}$ for the first and second derivatives of $H$ and $\Lambda_{\exp}^{\star}$, respectively, we see that the first and second derivatives of $Q$ for $t\in \left[0,\frac{7}{75}\right]$ are given by
\begin{align*}
    \frac{\dd}{\dd t} Q(t) & = - \frac{1}{2} H \left( \frac{1}{\frac{4}{3}-t} \right) 
    + 
    \frac{1}{2} \left( \frac{4}{3}-t \right)^{-1} H^\prime \left( \frac{1}{\frac{4}{3}-t} \right)
    +
    2 \Lambda_{\exp}^{\star} \left( \frac{\frac{\alpha}{3}}{\frac{1}{3}-t} \right) 
    - 
    2 \Lambda_{\exp}^{\star \prime} \left( \frac{\frac{\alpha}{3}}{\frac{1}{3}-t} \right) \frac{\frac{\alpha}{3}}{\frac{1}{3}-t}
    \\
    \frac{\dd^2}{\dd t^2} Q(t) & = \frac{1}{2} \frac{1}{\left( \frac{4}{3}-t \right)^3} H^{\prime\prime} \left(\frac{1}{\frac{4}{3}-t}\right)-2\Lambda_{\exp}^{\star\prime\prime} \left(\frac{\frac{\alpha}{3}}{\frac{1}{3}-t}\right)\cdot\frac{\left(\frac{\alpha}{3}\right)^{2}}{\left(\frac{1}{3}-t\right)^{3}}
    \\
    &
    = -\frac{1}{2\left(\frac{4}{3}-t\right) \left( \frac{1}{3}-t \right)} 
    -\frac{2}{\frac{1}{3}-t}
\end{align*} 
In particular, we see that for each $t \in \left[0,\frac{7}{75} \right]$, we have that 
$\frac{\dd^2}{\dd t^2} Q(t) < 0$. 
This implies that 
$\frac{\dd}{\dd t} Q(t)$ 
is decreasing for $t \in \left[0,\frac{7}{75} \right]$. 
We check further that 
$\frac{\dd}{\dd t} Q(t) \leq -0.03$ for $t = 0$, 
and hence 
$\frac{\dd}{\dd t} Q(t) \leq -0.03$ for all $t \in \left[0,\frac{7}{75} \right]$. 
In particular, this implies that the function $t \mapsto Q(t)$ is non-increasing on the interval $\left[0,\frac{7}{75}\right]$. Since we already established earlier that the function $t \mapsto Q(t)$ is non-increasing on the interval $\left[\frac{7}{75},\frac{1}{3}\right]$, this establishes \eqref{equation:reference1} and thus inequality \eqref{eq:inequalities_new 1}.

We proceed with the proof of inequality~\eqref{eq:inequalities_new 2}. 
We set $s\coloneqq u-w \in \left[0,\frac{1}{6} \right]$. Using that $u=s+w$, we have that showing~\eqref{eq:inequalities_new 2} is equivalent to showing that 
\begin{multline*}
         \left( \frac{1}{2} - (s+w) \right) H\left(\frac{1/3 - w}{1/2 - (s+w)}\right) - 2\overline{\Lambda}_{\min} \left(  \frac{\beta/3}{1/3 - w} \right) \cdot \left( \frac{1}{3} - w \right)
     - 2\overline{\Lambda}_{\max} \left(\frac{\gamma /6}{1/6 - s} \right) \cdot \left( \frac{1}{6} - s \right)\\
         \leq \frac{1}{2}H(2/3) - \frac{2}{3} \Lambda_{\min}^{\star} (\beta)-\frac{1}{3} \Lambda_{\max}^{\star} (\gamma).
\end{multline*}
for all $s\in \left[0,\frac{1}{6}\right]$ and $w\in \left[0,\frac{1}{3} \right]$. 
Setting 
$t:=\frac{1}{6}-s$ and 
$z:=\frac{1}{3}-w$, 
and using that $H(x)=H(1-x)$, we see that showing the inequality in the display above is equivalent to showing that
\begin{equation}\label{eq:ineq H min max}
    (t+z)H\left(\frac{t}{t+z}\right) - 2z \overline{\Lambda}_{\min} \left(  \frac{\beta}{3z} \right)
     - 2t \overline{\Lambda}_{\max} \left(\frac{\gamma }{6t} \right)  \leq \frac{1}{2}H(2/3) - \frac{2}{3} \Lambda_{\min}^{\star} (\beta)-\frac{1}{3} \Lambda_{\max}^{\star} (\gamma) 
\end{equation}
for all $z \in \left[0,\frac{1}{3}\right]$ and $t \in \left[0,\frac{1}{6}\right]$. 

To show~\eqref{eq:ineq H min max}, we first define the function
    \begin{equation*}
        F(z,t) := (t+z) H \left( \frac{t}{t+z} \right) - 2z \overline{\Lambda}_{\min} \left( \frac{\beta}{3z} \right) - 2t \overline{\Lambda}_{\max} \left( \frac{\gamma}{6 t} \right) 
        = F_{1}(z,t) + F_{2} (z,t) + F_{3}(z,t),
    \end{equation*}
for 
$(z,t) \in \R_{+}^{2}$, 
and where  
$F_{1}(z,t) := (t+z) H \left( \frac{t}{t+z} \right)$, 
$F_{2}(z,t) := F_{2}(z) := - 2z \overline{\Lambda}_{\min} \left( \frac{\beta}{3z} \right)$, 
and 
$F_{3}(z,t) := F_{3}(t) := - 2t \overline{\Lambda}_{\max} \left( \frac{\gamma}{6 t} \right)$. 
See Figure~\ref{fig:ffunction} for a plot of this function. 
Noting that the right hand side of~\eqref{eq:ineq H min max} is equal to $F\left(\frac{1}{3},\frac{1}{6} \right)$, our goal now is to show that 
\[
F(z,t) \leq F\left(\frac{1}{3},\frac{1}{6} \right) 
\quad \text{ for all } z \in \left[0,\frac{1}{3}\right], t \in \left[0,\frac{1}{6}\right].
\]

    \begin{figure}[t!]
    \centering
    \includegraphics[width=0.5\textwidth]{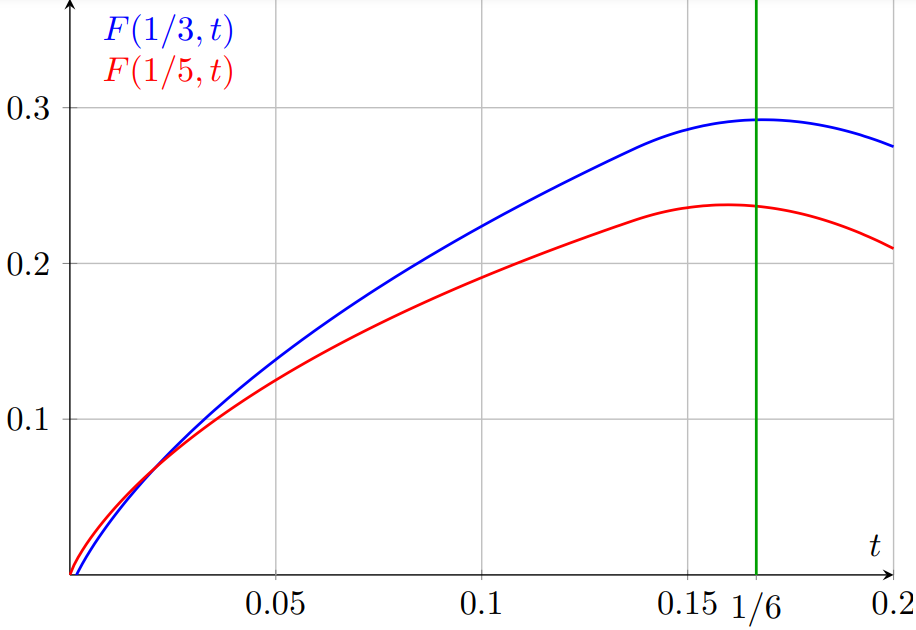}
    \caption{The function $F(z,t)$ for $z=\frac{1}{3}$ (blue) and $z=\frac{1}{5}$ (red). The global maximum of $F$ over the domain $0\leq z \leq \frac{1}{3}, 0\leq t \leq \frac{1}{6}$ is attained at $z=\frac{1}{3}, t= \frac{1}{6}$. Note also that the maximum of $t \mapsto F(\frac{1}{3}, t)$ is attained at a value of $t$ that is slightly larger than $\frac{1}{6}$. 
    }
    \label{fig:ffunction}
    \end{figure}

First, we argue that the function $F$, defined on the domain $\left[ 0, \tfrac{1}{3} \right] \times \left[ 0, \tfrac{1}{6} \right]$, has a maximizer $(z,t)$ with $z \geq \frac{\beta}{3 \cdot 0.75}$ and $t \geq \frac{\gamma}{6 \cdot 2.25}$. 
Note that 
$\overline{\Lambda}_{\min} \left( \frac{\beta}{3z} \right) = 0$ 
for $z \leq \frac{\beta}{3 \cdot 0.75}$ 
by definition. 
Consequently, for $z \leq \frac{\beta}{3 \cdot 0.75}$ 
we have that 
$F(z,t) = F_{1}(z,t) + F_{3}(t)$. 
For fixed $t$, the function 
$z \mapsto F_{1}(z,t) = (t+z) H \left( \frac{t}{t+z} \right)$ 
is increasing in $z$, 
and the function $F_{3}(t)$ is constant in $z$, 
so overall $F(z,t)$ is increasing in $z \in \left[0, \frac{\beta}{3 \cdot 0.75} \right]$ for every fixed $t$. 
An analogous argument shows that 
$F(z,t)$ is increasing in $t \in \left[0, \frac{\gamma}{6 \cdot 2.25} \right]$ for every fixed $z$. 
In sum, it remains to show that 
$F(z,t) \leq F\left( \frac{1}{3}, \frac{1}{6} \right)$
for all 
$(z,t) \in \left[ \frac{\beta}{3 \cdot 0.75}, \frac{1}{3} \right] \times \left[ \frac{\gamma}{6 \cdot 2.25}, \frac{1}{6} \right]$.

Next, we show that the function $F$ is concave on the domain $\mathbb{R}_{+}^{2}$ by showing that the functions $F_{1}$, $F_{2}$, and $F_{3}$ are each individually concave. 
We start with $F_{1}$. 
We have that
    \begin{equation*}
        F_1(z,t) = \limsup_{N\to \infty} \frac{\log \left(\binom{\lfloor (t+z)N \rfloor}{\lfloor tN \rfloor} \right)}{N}.
    \end{equation*}
    By elementary combinatorics we have for all integers $k_1$, $k_2$, $n_1$, and $n_2$ that $\binom{n_1+k_1+n_2+k_2}{k_1+k_2} \geq \binom{n_1+k_1}{k_1} \binom{n_2+k_2}{k_2}$. Using this inequality we have for any $z_1,t_1,z_2,t_2 \geq 0$ and $\lambda \in \left[0,1\right]$ that
    \begin{align*}
        F_1 \left( \lambda (z_1,t_1)+ (1-\lambda) (z_2,t_2)  \right)
        &=
        \limsup_{N\to \infty} \frac{\log \left( \binom{\lfloor \lambda (z_1+t_1)N \rfloor + \lfloor(1-\lambda)(z_2+t_2)N \rfloor }{\lfloor \lambda t_1 N \rfloor + \lfloor (1-\lambda) t_2 N \rfloor}\right)}{N}\\
        &\geq
        \limsup_{N\to \infty} \frac{\log \left( \binom{\lfloor \lambda (z_1+t_1)N \rfloor  }{\lfloor \lambda t_1 N \rfloor }
        \binom{ \lfloor(1-\lambda)(z_2+t_2)N \rfloor }{ \lfloor (1-\lambda) t_2 N \rfloor}
        \right)}{N}\\
        &= 
        \limsup_{N\to \infty} \frac{\log \left( \binom{\lfloor \lambda (z_1+t_1)N \rfloor  }{\lfloor \lambda t_1 N \rfloor }
        \right)}{N}
        +
        \limsup_{N\to \infty} \frac{\log \left( \binom{\lfloor (1-\lambda) (z_2+t_2)N \rfloor  }{\lfloor (1-\lambda) t_2 N \rfloor }
        \right)}{N}\\
        &= 
        \lambda F_1(z_1,t_1) + (1-\lambda) F_1(z_2,t_2),
    \end{align*}
    which shows the concavity of $F_1$. 

Next, we show that $F_{2}$ is concave. 
Since $F_{2}$ depends only on the variable $z$ and not on $t$, it suffices to show concavity in $z$. 
We compute the first and second derivatives of $F_{2}$ and get that 
    \begin{equation*}
        \frac{\dd}{\dd z} F_2(z) 
        = 
        -2 \overline{\Lambda}_{\min} \left( \frac{\beta}{3z} \right) + \frac{2\beta}{3z}\overline{\Lambda}_{\min}^\prime \left( \frac{\beta}{3z} \right) ,
        \qquad
        \text{and}
        \qquad 
        \frac{\dd^2}{\dd z^2} F_2(z) 
        = 
        - \frac{2\beta^2}{9 z^3} \overline{\Lambda}_{\min}^{\prime \prime} \left( \frac{\beta}{3z} \right) \leq 0.
    \end{equation*}
Here the final inequality holds because $\overline{\Lambda}_{\min}^{\prime \prime} \geq 0$, since $\overline{\Lambda}_{\min}$ is convex; the convexity of $\overline{\Lambda}_{\min}$ follows from its definition in \eqref{eq:barfunc} and the fact that $\Lambda_{\min}^\star$ is the large-deviation rate function of a random variable with expectation $3/4$.
An analogous computation and argument also shows that $F_{3}$ is concave. 
Thus $F$, as the sum of three concave functions, is also a concave function.

    Next, we check that
    \begin{align*}
        & F\left( \frac{1}{3} + 0.00001 , \frac{1}{6} \right) = 0.292271136\ldots > 0.292271108\ldots = F\left( \frac{1}{3} , \frac{1}{6} \right) 
        \intertext{ and that } 
        & F\left( \frac{1}{3}  , \frac{1}{6} + 0.00001 \right) = 0.292271679\ldots > 0.292271108\ldots = F\left( \frac{1}{3} , \frac{1}{6} \right).
    \end{align*}
    By concavity of $F$, this already implies that
    \begin{multline*}
        F\left( \frac{1}{3} + \lambda \cdot 0.00001 , \frac{1}{6} + (1-\lambda)\cdot 0.00001 \right) 
        \\
        \geq \lambda F\left( \frac{1}{3} +  0.00001 , \frac{1}{6} \right) + (1-\lambda) F\left( \frac{1}{3} , \frac{1}{6} + 0.00001 \right) >  F\left( \frac{1}{3} , \frac{1}{6} \right)
    \end{multline*}
    for all $\lambda \in \left[ 0,1 \right]$. Define the set
    \begin{equation*}
        S=\left\{\left( \frac{1}{3} + \lambda \cdot 0.00001 , \frac{1}{6} + (1-\lambda)\cdot 0.00001 \right): \lambda \in \left[0,1\right]\right\} \subset \R_{+}^2 ,
    \end{equation*}
    so that the above argument says that $F(s_1,s_2) \geq F\left(\frac{1}{3},\frac{1}{6}\right)$ for all $(s_1,s_2) \in S$.
    For a general point $(z,t) \in \left[ \frac{\beta}{3 \cdot 0.75}, \frac{1}{3} \right] \times \left[ \frac{\gamma}{6 \cdot 2.25}, \frac{1}{6} \right]$ we can write 
    \begin{equation*}
        \left( \frac{1}{3}, \frac{1}{6} \right) = \mu (z,t)+ (1-\mu) (s_1,s_2)
    \end{equation*}
    for some point $(s_1,s_2)\in S$ and some $\mu \in \left(0,1\right]$. The concavity of $F$ now implies that
    \begin{align*}
        F\left(\frac{1}{3},\frac{1}{6}\right) \geq \mu F(z,t) + (1-\mu) F(s_1,s_2)
        \geq \mu F(z,t) + (1-\mu) F\left(\frac{1}{3},\frac{1}{6}\right),
    \end{align*}
    which implies that $F\left(\frac{1}{3},\frac{1}{6}\right) \geq F(z,t)$, 
    as desired. 
\end{proof}

\section{Upper bounds} \label{sec:UB} 

In this section we prove that the largest common subtree of two independent UA trees cannot be too large, as quantified in Theorems~\ref{thm:UB_near1} and~\ref{thm:UB_near0}. 
In Sections~\ref{sec:offspring distribution} and~\ref{sec:roottoroot} we first derive some preliminary results regarding the offspring of vertices in a UA tree. 

In Section \ref{sec:offspring distribution}, we study the sizes of the subtrees of the root. It is well-known that these sizes can be described by a Chinese restaurant process, see Definition \ref{def:chinrest}. We study how likely it is for a deterministic vector $v \in \N^k$ (representing the subtree sizes) to be close (in a way to be defined) to a random vector generated by such a Chinese restaurant process. In Section \ref{sec:roottoroot}, we use these findings inductively to study the probability that a deterministic tree of size $N$ is isomorphic to a UA tree of size $N$. In Section \ref{sec:beyondroottoroot}, we use this, and a union bound over all subtrees $T^\prime$ contained in $T^1$ and $T^2$ of a certain size, to show that, with high probability, the largest common subtree of two UA trees of size $N+ \lfloor 0.0025 \cdot N\rfloor$ has size at most $N$, which implies Theorem \ref{thm:UB_near1}.

In Section \ref{subsec:not_bdd_0}, we prove Theorem \ref{thm:UB_near0}. This section is mostly self-contained, and the result follows from constructing two classes of trees that each occur with uniformly positive probability and whose common subtree is not too large.

\subsection{The structure of the offspring} \label{sec:offspring distribution} 

\noindent As a first step towards the proof of Theorem~\ref{thm:UB_near1}, we investigate the structure of the offspring of a vertex in a UA tree. Run the UA tree for $N$ many steps, so there is a total of $N+1$ vertices, including the root $\varnothing$, and let $K\in \{1,\ldots,N\}$ denote the degree of the root.  The evolution of the sizes of the subtrees of the root can be described in the following manner. 

\begin{definition}[Size of subtrees of $\varnothing$] \label{def:subtree_vector}
Fix $N\in\N$ and initialise $(w_1,\ldots, w_N)=(0,\ldots, 0)$. Then, let $w_1$ be uniformly distributed on $[N]$ and set $i=1$. If $w_1+\cdots +w_i<N$, let $w_{i+1}$ be uniformly distributed on $[N-\sum_{j=1}^i w_j]$, increase $i$ by one, and repeat. If $w_1+\cdots+ w_i=N$ instead, discard $w_{i+1},\ldots, w_N$. It then holds that $w(N)=(w_1,\ldots, w_i)$ is equal in distribution to the sizes of the subtrees rooted at the children of $\varnothing$ in $T_{N+1}$, and we call $w(N)$ the \emph{subtree vector}.    
\end{definition}

\begin{remark}\label{remark:pois diri}
It is readily observed that, for $v\in \N^N$ and $d\in [N]$ such that $v_1+\cdots+v_d=N$ and $v_{d+1}=\ldots=v_N=0$, 
\begin{equation*}
    \p\left(w(N)=v \right) = \prod_{i=1}^{d} \frac{1}{\sum_{j=i}^{d} v_j}.
\end{equation*}
\end{remark}

\noindent The size of the subtrees of $\varnothing$ in $T_N$ can be described in another way as well, namely by a \emph{Chinese restaurant process}. To this end, for $t\in\N_0^\infty$, we let $t_k$ denote the $k^{\mathrm{th}}$ element of $t$, with $k\in\N$, and let $a\cdot t:=(at_1,at_2,\ldots)$ for $a\in\R$ denote element-wise multiplication.  We then define the Chinese restaurant process as follows. 

\begin{definition}[Chinese restaurant process]\label{def:chinrest}
Let $(t^{(n)})_{n\in\N}$ be a sequence of elements $t^{(n)}\in\N_0^\infty$, defined inductively as follows. Initialise $\smash{t_1^{(1)}}:=1$ and $\smash{t_k^{(1)}}:=0$ for all $k\geq 2$.  Given $t^{(n)}$ for some $n\in \N$, let $K=\sup\{k\in\N: \smash{t^{(n)}_k}>0\}$ and define $\smash{t^{(n+1)}}$ by one of the two updates. 
With probability $1/(1+n)$, set   $\smash{t^{(n+1)}_{K+1}}=1$ and $\smash{t^{(n+1)}_j}=\smash{t^{(n)}_j}$ for $j\neq K+1$.
Otherwise, let $i\in[K]$ be a randomly chosen integer, chosen with probability proportionally to $\smash{t^{(n)}_i}$, and set $\smash{t^{(n+1)}_i}=\smash{t^{(n)}_i}+1$ and $\smash{t^{(n+1)}_j}=\smash{t^{(n)}_j}$ for $i\neq j$.
\end{definition}

\noindent This process is known as a Chinese restaurant process, as we can view each step as a customer entering the restaurant who then decides to sit at a table with a probability proportional to the number of customers already sitting at that table, or sits at an empty table with probability proportional to $1$.

The restaurant process is equivalent to the evolution of the subtree sizes of the root $\varnothing$ of a UA tree: $\smash{t^{(n)}}$ describes the sizes of the subtrees of the root $\varnothing$ in $T_{n}$ (see, e.g.,~\cite[Chapter $3$]{Pit06}). The following result regarding the limiting distribution of the restaurant process is well known.

\begin{proposition}[\cite{Pit06}]\label{prop:pdlim}
    Let $(\smash{t^{(n)}})_{n\in\N}$ denote the sequence of table sizes in a Chinese restaurant process, as defined in Definition~\ref{def:chinrest}. Then, 
    \be 
    \frac1n \smash{t^{(n)}}\overset d\longrightarrow Z, 
    \ee 
    where $Z=(Z_i)_{i\in\N}$ can be described as follows. Let $(Y_j)_{j\in\N}$ be i.i.d.\ $\mathrm{Unif}(0,1)$ random variables. Then, $Z_1\overset d=Y_1$, and $Z_i\overset d= Y_i\prod_{j=1}^{i-1}(1-Y_j)$ for $i\geq 2$. 
\end{proposition}

\begin{remark}
    The limit $Z$ is known as the Griffiths-Engen-McCloskey distribution. It can be defined more generally, but for the purposes in this article we restrict ourselves to these particular definitions to avoid using unnecessary notation. The construction of $Z$ using uniform random variables is known as the \emph{stick-breaking construction}.
\end{remark}

\noindent In what follows, we say that, for two vertices $v,w\in T_n$, the vertex
$v$ is a descendant of $w$ if $v$ is a child of $w$ or a descendant of a child of $w$ (i.e., $v$ is in the subtree rooted at $w$). In particular, this definition directly implies that each vertex that is not the root is a descendant of the root. 

As a final preliminary result, we observe the following: the sizes of the subtrees of any vertex $v$ in $T_n$, conditionally on the event that the total number of vertices in all subtrees of $v$ is equal to $k<n$,  are distributed as $t_k$. This holds as whenever a new vertex attaches to $v$, or to a descendant of $v$, then its place of attachment is uniformly distributed among $v$ and all existing descendants of~$v$ (see, e.g.,~\cite{MeirMoon74}). 

\invisible{
\begin{lemma}[\cite{MeirMoon74}]
    Suppose $T_n$ is a random recursive tree on $n$ vertices. Choose an edge uniformly at random and remove it, obtaining two trees $T'$ and $T''$. Conditionally on the event that the trees $T'$ and $T''$ have sizes $k$ and $n-k$, respectively, for some $k\in[n-1]$, the trees $T'$ and $T''$ are distributed as random recursive trees of sizes $k$ and $n-k$, respectively.
\end{lemma}

\noindent As a direct consequence of the above lemma, we obtain that, conditionally on the size $k_v$ of the subtree rooted at some vertex $v$ in $T_n$, the sizes of the subtrees rooted at the children of $v$ are distributed as $t_{k_v}\overset d= \mathrm{PD}(k_v)$.

}

We are interested in the (expected) number of ways to embed a vector into a subtree vector. Let $v\in \N^{d_1}, w\in \N^{d_2}$ be two vectors. We want to count the number of possibilities to reshuffle the vector $v$ and obtain the vector $w$. First, this is only possible if $d_1=d_2$. Furthermore, we do not want to count the number of ways to reshuffle the indices $i$ of $v$ with $v_i=1$, so these indices should appear in the same order in $v$ and in $w$.

\begin{definition}
    We say that a function $f:\{1,\ldots, d_1\} \to \{1,\ldots, d_2\}$ is an \textit{embedding of $v$ into $w$}  if $f$ is bijective, $v_{i}=w_{f(i)}$ for all $i$, and the $j$-th ``\textit{one}'' in $v$ gets mapped to the $j$-th ``\textit{one}'' in $w$, that is, $f: \{i: v_i=1\} \to \{i: w_i=1\}$ is increasing. We write $\{v\rightarrow w\}$ for the set of all possible embeddings of $v$ into $w$. In case there exists at least one embedding of $v$ into $w$, we also say that the vectors $v$ and $w$ are \textit{equivalent}, and we write $v \sim w$. 
\end{definition}

\noindent Let us discuss a few examples. If $v=(1,1,1), w=(2,1)$, then there exists no embedding of $v$ into $w$. If $v=w=(1,1,\ldots,1)$, then $\left| \{v\rightarrow w\} \right| = 1$, as the ordering of the $1$-entries needs to be preserved. If $v=w=(2,2,\ldots,2) \in \N^d$, then $\left| \{v\rightarrow w\} \right| = d!$, as every permutation on $d$ elements is a proper embedding here. If $v=(2,2,3,3)$ and $w=(2,3,2,3)$, then $\left| \{v\rightarrow w\} \right| = 4$, as one has two possibilities to swap the $2$-entries and two possibilities to swap the $3$-entries.  Let us assume that two vectors $v$ and $w$ each have $z_i$ many entries equal to $i$, for each $i\in\N$. We write $z=(z_1,z_2,\ldots)$ for this vector. As defined, $z$ is a vector in the space $\N^{\infty}$, but we have that $z_i=0$ for all $i$ large enough, say for $i>L$. With a slight abuse of notation, we also write $z$ for the vector where we omit these zero entries for $i>L$, that is, if $z \in \N^{\infty}$ is a vector with $z_L \neq 0$ and $z_i=0$ for all $i>L$, we also write $z=(z_1,\ldots, z_L)$. Given that $v$ and $w$ are equivalent, there is only one way to map the $1$-entries on each other. For the $2$-entries, there are $z_2!$ many ways to map them onto each other. This holds for all numbers of at least $2$, and thus there are $\prod_{i=2}^{\infty} z_i! = \prod_{i=2}^{L} z_i!$ many ways to embed $v$ into $w$, given that the two vectors are equivalent. 

Recall, from Definition~\ref{def:subtree_vector}, the subtree vector $w(N).$ Next, for a vector $v\in \N^K$, we estimate the expected number of embeddings in $\{v\rightarrow w(|v|_1)\}$.

\begin{lemma}\label{lemma:pdembed}
    Let $K\geq 1$ be an integer and let $v\in \N^K$ with $|v|_1 \geq 2$. Then,
    \be 
    \E[|\{v\rightarrow w(|v|_1)\}|]\leq 0.9^{K}.
    \ee 
\end{lemma}

\begin{proof} 
We start with the case $K=1$. In this case, we have that $|\{v\rightarrow w(|v|_1)\}| \in \{0,1\}$, with $|\{v\rightarrow w(|v|_1)\}| = 1$ if and only if $w(|v|_1) = (|v|_1)$, which occurs with probability $\frac{1}{|v|_1} \leq \frac{1}{2} \leq 0.9^1$.

From here on, we will assume that $K \geq 2$.
Suppose that, for each $i\in \N$, the vector $v$ contains exactly $z_i\in\N_0$ entries equal to $i$ such that $\sum_{i=1}^\infty z_i = K$ and let $z = (z_1, z_2, z_3, \dots)$.
The statement is clear when $z_1=K$ and $z_i = 0$ for $i\geq 2$, since in this case
\begin{equation*}
    \E[|\{v\rightarrow w(|v|_1)\}|] = \p \left( w(|v|_1) = v  \right) = \frac{1}{K!} \leq 0.9^K
\end{equation*}
for all $K \geq 2$. Thus, we will assume that $z_1 < K$ for the rest of the proof.
There are 
\be 
\binom{K}{z} = \frac{K!}{\prod_{i=1}^{\infty} z_i!}
\ee 
many ways to reorder the vector $v$. We write $\hat{v}$ for the reordering of $v$ in decreasing order, that is, such that $\hat{v}_1\geq \hat{v}_2\geq \ldots  \geq \hat{v}_K.$ For each reordering $\overline{v}$ of $v$ we have that 
\begin{align*}
	\p\left(w\left(|v|_1\right) = \overline{v} \right)
	\leq
	\p\left(w\left(|v|_1\right) =\hat{v} \right),
\end{align*}
as for any vector $v=(v_1,\ldots,v_K)$, by Remark \ref{remark:pois diri},
\begin{align}\label{eq:equality decreasing}
    \p\left(w\left(|v|_1\right) = v \right) = \prod_{i=1}^{K} \frac{1}{\sum_{j=i}^{K} v_j},
\end{align}
which is maximal when the values $(v_1,\ldots, v_K)$ are in decreasing order. Using a union bound over all $\binom{K}{z}$ many reorderings of $v$, we directly get that 
\be 
\p\left(w\left(|v|_1\right) \sim v \right) \leq \binom{K}{z} \p\left(w\left(|v|_1\right) = \hat{v} \right).
\ee
Thus,
\begin{align}\label{eq:isomorphis expect}
	 \notag \E \left[| \{ v \rightarrow w \left(|v|_1\right) \}|\right]  &= \p \left( v \sim w \left(|v|_1\right) \right)
	 \E \left[| \{ v \rightarrow w \left(|v|_1\right) \}| \ \Big| \ v \sim w \left(|v|_1\right) \right]\\
	 & \notag
	 \leq
	\binom{K}{z}
	\p\left(w\left(|v|_1\right) =  \hat{v}  \right)
	\prod_{i=2}^{\infty} z_i!\\
	&
	\leq 
	\frac{K!}{z_1! } \p\left(w\left(|v|_1\right) =  \hat{v}  \right).
\end{align}
From \eqref{eq:equality decreasing} we also get that
\begin{align*}
    \p\left(w\left(|v|_1\right) =  \hat{v}  \right) 
    &= 
    \prod_{i=1}^{K} \frac{1}{\sum_{j=i}^{K}  \hat{v}_j} 
    = 
    \frac{1}{z_1!} \prod_{i=1}^{K-z_1} \frac{1}{\sum_{j=i}^{K}  \hat{v}_j}\\
    &
    \leq
    \frac{1}{z_1!} \prod_{i=1}^{K-z_1} \frac{1}{z_1 + 2(K-z_1 - i+ 1)}
    =
    \frac{1}{z_1!} \prod_{\ell=1}^{K-z_1} \frac{1}{z_1 + 2\ell},
\end{align*}
where we used variable transform $\ell=K-z_1-i+1$ for the last equality. Thus, we obtain for all $K\geq 2$ that 
\begin{align}\label{eq:outdegree bound}
	\E \left[| \{ v \rightarrow w \left(|v|_1\right) \}|\right]\leq \frac{K!}{z_1! } \p\left(w\left(|v|_1\right) =  \hat{v} \right)
    \leq \frac{K!}{z_1!} \frac{1}{z_1!} \prod_{\ell=1}^{K-z_1} \frac{1}{z_1 + 2\ell}
	= 
	\frac{1}{z_1!} \prod_{\ell=1}^{K- z_1} \frac{z_1+\ell}{z_1+2\ell}.
\end{align}
It thus remains to prove that we can bound the right-hand side from above by $0.9^K$. To this end, we define $s:=K-z_1>0$. In the following, we consider three different cases for combinations of $s$ and $z_1$, and we prove the desired upper bound for all three cases.

    \noindent
    \textbf{Case 1:}  $s\geq 2z_1$. Then, for $\ell \geq \lceil s/2 \rceil$ one has $\frac{z_1+\ell}{z_1+2\ell} \leq \frac{2}{3}$. Furthermore, the condition $s\geq 2z_1$ implies that
    $s = 2s/3 + s/3 \geq 2s/3 + 2z_{1}/3 = 2K/3$. 
    Combining these two observations, we get that
    \begin{equation*}
        \prod_{\ell=1}^{s}\frac{z_1+\ell}{z_1+2\ell} \leq \prod_{\ell=\lceil s/2 \rceil}^{s}\frac{z_1+\ell}{z_1+2\ell} \leq \left(\frac{2}{3}\right)^{s/2} \leq \left(\frac{2}{3}\right)^{K/3} \leq 0.9^K.
    \end{equation*}
    \textbf{Case 2:}   $z_1> 1, s\leq 2z_1$. 
    It follows that $z_1 \geq K/3$, since 
    $3 z_{1} = 2 z_{1} + z_{1} \geq s + z_{1} = K$. 
    Thus, 
    \begin{align*}
        \frac{1}{z_1!}\leq \frac{1}{2^{z_1-1}} \leq \frac{1}{2^{z_1/2}} \leq \frac{1}{2^{K/6}} \leq 0.9^K.
    \end{align*}
\textbf{Case 3:} $z_1=1, s=1$. Here we directly have that
\begin{align*}
    \frac{1}{z_1!} \prod_{\ell=1}^{s} \frac{z_1+\ell}{z_1+2\ell} = \frac{2}{3} \leq 0.9^2 = 0.9^K,
\end{align*}
which concludes the proof of Lemma~\ref{lemma:pdembed}.
\end{proof}

\subsection{Root-to-root matchings}\label{sec:roottoroot}

For a uniform attachment tree of size $N$, denoted by $T_N$, and a rooted tree $T$, we write $T\sim^\star  T_N$ if there exists a {graph isomorphism} from $T$ to $T_N$ that also maps the root to the root. This section focuses on bounding the probability
$\p \left(T\sim^\star  T_N \right)$ for any rooted tree $T$ of size $N$.

\begin{lemma}\label{lemma:matchprob}
    Let $T=(V,E)$ be a fixed tree of size $N$ with a distinguished root $\varnothing \in V$.  Then for all $N$ sufficiently large,
    \begin{equation*}
        \p \left( T \sim^{\star} T_N \right) \leq 0.96^N.
    \end{equation*}
    \invisible{ 
    \basc{Based on improvement of Claim~\ref{claim:1}, this upper bound can now be improved to $o(\sqrt{c}^N)$ for any $c>11/2$.}
}
\end{lemma}

Before we go to the proof of this lemma, we introduce some more notation. For the proof, we view both $T$ and $T_N$ as subsets of $V_\infty = \bigcup_{k\geq 0}\N^k$ that respect the natural arrival structure, in the sense that they satisfy properties \ref{prop:UH1} through \ref{prop:UH3}. We think of $T$ as a deterministic tree and $T_N$ as a random element in $\mathcal{T}_N$, as discussed in Section \ref{sec:prelim}.
 For a vertex $\sigma=(\sigma_1,\ldots,\sigma_j)\in \N^j$ and $\mu=(\mu_1,\ldots,\mu_\ell) \in \N^\ell$, we write $\sigma\oplus\mu := (\sigma_1,\ldots,\sigma_j,\mu_1,\ldots,\mu_\ell) \in \N^{j+\ell}$, as in Section \ref{sec:beyondUH}. In particular, for $k\in \N$, we write $\sigma \oplus (k) \coloneqq (\sigma_1,\ldots,\sigma_j,k)$ for its $k$-th child.

For a tree $A \subset V_\infty$ and $v\in A$, write $\chi(v,A) = \{u \in A : u=v\oplus (k) \text{ for some } k\in \N, u \text{ is a leaf}\}$ for the set of children of $v$ in $A$ that are leaves. We write $\{T \to T_N\}$ for the set of graph isomorphisms $\varphi : T \to T_N$ for which $\varphi(\varnothing) = \varnothing$ and for which, for each $u \in T$, the function
\begin{equation*}
    \varphi : \chi(u,T) \to \chi(\varphi(u), T_N)
\end{equation*}
is increasing in the lexicographic ordering. Note that if $\varphi : T \to T_N$ is a graph isomorphism with $\varphi(\varnothing) = \varnothing$, then we always need to have that $\varphi ( \chi(u,T)) = \chi(\varphi(u), T_N)$. However, for the graph isomorphisms in $\{T \to T_N\}$ we require the additional constraint that the function is monotone on the set of leaves. Note that if there exists a graph isomorphism with $\varphi(\varnothing) = \varnothing$, then there also needs to exist an element in $\{T\to T_N\}$, so that Markov's inequality implies that
\begin{align}\label{markovtrick}
    \p (T \sim^\star T_N ) \leq \E \left[ \left| \{ T \to T_N \} \right| \right] .
\end{align}
Next, we prove Lemma \ref{lemma:matchprob}.

\begin{proof}[Proof of Lemma \ref{lemma:matchprob}]
	
	We estimate the expected number of isomorphisms in $\{T \to T_N\}$.  For a vertex $v\in T$, we write $b_T(v)=\left(v_1,\ldots,v_{\od{(v)}}\right)$ for the sizes of subtrees that are attached to $v$ (where $\od{(v)}$ denotes the {number of children} of $v$). In particular, $v_1+\ldots+v_{\od{(v)}}$ is the number of descendants of $v$. In Claim \ref{claim:product} below, we will prove that the expected size of the set $ \{ T \to T_N\}$ is given by
	\begin{align}\label{eq:numembed}
    \E\left[\left| \{ T \to T_N\} \right|\right] =
	\prod_{v\in T_+} \E \left[| \{ b_T(v) \rightarrow w \left(|b_T(v)|_1\right) \}|\right],
	\end{align}
    where we recall that $w(N)$, defined in Definition~\ref{def:subtree_vector}, denotes the subtree vector, that is,\ the sizes of the subtrees rooted at the children of the root of a uniform attachment tree of size $N$. Also, recall that $| \{ b_T(v) \rightarrow w \left(|b_T(v)|_1\right)\}|$ denotes the number of ways to embed the vector $b_T(v)$ into a subtree vector $w(|b_T(v)|_1)$.   It follows from Lemma~\ref{lemma:pdembed} that
	\begin{align*}
	\E \left[|\{ b_T(v) \rightarrow w \left(|b_T(v)|_1\right) \}|\right] 
	\leq  0.9^{\od(v)},
	\end{align*}
	if $v$ has at least two descendants, where $\od(v)$ is the number of children of $v$.  Combining all of the previous statements, we see that the expected number of ways to embed $T$ into $T_N$ is bounded by 
	\begin{equation*}
	\prod_{v\in T_+} \E \left[|\{ b_T(v) \rightarrow w \left(|b_T(v)|_1\right) \}|\right]
	\leq
	\prod_{v\in T_+}0.9^{\od(v)}=0.9^{\sum_{v\in T_+}\od(v)}
	\end{equation*}
	where $T_+$ are all vertices in $T$ with at least two descendants. For any tree $T$ of size $N$, we have that
	\begin{equation*}
	\sum_{v \in T_+} \od (v) \geq \frac12 (N-1),
	\end{equation*}
	which we show in Claim \ref{claim:2} below.
	This implies that the expected number of tree isomorphisms from $T$ to $T_N$ that map the root to the root is, for $N$ large enough, bounded by 
	\begin{align*}
    \prod_{v\in T_+} \E \left[|\{ b_T(v) \rightarrow w \left(|b_T(v)|_1\right) \}|\right]
	\leq 0.9^{\sum_{v\in T_+}\od(v)}
    \leq
	(\sqrt{0.9})^{N-1}\leq 0.96^N,
	\end{align*}
	which concludes the proof of Lemma~\ref{lemma:matchprob}, by \eqref{markovtrick}.
\end{proof}

It remains to prove the two claims made in the proof.

\begin{claim}\label{claim:2}
Let $T = (V,E)$ be a fixed tree of size $N$  and let $T_+$ be the set of vertices in $T$ with at least two descendants. Then, for sufficiently large $N$,
    \begin{align*}
    \sum_{v \in T_+} \mathrm{outdeg} (v) \geq \frac12 (N-1) .
\end{align*}
\end{claim}

\begin{proof}
Define $T_{1}$ to be the set of vertices in $T$ with exactly one descendant. This descendant needs to be a leaf. Thus, for each $v \in T_1$, there exists exactly one leaf $u(v)$, which is the unique descendant of $v$. The collection $(\{v,u(v)\})_{v \in T_1}$ is a disjoint collection of vertices, not containing the root (which has $N-1>1$ descendants). Thus
\begin{equation*}
    N-1 = |T\setminus \varnothing| \geq \Big| \bigcup_{v \in T_1} \{v,u(v)\} \Big| = 2|T_1|, 
\end{equation*}
which implies that $|T_{1}| \leq (N-1)/2$. 
Thus we get that
\begin{align*}
    \sum_{v \in T_+} \mathrm{outdeg} (v) &= \sum_{v \in T} \mathrm{outdeg} (v) - \sum_{v \in T_1} \mathrm{outdeg} (v)
    =
    \sum_{v \in T} \mathrm{outdeg} (v) - |T_1|
    \\
    &\geq
    \sum_{v \in T} \mathrm{outdeg} (v) - \frac{N-1}{2}
    =
    \frac{N-1}{2},
\end{align*}
where we used that the sum of out-degrees equals the number of edges $|E|=|T|-1=N-1$ in the last inequality. 
\end{proof}

\begin{claim}\label{claim:product}
	Let $N \geq 1$ and let $T \subset V_\infty$ be a fixed tree of size $N$.
    Let
    $\left\{ T \rightarrow T_N \right\}$ denote the set of graph isomorphisms defined above.
    The expected number of such isomorphisms is given by
	\begin{align*}
	\E\left[|\left\{ T \rightarrow T_N \right\}|\right] =
	\prod_{u\in T_+} \E \left[| \{ b_T(u) \rightarrow w \left(|b_T(u)|_1\right) \}|\right].
	\end{align*}
\end{claim}

\begin{proof}
    We prove the statement via induction on $N$. The statement is clearly true for $N=1$, that is, when the tree $T$ consists of a single node only. In that case $T_+$ is the empty set and by convention the product over $T_+$ equals $1$. Assume that the statement is true for all trees of size at most $N$. Let $T \subset V_\infty$ be a tree of size $N+1$. Write $K= \od_T(\varnothing) = \left| \left\{ k \in \N : k\in T \right\} \right|$ for the out-degree of $\varnothing$ in $T$. For $1 \leq i \leq K$, write $S_i \coloneqq \left\{ \sigma \in T : \sigma=(i) \oplus \mu \text{ for some } \mu \in V_\infty \right\}$ for the fringe subtree rooted at vertex $(i)$ and $s_i = |S_i|$ for its size. Write $b_T(\varnothing) = (s_1,\ldots,s_K)$ for the vector containing the sizes of the subtrees. 
    Similarly, let $\overline{K}= \od_{T_{N+1}}(\varnothing) = \left| \left\{ k \in \N : k\in T_{N+1} \right\} \right|$ be the degree of the root in $T_{N+1}$, let
    $U_i \coloneqq \left\{ \sigma \in T_{N+1} : \sigma=(i)\oplus \mu \text{ for some } \mu \in V_\infty \right\}$ be the fringe subtree rooted at vertex $(i)$ and $u_i = |U_i|$ its size. Write $b_{T_{N+1}}(\varnothing) = (s_1,\ldots,s_K)$ for the vector containing the sizes of the subtrees. 

Let $\Phi$ be the set of embeddings of $b_{T}(\varnothing)$ into $b_{T_{N+1}}(\varnothing)$. Note that $\Phi$ is non-empty if and only if $b_{T}(\varnothing)$ and $b_{T_{N+1}}(\varnothing)$ are equivalent, which, in particular, implies that $K=\overline{K}$. Furthermore, we write $\{S_i \to U_j\}$ for the set of graph isomorphisms $\varphi : S_i \to U_j$ with $\varphi((i))=(j)$ (i.e., the ``root" of $S_i$ gets mapped to the ``root" of $U_j$) for which, for each $u \in S_i$, the function
\begin{equation*}
    \varphi : \chi(u,S_i) \to \chi(\varphi(u), U_j)
\end{equation*}
is increasing in the lexicographic ordering (i.e., graph isomorphism that respect the order of the leaves). Each graph isomorphism $\psi \in \{T \to T_{N+1}\}$ induces an embedding $\phi \in \Phi$ of $b_T(\varnothing)$ to $b_{T_{N+1}}(\varnothing)$. For such an embedding $\phi$, for each $i \in [K]$, the graph isomorphism $\psi$ also induces a graph isomorphism from $S_i$ to $U_{\phi((i))}$, which respects the usual leaf-ordering. Thus
    \begin{equation*}
    	|\left\{T \rightarrow T_{N+1} \right\}| = \mathbbm{1}_{b_{T}(\varnothing)\sim b_{T_{N+1}}(\varnothing) } \sum_{\phi \in \Phi}  \prod_{i=1}^{K} \left|\left\{S_i \rightarrow U_{\phi((i))} \right\}\right| .
    \end{equation*}
    Let $\left(T_n^i\right)_{i,n\in \N}$ be independent UA trees, which are furthermore independent of $T_{N+1}$, with $T_n^i$ being a UA tree of size $n$. Note that, conditioned on $b_{T_{N+1}}(\varnothing)$, the collection of subtrees $(U_i)_{1\leq i \leq \overline{K}}$ has the same distribution as $(T_{u_i}^i)_{1\leq i \leq \overline{K}}$.
    Let $(S_i)_+$ be the vertices in $S_i$ that have at least $2$ descendants.
   Taking expectations on both sides of the equation above and applying the induction hypothesis yields
    \begin{align*}
     \E \left[\left|\left\{T \rightarrow T_{N+1} \right\}\right|\right] 
    &= \E \left[ \mathbbm{1}_{b_T(\varnothing)\sim b_{T_{N+1}}(\varnothing)}  \E \left[ \sum_{\phi \in \Phi}  \prod_{i=1}^{K} \left|\left\{S_i \rightarrow U_{\phi(i)} \right\}\right| \ \Big| b_{T_{N+1}}(\varnothing) \right] \right]\\
    & = \E \left[ \mathbbm{1}_{b_T(\varnothing)\sim b_{T_{N+1}}(\varnothing)}  \sum_{\phi \in \Phi}  \prod_{i=1}^{K} \Big| \left\{S_i \rightarrow T_{u_{\phi(i)}}^{\phi(i)} \right\} \Big|  \right] \\
    &
    = \E \left[ \mathbbm{1}_{b_T(\varnothing)\sim b_{T_{N+1}}(\varnothing)} \sum_{\phi \in \Phi} 1  \right] \prod_{i=1}^{K} \prod_{u\in (S_i)_+} \E \left[ | \{ b_T(u) \rightarrow w \left(|b_T(u)|_1\right) \}| \right]\\
    & =
    \E \left[ | \{ b_T(\varnothing) \rightarrow b_{T_{N+1}}(\varnothing) \}| \right] \prod_{u \in T_+ \setminus \{ \varnothing \} } \E \left[|\{ b_T(u) \rightarrow w \left(|b_T(u)|_1\right) \} | \right]
    \\
    &=
    \prod_{u\in T_+} \E \left[|\{ b_T(u) \rightarrow w \left(|b_T(u)|_1\right) \} | \right],
    \end{align*}
    as desired. 
\end{proof}

\subsection{Beyond root-to-root matchings}\label{sec:beyondroottoroot}

In this section, we prove Theorem \ref{thm:UB_near1}. The theorem is a direct consequence of the following lemma.
\begin{lemma}\label{lem:upper bound eps}
	Let $\eps=0.0025$. Then $\p \left(X_{N + \lfloor \eps N \rfloor} \geq N \right) \to 0$ as $N\to \infty$.
 \invisible{ 
 \basc{This can be strengthened to: There exists a (random) $N_0\in\N$ such that $X_{(1+\eps)N}<N$ for all $N\geq N_0$ almost surely. We can also improve the bound on $\eps $ to $\eps<0.0327952$.}
 }
\end{lemma}

\noindent To prove this lemma, we need several intermediary statements. For two trees $T$ and $T^\prime$, we write $T \subset T^\prime$ if there exists a graph isomorphism between $T$ and a subtree of $T^\prime$. In this case, we also say that $T$ is contained in $T^\prime$.

\begin{lemma}\label{lem:tree_count}
    Let $T=(V,E)$ be a fixed tree of size $N$. The number of unlabeled trees $T^\prime$ of size $N+n$ that contain $T$ is bounded by
    \begin{align*}
        4^n \binom{N+n}{n}. 
    \end{align*}
\end{lemma}


\begin{proof}
    Every unlabeled tree $T^\prime$ of size $N+n$ can be obtained from $T$ by replacing vertex $v\in V$ by a rooted tree $F_v$ so that $\sum_{v\in V} |F_v|=N+n$. So the number of unlabeled trees $T^\prime$ containing $T$ is bounded from above by
    \begin{equation*}
        \big| \big\{ (F_v)_{v\in V} : F_v \text{ rooted tree}, \sum_{v \in V} |F_v|=N+n \big\} \big|.
    \end{equation*}
    Let $s_v \in \N$ be the size of the tree $F_v$. A stars-and-bars argument shows that 
    \begin{equation*}
        \big|\big\{ (s_v)_{v \in V} \in \N^{|V|} : \sum_{v \in V} s_v = N+n \big\}\big| = \binom{N+n-1}{n},
    \end{equation*} using that $|V| = N$. For each $k \in \N$, the number of rooted trees of size $k$ is bounded by the number of rooted plane trees of size $k$, which is exactly the $k$-th Catalan number $C_k = \frac{1}{k+1} \binom{2k}{k} \leq 4^{k-1}$. Thus we get that
    \begin{multline*}
        \big| \big\{ (F_v)_{v\in V} : F_v \text{ rooted tree}, \sum_{v \in V} |F_v|=n \big\} \big|
        \\
        \leq
        \sum_{(s_v)_v : \sum_v s_v = N + n} \ \prod_{v \in V} 4^{s_v-1}
        \leq
        \big|\big\{(s_v)_v : \sum_{v\in V} s_v = N + n\big\} \big|  4^{n} \leq \binom{N+n}{N} 4^n ,
    \end{multline*}
    finishing the proof.
\end{proof}

\begin{lemma}\label{lemma:embedprob}
	Let $T$ be a {fixed} tree of size $N$. Then, for sufficiently large $N$  and $n \leq \frac{N}{2}$
	\begin{align*}
		\p \left(T \subset T_{N+n} \right) \leq  (N+n) \binom{N+n}{n} 4^n 0.96^{N+n}.
	\end{align*}
\end{lemma}

\begin{proof} 
	If $T\subset T_{N+n}$, then we can grow the tree $T$ to a tree $T^\prime$ that is isomorphic to $T_{N+n}$, and we can choose a vertex $v \in T^\prime$ so that this vertex gets mapped to the root $\varnothing \in T_{N+n}$ by the graph isomorphism. By Lemma~\ref{lem:tree_count}, we obtain an upper bound on the number of such trees $T^\prime$, which gives that
    \begin{align*}
		 \p \left(T \subset T_{N+n} \right) 
        &\leq  \binom{N+n}{n} 4^n \sup \left\{ \p \left(T^\prime \sim T_{N+n} \right)  | T^\prime \text{ tree of size $N+n$} \right\}\\
        &
        \leq   \binom{N+n}{n} 4^n (N+n) \sup \left\{ \p \left(T^\prime \sim^\star T_{N+n} \right)  | T^\prime \text{ rooted tree of size $N+n$} \right\}
        \\
        &\leq    \binom{N+n}{n} 4^n (N+n)0.96^{N+n},
    \end{align*}
    which concludes the proof.
\end{proof}

\noindent We also need the following estimate. 

\begin{claim}\label{claim:3}
	Let $\eps = 0.0025$ and let $n=\lfloor\eps N \rfloor$. Then,
	\begin{align*}
	\lim_{N\to\infty}4^n \binom{N+n}{n}^2  (N+n) 0.96^{N+n}=0.
	\end{align*} 
\end{claim}

\begin{proof}
    By~\eqref{eq:binom_coeff_entropy_bounds} we have that
    \be
    \binom{N+n}{n}=\exp\left( H \left( \frac{\eps}{1+\eps} \right)  (1+\eps)N + o(N) \right) . 
    \ee 
     We thus arrive at 
    \begin{multline*} 
    4^n \binom{N+n}{n}^2  (N+n) 0.96^{N+n} \\ 
    = \exp\left( 2 H\left( \frac{\eps}{1+\eps} \right)  (1+\eps)N + \log(0.96)(1+\eps)N +\log(4)\lfloor \eps N\rfloor+ o(N) \right),
    \end{multline*}
    and inserting the value of $\eps = 0.0025$ shows that the exponent is negative, and thus the entire term converges to $0$. 
\end{proof}

\invisible{
\begin{proof}
	We first study the binomial coefficient $\binom{N+n}{n}$. Let $A$ be a subset of $\{1,\ldots,N+n\}$ that is chosen uniformly at random among all $2^{N+n}$ many subsets. We can construct such a uniformly chosen subset by sampling $N+n$ many independent Bernoulli$\left(\tfrac{1}{2}\right)$ random variables $X_1,\ldots,X_{N+n}$, and including the element $i\in \left\{1,\ldots,N+n\right\}$ to the set $A$ if and only if $X_i=1$. Using this description of a random set $A$, we see that
	\begin{align*}
	\binom{N+n}{\leq n} 
	= 2^{N+n} \p \left(|A|\leq n\right) 
	= 2^{N+n} \p \left( \sum_{i=1}^{N+n} X_i \leq n\right)
	\end{align*}
	Write $\delta =\frac{\eps}{1+\eps} \approx \frac{n}{N+n}$. So we see that
	\begin{align}\label{eq:binomial coeff 1}
	\binom{N+n}{\leq n} 
	= 2^{N+n} \p \left( \sum_{i=1}^{N+n} X_i \leq n\right)
	= 2^{N+n} \p \left( \sum_{i=1}^{N+n} X_i \leq \delta (N+n)\right) .
	\end{align}
	The large deviation rate function of Bernoulli$\left(\tfrac{1}{2}\right)$ distributed random variables is given by $\Lambda(x) = x \log(2x) + (1-x) \log(2(1-x))$. So, using Cramer's Theorem for the large-deviation event in \eqref{eq:binomial coeff 1}, we get that
	\begin{align*}
	&\binom{N+n}{\leq n} \leq 2^{N+n} 
	\exp \left( - (N+n) \Lambda(\delta) +o(N) \right)
	\leq
	2^{(1+\eps) N} 
	\exp \left( - N(1+\eps) \Lambda(\eps) + o(N) \right)
	\end{align*}
	where we used that $\delta \leq \eps$ and thus $\Lambda(\delta) \geq \Lambda(\eps)$ for the last inequality. Using this upper bound on the binomial coefficient, we finally get that
	\begin{align*}
	& 8^n \binom{N+n}{n}^2 (N+n) 0.96^N \leq
	2^{\eps N} \left(2^{(1+\eps) N} 
	\exp \left( - N(1+\eps) \Lambda(\eps) \right) \right)^2 0.96^{N}  \e^{o(N)}
	\\  
	&
	=
	\exp\left( \log(8)\cdot\eps N + 2 \log(2)(1+\eps) N - 2N(1+\eps) \Lambda(\eps)
	+
	N
	\log(0.96) +o(N)
	\right)
	.
	\end{align*}
	To conclude, we use that
	\begin{align*}
	f(\eps)=\log(8)\cdot\eps + 2 \log(2)(1+\eps)  - 2 (1+\eps) \Lambda(\eps)
	+
	\log(0.96) < 0
	\end{align*}
	for $\eps < 0.0025$. This holds as $f$ is increasing between $0$ and $0.5$, and $f(0.0025) =-0.0005... < 0$.
\end{proof}
}

\noindent With this claim and Lemma~\ref{lemma:embedprob} at hand, we can finally prove Lemma \ref{lem:upper bound eps}.

\begin{proof}[Proof of Lemma \ref{lem:upper bound eps}]
    Let $N \geq 1$ and let $n \coloneqq \lfloor \eps N \rfloor$.
    Let  $T^1_{N+n}$ and $T^2_{N+n}$ be two independent uniform attachment trees, each of size $N+n$. Then, 
    \begin{align*}
        \p  \left( X_{N+n} \geq N \right)
        &\leq \p \left( \text{There exists a tree } T \subset T_{N+n}^1 \text{ of size $N$ s.t. }   T  \subset T_{N+n}^2  \right)
        \\
        &
        \leq \sum_T \p \left( T \subset T_{N+n}^1,  T \subset T_{N+n}^2 \right)
        =
        \sum_T \p \left( T \subset T_{N+n}^1 \right) \p \left( T \subset T_{N+n}^2 \right),
    \end{align*}
    where we sum over all trees $T$ of size $N$ and the equality follows by independence. By Lemma~\ref{lemma:embedprob}, for $N$ sufficiently large, the above can be further bounded from above by
    \begin{align*}
        \E \left[ \Big|\left\{  T \subset T_{N+n}^1  :  \text{$T$ is a tree of size $N$} \right\} \Big| \right] 4^n \binom{N+n}{n} (N+n) 0.96^{N+n}.
    \end{align*}
    The tree $T_{N+n}^1$, has at most $\binom{N+n}{N} = \binom{N+n}{n}$ many subtrees of size $N$. Thus,
    \begin{align*}
        \p & \left( X_{N+n} \geq N \right)
        \leq 
        4^n \binom{N+n}{n}^2   (N+n) 0.96^{N+n} \e^{o(N)} ,
    \end{align*}
    which converges to $0$ by Claim \ref{claim:3}.
\end{proof}

\subsection{Proof that $X_{n}/n$ is not bounded away from $0$} \label{subsec:not_bdd_0}

In this subsection we prove Theorem~\ref{thm:UB_near0} by showing that, with positive probability $\delta>0$, the sizes of subtrees near the root in $T_n^1$ and $T_n^2$ show dissimilar behaviour, and hence do not allow for a large common subtree. Specifically, we impose that in $T_n^1$ the root has many children which all have small subtrees (size $\leq \beta n$) attached to it, while in $T_n^2$, up to some error, the tree looks like a line of length of order $\frac{1}{\beta}$ with trees of size $\approx \beta n$ attached to the different vertices in the line.
This dissimilarity forces the  maximal common subtree of the two trees $T_n^1, T_n^2$ to be small, by the following lemma.

\begin{lemma}\label{lem:point removal}
    Let $T^1, T^2$ be two fixed trees. Let $v_1 \in T^1$ and $v_2 \in T^2$ be two vertices. Let $a_1,\ldots,a_j$ be the sizes of the connected components of $T^1$ when removing the vertex $v_1$, in decreasing order. Analogously, let $b_1,\ldots,b_m$ be the sizes of the connected components of $T^2$ when removing the vertex $v_2$, in decreasing order.
    Let $X$ be the maximal size over all rooted trees $T$ that are contained in both $T^1$ and $T^2$, such that the root in $T$ is mapped to $v_1$ in $T^1$, respectively to $v_2$ in $T^2$. Then
    \begin{equation*}
        X\leq 1+\sum_{i=1}^{\min(j,m)} \min(a_i,b_i).
    \end{equation*}
\end{lemma}
\begin{proof}
    When removing the vertex $v_1$ from the tree $T^1$, the tree $T^1$ falls into the components $S_1,\ldots,S_j$ with $|S_i| = a_i$. Analogously, let $U_1,\ldots,U_m$ be the connected components of $T^2\setminus\{v_2\}$ with $|U_i| = b_i$.  Let $T$ be a rooted tree that is contained in both $T^1$ and $T^2$ and let $\psi:T\to T^1$ be an embedding (i.e., an injective map respecting the tree-structure) of the tree $T$ into $T^1$ for which $\psi(\varnothing)=v_1$, with $\varnothing$ the root of $T$. We define the sets $W_i=\psi^{-1} (S_i)$ for $i\in \{1,\ldots,j\}$.  Note that these sets can be empty and that the non-empty sets among $W_1,\ldots,W_j$ are the connected components of $T\setminus\{\varnothing\}$. By the injectivity of $\psi$ we get that $|W_i|\leq |S_i|= a_i$. Let $\widetilde{W}_1,\ldots,\widetilde{W}_\ell$ be all the non-empty sets among $W_1,\ldots,W_j$ in decreasing order of size. As $|W_i|\leq a_i$ and we assume that the $a_1,\ldots,a_j$ are decreasing, we also get that $|\widetilde{W}_i|\leq a_i$ for $i=1,\ldots,\ell$.
    If there exists an embedding of the tree $T$ into $T^2$, that is, an injective map $\varphi:T \to T^2$ with $\varphi(\varnothing)=v_2$ respecting the tree-structure, then for each index $i\in \{1,\ldots,\ell\}$ there exists exactly one index $\Delta(i) \in \{1,\ldots,m\}$ such that $\varphi(\widetilde{W}_i)\subseteq U_{\Delta(i)}$. So the function $\Delta:\{1,\ldots,\ell \} \to \{ 1,\ldots,m\}$ is an injective function. Further, as the function $\varphi$ is also injective, we get that $|\widetilde{W}_i| \leq |U_{\Delta(i)}|$ for $i \in \{1, \ldots,\ell \}$. In particular, we get that
    \begin{equation*}
        |T| = 1 + \sum_{i=1}^{\ell} |\widetilde{W}_i|
        = 1 + \sum_{i=1}^{\ell} \min(|\widetilde{W}_i|, |U_{\Delta(i)}|)
        \leq
        1 + \sum_{i=1}^{\ell} \min(a_i, b_{\Delta(i)}).
    \end{equation*}
    Among all $\ell \in \{1,\ldots,j\}$ and all injective functions $\Delta: \{1,\ldots,\ell \} \to \{1,\ldots, m\}$, the expression $\sum_{i=1}^{\ell} \min(a_i, b_{\Delta(i)})$ is maximised when $\ell = \min(j,m)$ and $\Delta$ is the identity on $\{1,\ldots,\min(j,m)\}$. Therefore, for any tree $T$ that can be embedded into $T^1$ and $T^2$, it holds that
    \begin{equation*}
        |T| 
        \leq
        1 + \sum_{i=1}^{\min(j,m)} \min(a_i, b_{i}).
    \end{equation*}
   The proof is concluded by taking the maximum over all rooted trees $T$ that are contained in both $T^1$ and $T^2$.
\end{proof}

\invisible{

\begin{figure}[h]
\centering
\begin{tikzpicture}
  \draw[thick,blue] (0,0) .. controls (-1,-2) and (-2,-1) .. (-3,-3);

  \draw[thick] (-1.41,-2.1) .. controls (-1.5,-2.6) and (-1.8,-2.5) .. (-2.2,-3);

  \draw[thick,blue] (-1.5,-1.5) .. controls (-1.2,-2.5) and (-1.8,-2.8) .. (-1,-3);

  \draw[thick,blue] (-0.72,-1) .. controls (0,-1) and (-0.3,-2) .. (0.8,-2.5);

  \draw[thick] (1,-1.15) .. controls (1.2,-1.3) and (0.9,-1.5) .. (0.8,-1.4);

  \draw[thick] (0.3,-0.71) .. controls (0.2,-0.6) and (0.4,-0.9) .. (0.3,-1);

  \draw[thick,blue] (0,0) .. controls (-0.3,-1) and (2.5,-1.5) .. (2,-2);

  \draw[thick] (-0.5,-2.4) .. controls (-0.7,-2.6) and (-0.8,-2.8) .. (-0.9,-2.8);
  
  \draw[thick,blue] (0,-1.75) .. controls (0,-1.9) and (-0.3,-2.2) .. (-0.5,-2.4);
  
  \fill[black] (0,0) circle (2pt);
  \node at (0,0.3) {$\varnothing$};

  \fill[red] (-0.72,-1) circle (2pt);
  \node at (-0.72,-0.7) {$v_1$};

  \node at (-2,-0.3) {$T^1$};


  \draw[thick,blue] (6,0) .. controls (5,-2) and (4,-1) .. (3,-3);

  \draw[thick,blue] (4.5,-1.5) .. controls (4.8,-2.5) and (4.2,-2.8) .. (5,-3);


  \draw[thick,blue] (5.28,-1) .. controls (6,-1) and (5.7,-2) .. (6.8,-2.5);

  \draw[thick,blue] (6,0) .. controls (5.7,-1) and (8.5,-1.5) .. (8,-2);



  \draw[thick,blue] (6,-1.75) .. controls (6,-1.9) and (5.7,-2.2) .. (5.5,-2.4);

  \draw[thick] (6,0) .. controls (6.5,-0.6) and (7.5,-0.8) .. (7.5,-1);

  \draw[thick] (7.4,-1.35) .. controls (7,-1.7) and (8,-2.2) .. (7.5,-2.8);

  \draw[thick] (3.35,-2.4) .. controls (4,-2.5) .. (4.5,-3);
  
  \fill[black] (6,0) circle (2pt);
  \node at (6,0.3) {$\varnothing$};

  \fill[red] (5.28,-1) circle (2pt);
  \node at (5.28,-0.7) {$v_2$};

  \node at (4,-0.3) {$T^2$};
\end{tikzpicture}

\caption{Two (schematic representations of the) trees $T^1$ and $T^2$ with two distinguished vertices $v_1$ and $v_2$ in red. The blue parts of either tree represent the smallest parts of each component when $v_i$ is removed from $T^i$ that can be found in both $T^1$ and $T^2$.}\label{fig:treecomppsize}
\end{figure}

}

\invisible{
\begin{figure}[h]
\centering

\begin{tikzpicture}
\fill[black] (0,0) circle (3pt);
\node at (0,0.27) {$v_1$};

\fill[black] (1.35,-0.7794) circle (3pt);     
\fill[black] (0.45,-0.7794) circle (3pt); 
\fill[black] (-0.45,-0.7794) circle (3pt);
\fill[black] (-1.35,-0.7794) circle (3pt);

\draw[thick] (0,0) -- (1.35,-0.7794);
\draw[thick] (0,0) -- (0.45,-0.7794);
\draw[thick] (0,0) -- (-0.45,-0.7794);
\draw[thick] (0,0) -- (-1.35,-0.7794);

\draw[thick] (-1.35,-0.7794) .. controls (-2.7,-1.53) and (-0.9,-3.69) .. (-1.35,-0.7794);

\draw[thick] (-0.45,-0.7794) .. controls (-1.62,-3.5874) and (0.54,-4.3866) .. (-0.45,-0.7794);

\draw[thick] (0.45,-0.7794) .. controls (-0.45,-5.4) and (1.8,-3.6) .. (0.45,-0.7794);

\draw[thick] (1.35,-0.7794) .. controls (0.45,-3.6) and (2.25,-3.6) .. (1.35,-0.7794);

\node at (-2.25,-0.27) {$T^1$};


\fill[black] (5.4,0) circle (3pt);
\node at (5.4,0.27) {$v_2$};

\fill[black] (6.3,-0.7794) circle (3pt);     
\fill[black] (5.4,-0.7794) circle (3pt); 
\fill[black] (4.5,-0.7794) circle (3pt);

\draw[thick] (5.4,0) -- (6.3,-0.7794);
\draw[thick] (5.4,0) -- (5.4,-0.7794);
\draw[thick] (5.4,0) -- (4.5,-0.7794);

\draw[thick] (4.5,-0.7794) .. controls (3.15,-1.7874) and (5.58,-2.5866) .. (4.5,-0.7794);

\draw[thick] (5.4,-0.7794) .. controls (3.87,-4.77) and (7.02,-4.77) .. (5.4,-0.7794);

\draw[thick] (6.3,-0.7794) .. controls (5.85,-2.25) and (6.75,-2.25) .. (6.3,-0.7794);

\node at (7.2,-0.27) {$T^2$};


{\color{blue}
\fill[blue] (0,-4.5) circle (3pt);
\node at (0,-4.23) {$v_1$};

\fill[blue] (1.35,-5.2794) circle (3pt);     
\fill[blue] (0.45,-5.2794) circle (3pt); 
\fill[blue] (-0.45,-5.2794) circle (3pt);
\fill[blue] (-1.35,-5.2794) circle (3pt);

\draw[thick] (0,-4.5) -- (1.35,-5.2794);
\draw[thick] (0,-4.5) -- (0.45,-5.2794);
\draw[thick] (0,-4.5) -- (-0.45,-5.2794);
\draw[thick] (0,-4.5) -- (-1.35,-5.2794);

\draw[thick] (1.35,-5.2794) .. controls (0,-6.03) and (1.8,-8.19) .. (1.35,-5.2794);

\draw[thick] (-0.45,-5.2794) .. controls (-1.62,-8.0874) and (0.54,-8.8866) .. (-0.45,-5.2794);

\draw[thick] (-1.35,-5.2794) .. controls (-2.25,-9.9) and (0,-8.1) .. (-1.35,-5.2794);

\draw[thick] (0.45,-5.2794) .. controls (-0.45,-8.1) and (1.35,-8.1) .. (0.45,-5.2794);
}
\node at (-2.25,-4.77) {$T^1$};


{\color{red} 
\fill[red] (5.4,-4.5) circle (3pt);
\node at (5.4,-4.23) {$v_2$};

\fill[red] (6.3,-5.2794) circle (3pt);     
\fill[red] (5.4,-5.2794) circle (3pt); 
\fill[red] (4.5,-5.2794) circle (3pt);

\draw[thick] (5.4,-4.5) -- (6.3,-5.2794);
\draw[thick] (5.4,-4.5) -- (5.4,-5.2794);
\draw[thick] (5.4,-4.5) -- (4.5,-5.2794);

\draw[thick] (5.4,-5.2794) .. controls (4.05,-6.2874) and (6.48,-7.0866) .. (5.4,-5.2794);

\draw[thick] (4.5,-5.2794) .. controls (2.97,-9.27) and (6.12,-9.27) .. (4.5,-5.2794);

\draw[thick] (6.3,-5.2794) .. controls (5.85,-6.75) and (6.75,-6.75) .. (6.3,-5.2794);
}
\node at (7.2,-4.77) {$T^2$};


{\color{blue}
\draw[thick] (4.05,-9.7794) .. controls (2.7,-10.53) and (4.5,-12.69) .. (4.05,-9.7794);

\draw[thick] (2.25,-9.7794) .. controls (1.08,-12.5874) and (3.24,-13.3866) .. (2.25,-9.7794);

\draw[thick, fill=green] (1.35,-9.7794) .. controls (0.45,-14.4) and (2.7,-12.6) .. (1.35,-9.7794);

\draw[thick] (3.15,-9.7794) .. controls (2.25,-12.6) and (4.05,-12.6) .. (3.15,-9.7794);
}


{\color{red}
\draw[thick] (2.25,-9.7794) .. controls (0.9,-10.7874) and (3.33,-11.5866) .. (2.25,-9.7794);

\draw[thick] (1.35,-9.7794) .. controls (-0.09,-13.77) and (2.97,-13.77) .. (1.35,-9.7794);

\draw[thick,fill=green] (3.15,-9.7794) .. controls (2.7,-11.25) and (3.6,-11.25) .. (3.15,-9.7794);
}

\fill[black] (2.7,-9) circle (3pt);

\fill[black] (4.05,-9.7794) circle (3pt);     
\fill[black] (3.15,-9.7794) circle (3pt); 
\fill[black] (2.25,-9.7794) circle (3pt);
\fill[black] (1.35,-9.7794) circle (3pt);

\draw[thick] (2.7,-9) -- (4.05,-9.7794);
\draw[thick] (2.7,-9) -- (3.15,-9.7794);
\draw[thick] (2.7,-9) -- (2.25,-9.7794);
\draw[thick] (2.7,-9) -- (1.35,-9.7794);

\end{tikzpicture}
\end{figure}

}

\noindent With this lemma at hand, we are ready to prove Theorem \ref{thm:UB_near0}.

\begin{proof}[Proof of Theorem~\ref{thm:UB_near0}]
    Given $\eps>0$, let $\beta = \eps/4$ and let $\delta>0$ be small enough so that $(\beta/3)^{2/\beta} \geq \sqrt{\delta}$.
    We view $T^1=T^1_n$ and $T^2=T^2_n$ as random subtrees of the Ulam--Harris tree as described in Section~\ref{sec:UlamHarris}. For $i \in \{1,2\}$, we define the {\sl weight} $w_i(\sigma)$ of a vertex $\sigma=(\sigma_1, \ldots, \sigma_\ell)\in V_\infty$, as the proportion of vertices in $T^i$ that are in the subtree rooted at $\sigma$, that is,
    \begin{equation*}
        w_i(\sigma)= \frac{1}{n}\cdot |\{ \tau \in V(T^i) \colon \tau = \sigma\oplus \mu \text{ for some } \mu \in \N^\ell, \ell \in \N_0\}| .
    \end{equation*}
    Further, we say that a vertex $\sigma$ has weight $0$ if $\sigma \notin V(T^i)$.  Given the weights of a vertex $\sigma=(\sigma_1,\ldots,\sigma_j)$ and its first $k-1$ children $\{\sigma\oplus (i):1\leq i \leq k-1\}$ for some $k \in \N$, the normalized weight $n\cdot w(\sigma \oplus (k))$ of its $k$-th child $\sigma \oplus (k)$ is uniformly distributed on $\{1,\ldots,nw(\sigma)-1-n\sum_{i \in [k-1]}w(\sigma\oplus (i))\}$, see Section~\ref{sec:offspring distribution}.

\begin{figure}[h]
\centering
\begin{tikzpicture}
  \fill[black] (0,0) circle (3pt);
  \node at (0,0.3) {$\varnothing$};
  \fill[black] (2.0,-0.866) circle (3pt); 
  \fill[black] (1.0,-0.866) circle (3pt);
  \fill[black] (0,-0.866) circle (3pt);
  \fill[black] (-1.0,-0.866) circle (3pt);
  \fill[black] (-2.0,-0.866) circle (3pt);
  \fill[black] (2.5,-0.866) circle (1pt);
  \fill[black] (2.7,-0.866) circle (1pt);
  \fill[black] (2.9,-0.866) circle (1pt);
  \fill[black] (3.1,-0.866) circle (1pt);
  
  \draw[thick] (0,0) -- (2.0,-0.866);
  \draw[thick] (0,0) -- (1.0,-0.866);
  \draw[thick] (0,0) -- (0,-0.866);
  \draw[thick] (0,0) -- (-1.0,-0.866);
  \draw[thick] (0,0) -- (-2.0,-0.866);
  
  \draw[thick] (-2.0,-0.866) .. controls (-5.466,-1.7) and (-3.866,-4.1) .. (-2.0,-0.866);
\draw[thick] (-1.0,-0.866) .. controls (-3.2,-3.986) and (-0.32,-4.874) .. (-1.0,-0.866);
\draw[thick] (0,-0.866) .. controls (-1,-6) and (1.5,-4) .. (0,-0.866);
\draw[thick] (1.0,-0.866) .. controls (0.8,-4) and (3.0,-3) .. (1.0,-0.866);
\draw[thick] (2.0,-0.866) .. controls (3.1,-3.5) and (3.8,-0.85) .. (2.0,-0.866);
  \node at (-3.0,-1.6) {\small $\leq \beta $};
\node at (-1.3,-2.3) {\small $\leq \beta $};
\node at (0.1,-2.5) {\small $\leq \beta $};
\node at (1.4,-2.1) {\small $\leq \beta $};
\node at (2.6,-1.4) {\small $\leq \beta $};
\node at (-2.3,-0.6) {$(1)$};
\node at (-1.35,-0.866) {$(2)$};
\node at (-0.4,-0.866) {$(3)$};
\node at (0.5,-0.866) {$(4)$};
\node at (1.45,-0.866) {$(5)$};
\node at (-3.6,-1) {$T^1$};
  \fill[black] (8,0) circle (3pt);
  \node at (8,0.3) {$\varnothing$};
  \fill[black] (7.2,-0.8) circle (3pt); 
  \fill[black] (6.4,-1.6) circle (3pt);
  \fill[black] (5.6,-2.4) circle (3pt);
  \fill[black] (4.8,-3.2) circle (3pt);
  \fill[black] (4,-4) circle (3pt);
  \draw[thick] (8,0) -- (4,-4);
  \draw[thick] (8,0).. controls (8.6,-3.134) and (11,-2.134) .. (8,0);  
  \draw[thick] (7.2,-0.8).. controls (7.8,-3.934) and (10.2,-2.934) .. (7.2,-0.8);
  \draw[thick] (6.4,-1.6) .. controls (7,-4.734) and (9.4,-3.734) .. (6.4,-1.6);
  \draw[thick] (5.6,-2.4) .. controls (6.2,-5.534) and (8.8,-4.534) .. (5.6,-2.4);
  \draw[thick] (4.8,-3.2) .. controls (5.4,-6.334) and (7.8,-5.334) .. (4.8,-3.2);
  \draw[thick] (4,-4) .. controls (4.6,-7.134) and (7,-6.134) .. (4,-4);
  \node at (8.85,-1.25) {$\leq \beta$};
  \node at (8.05,-2.05) {$\leq \beta$};
  \node at (7.25,-2.85) {$\leq \beta$};
  \node at (6.45,-3.65) {$\leq \beta$};
  \node at (5.65,-4.45) {$\leq \beta$};
  \node at (4.85,-5.25) {$\leq \beta$};
 \node at (5,-1) {$T^2$};
  \node[anchor=east] at (7.1,-0.8)  {$(1)$};
  \node[anchor=east] at (6.3,-1.6)  {$(1,1)$};
  \node[anchor=east] at (5.5,-2.4)  {$(1,1,1)$};
  \node[anchor=east] at (4.7,-3.2)  {$(1,1,1,1)$};
  \node[anchor=east] at (3.9,-4)    {$(1,1,1,1,1)$};
  \fill[black] (3.8,-4.2) circle (1pt);
  \fill[black] (3.6,-4.4) circle (1pt);
  \fill[black] (3.4,-4.6) circle (1pt);
  \fill[black] (3.2,-4.8) circle (1pt);
\end{tikzpicture}
\caption{Properties~\eqref{cond:T1} and \eqref{cond:T2} that hold for the trees $T^1$ and $T^2$, respectively. In $T^1$, each child of $\varnothing$ has a total weight at most $\beta$.  In $T^2$, for each vertex on the path $(\varnothing,(1),(1,1),\ldots , (1,1,1,1,1), \ldots)$, all its siblings $\{(1,\ldots,1,k)\}_{k\geq 2}$ have a combined total weight of at most $\beta$. Note that the dots do not imply 
that, in $T^1$, the root has an infinite degree, or, in $T^2$, that the path $(\varnothing,(1),(1,1),(1,1,1),(1,1,1,1),(1,1,1,1,1),\ldots)$ is infinitely long.}\label{fig:subtreeweight}
\end{figure}
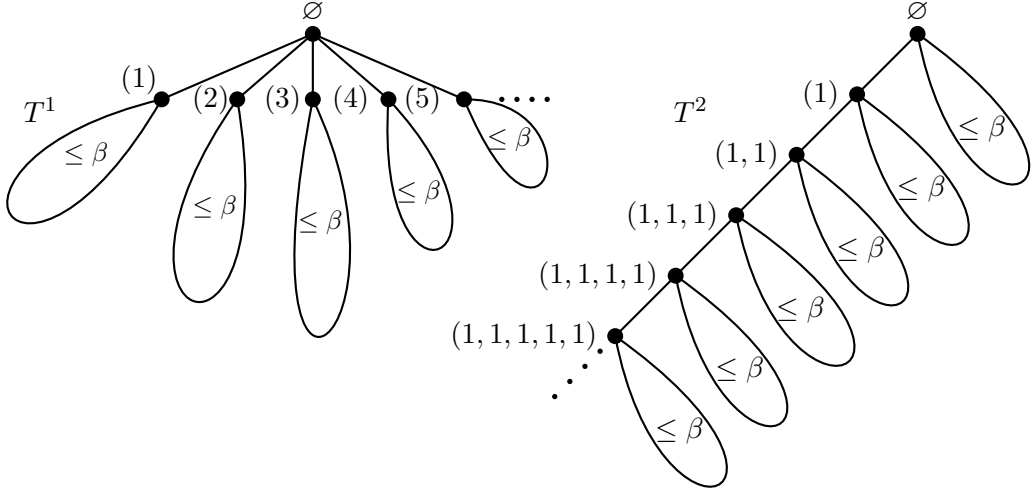
     
    We argue that the following conditions on the weights of vertices in the first $N$ levels of the trees $T^1=T^1_n$ and $T^2=T^2_n$ are met with probability at least $\delta$. 
    \begin{enumerate}[(a)]
        \item In $T^1$, all children of the root have weight at most $\beta$: for every $m \in \N \cap V(T^1) $, $w_1((m))\leq\beta$. \label{cond:T1}
        \item In $T^2$, any vertex $\sigma$ of the form $\sigma=(1,\ldots,1)$ or  $\sigma = \varnothing$, satisfies $\sum_{k=2}^{\infty} w_2(\sigma\oplus (k)) \leq \beta$.  \label{cond:T2}
    \end{enumerate}
    These two properties are visualised in Figure~\ref{fig:subtreeweight}. Using the fact that $\sum_{j= 1}^{\infty}w((j))\leq 1$, note that  property~\eqref{cond:T1} holds if for all $m \in \N$ with $\sum_{j=1}^{m-1} w_1({(j)}) < 1-\beta$ we have $w_1({(m)})\in [\beta/2, \beta]$. Indeed, if $M$ is the first value for which $\sum_{j=1}^{M} w_1({(j)}) \geq 1-\beta$, then, for $\ell \geq M+1$ we have that $w_1({(\ell)}) \leq 1 - \sum_{j=1}^{M} w_1({(j)}) \leq \beta$.
    Let $A_m$ be the event defined by
    \begin{equation*}
        A_m \coloneqq \Big\{ \sum_{j=1}^{m-1} w_1((j)) \geq 1- \beta \Big\} \cup \Big\{ w_1((m)) \in \left[\beta/2,\beta\right] \Big\} .
    \end{equation*}
    First of all, note that if all events $A_m$ for $1\leq m \leq \frac{2}{\beta}$, hold, then we also get that all the events $A_m$ for  $m > \frac{2}{\beta}$ hold, since there needs to exist $j^\prime \leq \frac{2}{\beta}$ with $\sum_{j=1}^{j^\prime} w_1((j)) \geq 1-\beta$. Further, if $w_1((1)),\ldots, w_1((m)) \in (0,1)$ are so that $\sum_{j=1}^m w_1((j)) \geq 1-\beta$, then $A_{m+1}$ occurs with probability $1$. If $w_1((1)),\ldots, w_1((m)) \in (0,1)$ are so that $\sum_{j=1}^m w_1((j)) < 1-\beta$, then, conditionally on $w_1((1)),\ldots, w_1((m))$, we have that $w_1((m+1))$ is uniformly distributed on $\{1,\ldots,n - 1 - n\sum_{j=1}^m w_1((j))\}$. Hence, for large enough $n$, 
    \begin{equation*}
        \p \left(w_1((m+1)) \in \left[\beta/2,\beta\right] \mid w_1((1)),\ldots, w_1((m)) \right) \geq \frac{\lfloor \beta n \rfloor - \lceil n \beta/2 \rceil}{n - 1 - n\sum_{j=1}^m w_1((j))} \geq \beta / 3.
    \end{equation*}
    So we see that for all $w_1((1)),\ldots, w_1((m)) \in (0,1)$, we have that $\p \left( A_{m+1} \mid w_1((1)),\ldots, w_1((m)) \right) \geq \beta/3$.
    Using this idea inductively, we see that 
    \begin{multline*}
        \p \Big( \bigcap_{m=1}^{\infty} A_{m} \Big) = 
        \p \Big( \bigcap_{m=1}^{\lfloor 2/\beta \rfloor} A_{m} \Big)
        =
        \prod_{m=1}^{\lfloor 2/\beta \rfloor}
        \p \Big(  A_{m} \mid \bigcap_{j=1}^{m-1} A_j \Big)
        \\
        =
        \prod_{m=1}^{\lfloor 2/\beta \rfloor}
        \frac{\E \left[ \p \left( A_{m} \mid w_1((1)),\ldots, w_1((m-1)) \right)  \mathbbm{1}_{\bigcap_{j=1}^{m-1} A_j}\right]}{\p \left( \bigcap_{j=1}^{m-1} A_j \right)}
        \geq
         \prod_{m=1}^{\lfloor 2/\beta \rfloor}
        \frac{\E \left[ (\beta/3)  \mathbbm{1}_{\bigcap_{j=1}^{m-1} A_j}\right]}{\p \left( \bigcap_{j=1}^{m-1} A_j \right)}
        \geq
        (\beta/3)^{2/\beta}.
    \end{multline*}
    The previous argument thus implies that property (\ref{cond:T1}) holds with probability at least $(\beta/3)^{2/\beta}$.

    For $k\in \N_0$, let $1^{(k)}$ be the vertex $(1,\ldots,1)$ of length $k$ (with $1^{(0)} = \varnothing$ the root of the Ulam--Harris tree). Property~\eqref{cond:T2} holds if $\sum_{\ell=2}^{\infty} w_2( 1^{(m)} \oplus ( \ell ) ) \in \left[ \beta/2,\beta \right]$ for all $m \in \N$ with  $w_{2}(1^{(m)}) \geq \beta$.  A similar argument as above shows that this happens, for $n$ large enough, with probability at least $(\beta/3)^{2/\beta}$.
    As $T^1$ and $T^2$ are independent,~\eqref{cond:T1} and~\eqref{cond:T2} hold at the same time with probability at least~$\left((\beta/3)^{2/\beta}\right)^2 \geq \delta$.

    \medskip

Assume now that  $T^1$ and $T^2$  satisfy these two properties~\eqref{cond:T1} and~\eqref{cond:T2}. We argue that this assumption gives the two trees such different weight structures that no subtree of size at least $\eps n$ can be embedded into both trees.  
Let $f$ be a graph isomorphism from a subtree $S^1$ of $T^1$ to a subtree of $T^2$. We distinguish the cases $\varnothing \notin S^1$ and $\varnothing \in S^1$.

\textbf{Case 1:} $\varnothing \notin S^1$.
\\
In this case, $S^1$ has to be contained in a subtree that is rooted at one of the children of $\varnothing$. Since all subtrees of children of $\varnothing$ have size at most $\beta n$ by Property~\eqref{cond:T1}, this shows that $S^1$ has size at most $\beta n \leq \eps n$.

\textbf{Case 2:} $\varnothing \in S^1$.
\\ 
Let $f(\varnothing) \eqqcolon v \in T^2$. If we remove $v$ from  $T^2$ and the tree decomposes into connected components of sizes $b_1\geq\ldots\geq b_m$ for some $m \in \N$, then we claim that, $\sum_{i=3}^{m} b_i \leq \beta n$. To show this holds, we divide into two further cases. First, suppose  $v= 1^{(k)}$ for some $k\in\N_0$. When removing $v$ from $T^2$, the tree decomposes into the connected components containing $1^{(k-1)}$ (for $k>0$), $1^{(k+1)}$, and $1^{(k)}\oplus (2), 1^{(k)}\oplus(3),\ldots, 1^{(k)}\oplus(m-1)$. The forest containing $1^{(k)}\oplus (2), 1^{(k)}\oplus(3),\ldots, 1^{(k)}\oplus(m-1)$ has a total size of at most $\beta n$ by {Property}~\eqref{cond:T2}. Hence the sum of the sizes of the $m-2$ smallest components is at most   $\beta n$, too.

Second, assume that $v\neq 1^{(k)}$ for all $k \in \N_0$. Say that $v=(v_1,\ldots,v_j)$ and say that $v_1=\ldots=v_k=1$ and $v_{k+1} \neq 1$. In this case, the connected component containing the parent of $v$ can potentially be large, but the forest containing all other connected components has size at most $n\sum_{\ell=2}^{\infty} w_2(1^{(k)}\oplus(\ell)) \leq \beta n$.

Using this intermediate result, let $a_1,\ldots,a_j$ be the sizes of the connected components of $T^1$ when removing the root $\varnothing$, in decreasing order. Using Lemma~\ref{lem:point removal}, we thus get that
\begin{equation*}
    |S^1| \leq 1 + \sum_{i=1}^{\min(j,m)} \min(a_i,b_i) \leq 
    1+a_1+a_2+\sum_{i=3}^{m} b_i
    \leq 1 +3 \beta n \leq 4 \beta n=\eps n,
\end{equation*}
where $a_1,a_2\leq \beta n$ holds by property~\eqref{cond:T1} and the last inequality holds for $n$ large enough. 

In total, we see that, conditionally on~\eqref{cond:T1} and \eqref{cond:T2}, in Case 1 as well in Case 2, any tree $S^1$ that is contained in both $T^1$ and $T^2$ can have size of at most $4 \beta n \leq \eps n$, which implies that $\p(X_n \leq \eps n) \geq \delta > 0$ for all large enough~$n$, and concludes the proof.
\end{proof}

\invisible{

\underline{Case 2:} $r \in T_1'$ and the Ulam--Harris label $\sigma$ of $f(r)$ in $T_2$ satisfies $|\sigma|_0<N$, $\sigma_i \in\{1,2\}$ for all $i\in[|\sigma|_0]$ and $w_2(\sigma)\geq 2 \beta$. 
\\
In this case, by condition~\eqref{cond:T2}, the vertex $f(r)$ has $2$ children $v_1,v_2 \in V(T_2)$ such that 
\begin{align}\label{eq:v1v2}
w(v_1)+w(v_2)\geq w(f(r))-2 \beta.
\end{align}We now bound the size of the subtrees rooted at each of the children in $T_1'$: There may be children of $r$ that $f$ respectively maps to $v_1$, $v_2$, and the parent vertex of $f(r)$. Since, by condition~\eqref{cond:T1}, these children of $r$ all have weight at most $\beta$ in $T_1$, the sizes of the subtrees rooted at these vertices add up to at most $3 \beta n$. The remaining children of $r$ in $T_1'$ are mapped to children of $f(r)$ that are not $v_1$ or $v_2$. Due to \eqref{eq:v1v2}, these children have cumulated weight at most $2\beta$ in $T_2$. Putting everything together, we see that the size of $T_1'$ is at most $5 \beta n\leq \eps n$.

\underline{Case 3:} $r \in T_1'$, but Case 2 does not hold.
\\
We show that in this case the weight of $\sigma=f(r)$ in $T_2$ is at most $2 \beta$. This is immediate if $\sigma$ fails the above condition because of $w_2(\sigma)<2\beta$. In the first $N-1$ generations, labels $\sigma'$ which have some index $i$ with $\sigma'_i \notin\{1,2\}$, have weight at most $2 \beta$ since, due to~\eqref{cond:T2}, in generation $i$ all but $2 \beta$ of the weight available was split between the first two children of $\sigma_1 \dots \sigma_{i-1}$.

It remains to deal with (descendants of) $\sigma$ with $\sigma_i \in \{1,2\}$ and $|\sigma|_0=N-1$. By condition~\eqref{cond:T2} and since the root of $T_2$ has weight $1$, all descendants $\sigma=\sigma_1\ldots \sigma_\ell$ with $\ell<N$ and $\sigma\in\{1,2\}$ for all $i\in[\ell]$ satisfy that their weight is at most $\max\{2\beta, (1/2)^\ell\}$. By the choice of $N$, it follows that this maximum is at most $2\beta$ when $\ell=N-1$. 

Given that $w_2(f(r))\leq 2 \beta$, an argument similar to that from Case 2 gives that the size of $T_1'$ is at most $2 \beta n\leq \eps n$. 

}

\section{General attachment rules} \label{sec:general} 

So far our focus has been on uniform attachment trees, but some of the ideas and techniques developed in the previous sections also yield results for general models of randomly growing trees. 
In particular, in this section we study general models of preferential attachment trees, where the probability that the incoming vertex attaches to a particular existing vertex is proportional to an (essentially) arbitrary function of its degree. 

Formally, we study a general model of randomly growing trees $\left\{ T_{n} \right\}_{n \geq 2}$ defined as follows. Let $f : \mathbb{N} \to (0,\infty)$ be a function, which we refer to as the \emph{attachment function} or \emph{attachment rule}. Let~$T_{2}$ be the unique tree on two vertices. Given $T_{n}$ with vertices $v_{1}, \ldots, v_{n}$, the tree $T_{n+1}$ is formed from~$T_{n}$ by adding a new vertex $v_{n+1}$ and connecting it with an edge to an existing vertex $u$, chosen according to the following probability distribution: 
\begin{equation}\label{eq:general_attachment}
\p \left( u = v_{i} \, \middle| \, T_{n} \right) 
= \frac{f(d_{T_{n}}(v_{i}))}{\sum_{j=1}^{n} f(d_{T_{n}}(v_{j}))},
\end{equation}
where $d_{T_{n}}(v)$ denotes the degree of vertex $v$ in the tree $T_{n}$. 

\emph{Uniform attachment} corresponds to the special case when $f$ is a constant function, such as~$f \equiv 1$. 
\emph{Preferential attachment} often refers to the model where $f(n) = n$, which is very well studied~\cite{Mah92,BA99,BRST01}. 
Further special cases that are worth noting include 
\emph{linear} (or \emph{affine}) \emph{preferential attachment}, where $f(n) = n + a$ for some $a > -1$; 
\emph{sublinear preferential attachment}, where $f(n) = n^{\alpha}$ for some $\alpha \in (0,1)$; 
and \emph{superlinear preferential attachment}, where $f(n) = n^{\alpha}$ for some $\alpha > 1$. 
The general model given by~\eqref{eq:general_attachment} has also been studied extensively~\cite{RTV07}. 

We are interested in the largest common subtree of two independent trees generated via the model above. We will again denote by $X_{n}$ the size of the largest common subtree when both trees have $n$ vertices (we suppress the dependence on $f$ in the notation for simplicity). 
In Section~\ref{sec:linear} we prove Theorem~\ref{thm:general_lb} by adapting the argument of Section~\ref{sec:polyLB}, 
giving a simple polynomial lower bound on $X_{n}$ for linear and sublinear preferential attachment models. We also provide an alternative lower bound based on the maximum degree. 
Finally, in Section~\ref{sec:superlinear} we show that $X_{n}$ is linear in $n$ (and, in fact, $X_{n} = (1-o(1))n$ with high probability) for superlinear preferential attachment trees, by building on the detailed results of Oliveira and Spencer~\cite{oliveira2005connectivity} for this~model.

\subsection{A polynomial lower bound for linear and sublinear attachment}\label{sec:linear}

We next prove Theorem~\ref{thm:general_lb}. 
We show this using the same methods as in Section~\ref{sec:polyLB}, proving recursive bounds on time intervals that are growing exponentially in length. This underlines that this simple technique is robust in a variety of different settings.

\begin{proof}[Proof of Theorem~\ref{thm:general_lb}]
Multiplying the attachment function $f$ by a positive constant does not change the process. Thus we may assume that $f(n)\leq n$ and that $f(n)>c$ for all $n\in \N$ and some constant $c \in (0,1)$.

Let $\left\{ T_{n}^{1} \right\}_{n \geq 1}$ and $\left\{ T_{n}^{2} \right\}_{n \geq 1}$ denote the two randomly growing trees. We will show that the following inequality holds for every $k \geq 1$: 
\begin{equation}\label{eq:recursive_bound_general}
\E \left[  X_{2^{k+1}} \, \middle| \, T_{2^{k}}^{1}, T_{2^{k}}^{2} \right] \geq \left\{ 1 + \left(1 - \e^{-c/4} \right)^{2} \right\} X_{2^{k}}.
\end{equation}
Then, as in the proof of Claim~\ref{claim:simple}, 
the result follows by taking expectations in~\eqref{eq:recursive_bound_general} 
and applying the resulting inequality recursively (as in~\eqref{eq:recursion_simple}).

Let $v$ denote a vertex in $T_{2^{k}}^{1}$ 
and let us consider the evolution of the first tree from time $2^{k} + 1$ to time $2^{k+1}$. 
By the arguments of Claim~\ref{claim:simple}, 
in order to show~\eqref{eq:recursive_bound_general}, 
it suffices to argue that 
the probability that no new edge attaches to $v$ during this time interval is bounded from above by~$\e^{-c/4}$. 

To see that this holds, 
let $j \in \{ 2^{k}+1, \ldots, 2^{k+1} \}$ and observe that, regardless of the realization of~$T_{j}^{1}$, we have~that 
\[
\sum_{u \in T_{j}^{1}} f( d_{T_{j}^{1}}(u)) 
\leq \sum_{u \in T_{j}^{1}} d_{T_{j}^{1}} (u) 
= 2(j-1) 
\leq 2^{k+2},
\]
where the first inequality uses the assumption that $f(n) \leq n$, 
the equality uses the fact that $T_{j}^{1}$ has $j$ nodes and $j-1$ edges, and the last inequality uses the upper bound on the time interval. 
Since we also have that $f(n) > c$ for all $n$, 
the probability that the vertex joining $T_{j}^{1}$ attaches to $v$ is bounded from below by $c/ 2^{k+2}$, 
regardless of the realization of $T_{j}^{1}$. 
This implies that the probability that no new edge attaches to $v$ from time $2^{k} + 1$ to time $2^{k+1}$ is at most 
$(1 - c / 2^{k+2})^{2^{k}} \leq \e^{-c/4}$. 
\end{proof}

The following lemma provides an alternative argument that can provide useful lower bounds in certain cases.

\begin{lemma}[Maximum degree lower bound]\label{lem:max_degree_lb}
Let $T^{1}$ and $T^{2}$ be two (unlabeled) trees and let $X$ denote the size of their largest common subtree. We have that 
\[
X \geq \min \left\{ \Delta(T^{1}), \Delta(T^{2}) \right\} + 1,
\]
where $\Delta(T)$ denotes the maximum degree of tree $T$. 
\end{lemma}
\begin{proof}
The tree $T^{1}$ contains, as a subtree, a star of size $\Delta(T^{1})+1$ centered at a maximum degree vertex. The analogous statement holds for $T^{2}$ too. Thus both trees contain, as a subtree, a star of size $\min \left\{ \Delta(T^{1}), \Delta(T^{2}) \right\} + 1$. 
\end{proof}

\begin{corollary}[Affine preferential attachment]\label{cor:affine}
Let $f(n) = n + a$, where $a > -1$. 
Let $\left\{ \delta_{n} \right\}_{n \geq 1}$ be any sequence such that $\delta_{n} \to 0$ as $n \to \infty$. 
Then, with high probability, $X_{n} \geq \delta_{n} n^{1/(2+a)}$.
\end{corollary}
\begin{proof}
Let $\left\{T_{n} \right\}_{n \geq 2}$ be a sequence of randomly growing trees with $f(n) = n + a$. 
M\'ori~\cite{mori05maxdegree} showed that 
$\Delta(T_{n}) / n^{1/(2+a)}$ converges almost surely and the limit is positive almost surely. 
Thus, with high probability, $\Delta(T_{n}) \geq \delta_{n} n^{1/(2+a)}$. 
The result then follows from Lemma~\ref{lem:max_degree_lb}. 
\end{proof}

We note that the lower bounds of Theorem~\ref{thm:general_lb} and Lemma~\ref{lem:max_degree_lb} are complementary: each can be better in certain cases. 
For instance, consider affine preferential attachment, that is, $f(n) = n + a$, where $a > -1$. 
Corollary~\ref{cor:affine} shows that $X_n \geq \delta_{n} n^{1/(2+a)}$ with high probability, for any vanishing sequence $\delta_{n}$, 
whereas (an inspection of the proof of) Theorem~\ref{thm:general_lb} shows that 
$\E[X_{n}] \geq n^{\log_{2}(1 + (1 - \exp(- (1+a)/(4+2a)))^{2})}$. 
When $a \to -1$, the power of the lower bound given by Corollary~\ref{cor:affine} converges to $1$, 
while the power of the lower bound given by Theorem~\ref{thm:general_lb} converges to $0$; 
thus for small values of $a$ (and in particular for $a = 0$), the maximum degree lower bound is better. 
On the other hand, when $a \to \infty$, 
the power of the lower bound given by Corollary~\ref{cor:affine} converges to $0$, 
while the power of the lower bound given by Theorem~\ref{thm:general_lb} is bounded away from $0$, converging to $\log_{2}(1+(1-1/\sqrt{e})^{2}) \approx 0.21$; 
thus the bound of Theorem~\ref{thm:general_lb} is better for large values of $a$.

\subsection{Superlinear preferential attachment}\label{sec:superlinear}

When the attachment function $f$ is (polynomially) \emph{superlinear}, the resulting trees are quite extreme, as described in detail by Oliveira and Spencer~\cite{oliveira2005connectivity} and more generally by Iyer and Lodewijks~\cite{LodIyer23}. In particular, in such trees a dominant ``supervertex'' emerges, which attracts a positive fraction of all future edges, and thus has degree of order~$n$. 
Such a supervertex directly implies that 
the largest common subtree of two independent superlinear preferential attachment trees has a size that is linear in $n$, 
since both contain a linear size star (around the supervertex of each tree). 
In fact, even more is true, as the following two results show; in particular, we have that $X_{n} = (1-o(1)) n$ with high probability. This is in stark contrast to the case of uniform attachment, as shown by Theorems \ref{thm:UB_near1} and \ref{thm:UB_near0}.

\begin{theorem}[Maximum degree in superlinear preferential attachment trees]\label{thrm:supermax}
    Fix $\alpha>1$ and assume that $f(n)=\Theta(n^\alpha)$ as $n\to\infty$ and that $f(n)>0$ for all $n\in\N$. Then, the following hold:
    \begin{itemize}
        \item If $\alpha>2$, then the sequence $\{n-\Delta(T_n)\}_{n\geq 1}$ has an almost surely finite limit.
        \item If $\alpha=2$, then the sequence $\{(n-\Delta(T_n))/\log n\}_{n\geq 2}$ is tight.
        \item If $\alpha\in(1,2)$, then the sequence $\{(n-\Delta(T_n))n^{\alpha-2}\}_{n\geq 1}$ is tight.
    \end{itemize}
\end{theorem}

\begin{remark}
    Theorem~\ref{thrm:supermax} improves on the result in~\cite[Theorem $9.1$ $(2)$]{SethVenk16}, where it is shown that $\Delta(T_n)/n$ converges to $1$, almost surely, by showing sublinear fluctuations. 
\end{remark}

The above theorem together with Lemma \ref{lem:max_degree_lb} directly imply the following corollary, which is the `same' result for the largest common subtree of two independently grown superlinear preferential attachment trees.

\begin{corollary}[Largest common subtree in superlinear preferential attachment trees]\label{thrm:suplin}
Fix $\alpha > 1$ and suppose that 
$f(n) = \Theta\left( n^{\alpha} \right)$ as $n \to \infty$ 
and that $f(n) > 0$ for all $n\in\N$. Then, the following hold: 
\begin{itemize} 
\item If $\alpha > 2$, then the sequence 
$\left\{ n - X_{n} \right\}_{n \geq 1}$ has an almost surely finite limit.
\item If $\alpha=2$, then the sequence $\{(n-X_n)/\log n\}_{n\geq 2}$ is tight.
\item If $\alpha\in(1,2)$, then the sequence $\{(n-X_n)n^{\alpha-2}\}_{n\geq 1}$ is tight.
\end{itemize}
In particular, in all of the above cases we have that $X_{n} = (1-o(1))n$ with high probability. 
\end{corollary}

\begin{proof}
    The cases $\alpha=2$ and $\alpha\in(1,2)$ directly follow from the corresponding tightness results in Theorem~\ref{thrm:suplin} combined with Lemma~\ref{lem:max_degree_lb}. For $\alpha>2$, we argue as follows. Since $n-\Delta(T_n)$ converges almost surely by Theorem~\ref{thrm:suplin}, there almost surely exists a random integer $N$ and a random vertex $v^*$ in the tree $T_N$ constructed up to step $N$, such that all vertices added after step $N$ are connected to $v^*$. This is thus also true for two tree processes $\{T^1_n\}_{n\geq 2}$ and $\{T^2_n\}_{n\geq 2}$, where we have integers $N_1$ and $N_2$ and vertices $v^*_1$ and $v^*_2$, respectively. Suppose that $N_2\geq N_1$ without loss of generality. After step $N_2$, in both trees we only attach to a single vertex, so that after some step $N_3\geq N_2$, the largest common subtree grows by one in size at each step. Indeed, if $v^*_1$ and $v^*_2$ are both included in the largest common subtree, which they eventually are due to their large degrees, then in the next step the vertex that is attached to both of them (in their respective trees) is also included in the largest common subtree. 
\end{proof}

The rest of this section is devoted to the proof of Theorem~\ref{thrm:supermax}. 

\paragraph{Embedding in a continuous-time branching process.}
To prove Theorem~\ref{thrm:supermax} we use a continuous-time branching process (CTBP) embedding of the superlinear preferential attachment tree, similar to the embedding discussed in Section~\ref{sec:prelim}. To this end, recall the Ulam--Harris tree~$\cU_\infty$, as defined in Section~\ref{sec:UlamHarris}. We assign to each $v\in \cU_\infty$ a sequence of i.i.d.\ rate $1$ exponential random variables, $\{E^{\scriptsize (v)}_j\}_{j\in\N}$. Define  
\be S^{\scriptsize(\varnothing)}_i:=E^{(\varnothing)}_1+\sum_{j=2}^i \frac{E^{(\varnothing)}_j}{f(j-1)}\qquad \text{and}\qquad  S^{\scriptsize (v)}_i:=\sum_{j=1}^{i}E_j^{(v)}/f(j)\qquad \text{for }v\in \cU_\infty\setminus\{\varnothing\}. 
\ee 

Finally, the birth process associated with $v$ is 
\be 
\xi_v(\cdot):=\sum_{i=1}^\infty \delta_{S^{\scriptsize (v)}_i}(\cdot),
\ee 
where $\delta$ is a Dirac measure. In words, every vertex $v\in \cU_\infty\setminus\{\varnothing\}$ has an independent and identically distributed birth process, where the first child is born after an exponential clock with rate $f(1)$ rings, after which the second child is born when an exponential clock with rate $f(2)$ rings, etc. The root $\varnothing$ has a slightly different birth process, where the first birth occurs after an exponential time with rate $1$, the second child at rate $f(1)$, etc. This accounts for the fact that the degree of a leaf  in the discrete tree $T_n$ is $1$, whereas the degree of the root in $T_1$ is $0$.  We then define the process $\{\text{BP}(t): t\geq 0\}$ as CTBP started at time $t=0$ with one individual $\varnothing$ with birth process~$\xi_\varnothing$. We view $\varnothing$ as the root of the branching process. All vertices born into the process are labeled according to the Ulam--Harris labeling (i.e., according to $\cU_\infty$), and every individual $v\in \cU_\infty$, once born, produces offspring according to $\xi_v$, independently of all other vertices.

We also recall the stopping times $\{\tau_n\}_{n\in\N}$ from~\eqref{eq:stoppingtimes}. We can again use the the Athreya--Karlin embedding, as in Lemma~\ref{thrm:embed}, to conclude that the sequence of superlinear preferential attachment trees $\{T_n\}_{n\in\N}$ can be embedded into the CTBP $\{\bp(t): t\geq 0 \}$, in the sense that
\be \label{eq:embedding}
\{T_n: n\in\N\}\overset d= \{\bp(\tau_n):n\in\N\}.
\ee 

\paragraph{Explosion.} The main crucial difference between the embedding of uniform attachment trees (or, more generally, (sub-)linear preferential attachment trees) and superlinear preferential attachment trees in CTBPs, is that the CTBP used for the latter embedding produces an infinite total progeny in finite time. That is, $\tau_\infty:=\lim_{n\to\infty}\tau_n<\infty$ almost surely. This is readily checked, as the expected amount of time required for $\varnothing$ to produce infinite offspring equals 
\be \label{eq:explcond}
1 +\sum_{j=1}^\infty \frac{1}{f(j)}<\infty, 
\ee 
where the finiteness follows as $f(j)=\Theta(j^\alpha)$ with $\alpha>1$. In particular, this implies that $\tau_\infty<\infty$ almost surely. A CTBP $\{\bp(t): t\geq 0\}$ that produces an infinite total progeny in finite time is called \emph{explosive}, and the event that an infinite progeny is produced in finite time is called \emph{explosion}. 

Moreover, the continuous-time embedding only holds up to time $\tau_\infty$. That is, the limiting tree $T_\infty:=\cup_{n=1}^\infty T_n$ is equal in distribution to $\bp(\tau_\infty)$ by~\eqref{eq:embedding}, as  follows from~\cite[Theorem~$4.3$]{oliveira2005connectivity}, but there is no correspondence between the superlinear preferential attachment tree and $\bp(t)$ for~$t>\tau_\infty$.

 We refer the interested reader to~\cite{oliveira2005connectivity,Athr07} and references therein for more background on explosive superlinear preferential attachment trees and their connection to explosive continuous-time branching processes, as well as to~\cite{Komjthy2016ExplosiveCB} and the references therein for more information on explosive continuous-time branching processes (or the more general Crump--Mode--Jagers branching processes) in general.

\paragraph{Preliminary results.} We now introduce some notation and present several preliminary results that we need for the proof of Theorem~\ref{thrm:supermax}. We define  
$\cB(\varnothing):=0$ and 
\be 
\cB(v):=E^{(\varnothing)}_1+\sum_{j=2}^{v_1}\frac{E^{(\varnothing)}_j}{f(j-1)}+\sum_{i=2}^k \sum_{j=1}^{v_i}\frac{E^{\scriptsize (v_1, \ldots, v_{i-1})}_j}{f(j)} \qquad \text{for }v=(v_1,\ldots, v_k)\in \cU_\infty\setminus\{\varnothing\},
\ee 
and 
\be 
\cP(\varnothing):=E^{(\varnothing)}_1+\sum_{j=2}^\infty \frac{E^{(\varnothing)}_j}{f(j-1)}, \qquad\text{and}\qquad \ \cP(v):=\sum_{j=1}^\infty \frac{E^{\scriptsize (v)}_j}{f(j)}\qquad \text{for }v\in \cU_\infty\setminus\{\varnothing\}. 
\ee 
In words, $\cB(v)$ is the birth time of $v$ and $\cB(v)+\cP(v)$ is the \emph{explosion time} of $v$, that is, the time at which $v$ has produced infinite offspring. As discussed, the explosion time of each vertex is almost surely finite. In particular, it holds that $\tau_\infty\leq \cB(v)+\cP(v)$ for all $v\in \cU_\infty$ almost surely, and we know from~\cite[Theorem $1.2$]{oliveira2005connectivity} that in the limiting infinite tree $T_\infty$ there exists a unique vertex with infinite degree almost surely. That is, almost surely there exists a unique vertex $v^*\in \cU_\infty$ such that 
\begin{align}
    \cB(v^*)+\cP(v^*)=\tau_\infty,\qquad \text{and}\qquad
    \cB(v)+\cP(v)>\tau_\infty, \qquad \text{for all } v\in \cU_\infty\backslash \{v^*\}.
\end{align}
Furthermore, for $k,n\in\N$ and $v \in \cU_\infty$, we define 
\be 
Y_n^k(v):=|\{i\in[n]: \text{The tree rooted at $v\oplus(i)$ is of size at least $k+1$ at time $\cB(v)+\cP(v)$}\}|.
\ee 
In words, $Y_n^k(v)$ counts the number of children of $v$, among its first $n$ children, that have produced a subtree of size at least $k+1$ by the time $v$ has produced infinite offspring. We have the following proposition.

\begin{proposition}[Expected subtree counts]\label{prop:subtree}
    Fix $k\in\N$ and $\alpha>1$. There exists a constant $C>0$ such that  
    \be
    \E[Y^k_n(\varnothing)]\leq \begin{cases} C, &\mbox{when } k(\alpha-1)>1, \\
    C\log n, &\mbox{when }  k(\alpha-1)=1,\\ 
    Cn^{1-k(\alpha-1)}, &\mbox{when }  k(\alpha-1)\in (0,1).
    \end{cases}
    \ee  
\end{proposition}

\begin{remark}\label{rem:offspring}
    Since the offspring distribution of each vertex $v\in \cU_\infty\setminus\{\varnothing\}$   is equivalent to that of $\varnothing$ up to a shift by a rate-one exponential random variable,  Proposition~\ref{prop:subtree} holds when we replace $\varnothing$ by any $v\in\cU_\infty$ and, if necessary, increase the value of $C$ (though a uniform constant that does not depend on the choice of $v\in \cU_\infty$ suffices).
\end{remark}

\begin{remark}
    Though superlinear preferential attachment models with attachment function $f(n)=\Theta(n^\alpha)$ as $n\to \infty$  for some $\alpha>1$ are most studied in the literature, there are other examples  such that~\eqref{eq:explcond} is satisfied, but $f(n)=o(n^\alpha)$ for any $\alpha>1$. For such models, the result in Proposition~\ref{prop:subtree} does not hold, as shown in~\cite[Theorem $3.21$]{LodIyer23}. Indeed, for such models the random variable $Y^k_n(\varnothing)$ diverges as $n$ tends to infinity almost surely, for any $k\in\N$. It would be interesting to understand what kind of scaling would be required, and if the result $\Delta(T_n)=n-o(n)$, as in Theorem~\ref{thrm:supermax}, still~holds.
\end{remark}

\begin{proof}[Proof of Proposition~\ref{prop:subtree}]
Take any rooted tree $S$ of size $k+1$, where $k\in\N$, and a vertex $i\in\cU_\infty\setminus\{\varnothing\}$, and let $T(S)$ denote the time it takes for the tree $S$, rooted at $i$, to be constructed (from the birth of $i$ onward). Let $\underline m_k:=\min_{i=1,\ldots, k}f(i)\wedge 1$ and $\overline m_k:=\max_{i=1,\ldots, k}f(i) \vee 1$. It is then clear that
    \be \label{eq:stochdom}
    \sum_{\ell=1}^k\mathrm{Exp}^{(\ell)}(k\overline m_k)\preceq \min_{\substack{\text{Trees }S'\\ |S'|=k+1}}T(S')\preceq T(S)\preceq \sum_{\ell=1}^k \mathrm{Exp}^{(\ell)}(\underline m_k),
    \ee 
    where $\{\mathrm{Exp}^{(\ell)}(k\overline m_k)\}_{\ell\in[k]}$ and $\{\mathrm{Exp}^{(\ell)}(\underline m_k)\}_{\ell\in[k]}$ denote sequences of i.i.d.\ exponential random variables with rates $k\overline m_k$ and $\underline m_k$, respectively, and where $\preceq$ denotes stochastic domination. 
    The stochastic domination holds since the rates of the exponential times at which the vertices in $S$ are born are at least $\underline m_k$ and at most $\overline m_k$, and the time it takes for each vertex to be born is at least the minimum of $k$ many exponential random variables (each of rate at most $\overline m_k$ since each vertex in $S$ has degree at most~$k$). We also observe that this stochastic domination holds  independently of which vertex $S$ is rooted at (i.e., $S$ could be rooted at any vertex $v\in \cU_\infty$). The first stochastic domination can also be found in~\cite[Lemma $5.6$]{oliveira2005connectivity}.

    For some rooted tree $S$ of size $k+1$, to be rooted at vertex $i$, let $\cE_i(S)$ be the event that the $i^{\text{th}}$ child of the root produces a tree $S$ before time $\cP(\varnothing)$, that is, before the root produces infinite offspring. Let $\{E_\ell'\}_{\ell\in[k]}$ be i.i.d.\ exponential random variables with rate $k\overline m_k$, and let $\{E_j\}_{j\in\N}$ be independent exponential random variables, where $E_1$ has rate $1$ and $E_j$ for $j\geq 2$ has rate $f(j-1)$, which are independent of the $\{E_\ell'\}_{\ell\in[k]}$. We associate the random variables $\{E_\ell'\}_{\ell\in[k]}$ with the formation of the tree $S$ that is rooted at $i$. In particular, their sum forms a (stochastic) lower bound for $T(S)$, as follows from~\eqref{eq:stochdom}. The random variables $\{E_j\}_{j\in\N}$ are associated with the inter-birth times of the root $\varnothing$. Importantly, since the birth process $\xi_\varnothing$ of the root and the birth processes associated to vertex $i$ and its descendants are independent, it follows that the two aforementioned collections of random variables are independent. We can then bound 
    \be 
    \P\bigg(\bigcup_{\substack{\text{Rooted trees $S$}\\|S|=k+1}}\!\!\!\!\!\!\!\!\cE_i(S)\bigg)=\P\bigg(\bigcup_{\substack{\text{Rooted trees $S$}\\|S|=k+1}}\bigg\{T(S)\leq \sum_{j=i+1}^\infty E_j\bigg\}\bigg)\leq \P\bigg(\sum_{\ell=1}^k E_\ell'\leq \sum_{j=i+1}^\infty E_j\bigg),
    \ee 
    where the final step follows by~\eqref{eq:stochdom}. We now introduce
    \be \label{eq:mui}
    \mu_i:=\sum_{j=i}^\infty \frac{1}{f(j)},
    \ee 
    and observe that $\mu_i=\Theta(i^{-(\alpha-1)})$, so that $f(j-1)\mu_i\geq 2$ for all $j\geq i$ and all $i$ sufficiently large. As a result, using a Chernoff bound yields
    \be
     \P\bigg(\sum_{\ell=1}^k E_\ell'\leq \!\!\sum_{j=i+1}^\infty\!\! E_j\bigg)= \P\bigg(\exp\bigg(\frac{1}{\mu_i}\bigg( \sum_{j=i+1}^\infty E_j-\sum_{\ell=1}^k E_\ell'\bigg)\bigg)\geq 1\bigg)\leq \prod_{\ell=1}^k \E\big[\e^{-\mu_i^{-1} E_\ell'}\big]\prod_{j=i+1}^\infty \E\big[\e^{\mu_i^{-1}E_j}\big]. 
    \ee 
    We note that the expected values are well-defined when $i$ is sufficiently large, as $\mu_i^{-1}\leq f(j-1)/2<f(j-1)$ for all $j\geq  i+1$. Since each $E_\ell'$ has a rate $k\overline m_k$,  the right-hand side equals 
    \be \label{eq:mgfexplicit}
    \Big(\frac{k\overline m_k}{k\overline m_k+\mu_i^{-1}}\Big)^k\!\!\! \prod_{j=i+1}^\infty \frac{f(j-1)}{f(j-1)-\mu_i^{-1}}.
    \ee  
    As $f(j-1)\mu_i\geq 2$ implies that $(f(j-1)\mu_i-1)^{-1}\leq 2/(f(j-1)\mu_i)$, and by using the definition of $\mu_i$, as in~\eqref{eq:mui}, we can rewrite the product as
    \be 
    \prod_{j=i+1}^\infty \Big(1+\frac{1}{f(j-1)\mu_i-1}\Big)\leq \exp\bigg(\sum_{j=i+1}^\infty \frac{1}{f(j-1)\mu_i-1}\bigg)\leq \exp\bigg(\frac{2}{\mu_i}\sum_{j=i}^\infty \frac{1}{f(j)}\bigg)=\e^2.
    \ee 
    Then, omitting the term $k\overline m_k$ from the denominator, we bound the first term in~\eqref{eq:mgfexplicit} from above by $(k\overline m_k)^k \mu_i^k$, to arrive at
    \be 
     \P\bigg(\sum_{\ell=1}^k E_\ell'\leq \sum_{j=i+1}^\infty E_j\bigg)\leq \e^2 (k\overline m_k)^k \mu_i^k=\cO(i^{-k(\alpha-1)}).
    \ee 
    Combining all of the above we finally arrive at the following bound for all $i$ sufficiently large:
    \be \ba\label{eq:probSbounds}
    \P(\text{vertex $i$ grows a tree of size at least $k+1$ before time $\cP(\varnothing)$})&=\P\bigg( \bigcup_{\substack{\text{Rooted trees }S\\ |S|=k+1}}\!\!\!\!\!\!\cE_i(S)\bigg)\\
    &=\cO(i^{-k(\alpha-1)}).
    \ea \ee
    Now summing over $i$ from $1$ to $n$ (and using the trivial upper bound of $1$ for all small $i$ for which~\eqref{eq:probSbounds} cannot be used) yields the desired result based on the different values of $k(\alpha-1)$.
\end{proof}

Proposition~\ref{prop:subtree} provides a more precise insight into the number of subtrees that are rooted at the children of the unique infinite-degree vertex $v^*$ in $T_\infty$, and helps us determine the magnitude of the maximum degree in $T_n$ more precisely beyond `$\Delta(T_n)$ is of order $n$'. We are now ready to prove Theorem~\ref{thrm:supermax}.

\begin{proof}[Proof of Theorem~\ref{thrm:supermax}]
    We first observe that in all three cases of the theorem, the sequences of random variables, $\{n-\Delta(T_n)\}_{n\geq 1}, \{(n-\Delta(T_n))/\log n\}_{n\geq 2}$, and $\{(n-\Delta(T_n))n^{\alpha-2}\}_{n\geq 1}$, are bounded from below by zero.
    
    Let us then start with the case $\alpha>2$. Here we can immediately draw upon results from~\cite[Theorem $1.2$]{oliveira2005connectivity}. Namely, in this setting (for $\alpha>2$), it holds that there almost surely exists some $N\in\N$ and some distinguished vertex $v^*$ in $T_N=\bp(\tau_N)$, such that all vertices born after time $\tau_N$ (or after step $N$ in discrete time) are children of $v^*$. In particular, $v^*$ has the largest degree in $T_n=\bp(\tau_n)$ for all $n>2N-1$, almost surely. Hence, the infinite limiting tree $T_\infty$ consists of a finite tree $T=T_N$ that contains a unique distinguished vertex $v^*$, and additionally an infinite number of vertices connected to $v^*$ (these are all children of $v^*$), almost surely. It is possible that $v^*$ has a number of children that are part of $T$ and hence do not contribute to the infinite number of vertices that connect to $v^*$ that are not part of $T$, namely $\deg_T(v^*)$ many. We thus have that $n-\Delta(T_n)\overset{\mathrm{a.s.}}{\longrightarrow} |T|-\deg_T(v^*)$, which concludes the proof of the case $\alpha>2$.

    For the other two cases, it suffices to prove that for any $\eps>0$ fixed there exists a constant $K(\eps)>0$ such that $(n-\Delta(T_n))/\log n$  exceeds $K(\eps)$ with probability at most $\eps$ when $\alpha=2$, and such that $(n-\Delta(T_n))n^{\alpha-2}$ exceeds $K(\eps)$ with probability at most $\eps$ when $\alpha\in(1,2)$. We provide a proof for the former, the proof of the latter claim follows analogous steps where one only needs to substitute $n^{2-\alpha}$ for $\log n$.

    Let us consider the event $\{\Delta(T_n)\leq n-K(\eps)\log n\}$ for some $K(\eps)$ to be specified later. This event holds exactly when the number of vertices in $T_n$ that are not children of the maximum degree vertex, is at least $K(\eps)\log n$. We aim to show that the latter event holds with probability at most $\eps$. Let $k^*:=\inf\{k\in\N: k>1/(\alpha-1)\}$. We recall that, by~\cite[Theorem $1.2$]{oliveira2005connectivity}, the limiting infinite tree $T_\infty$ has a finite height and contains a unique vertex $v^*$ that has an infinite degree, almost surely. A combination of Proposition~\ref{prop:subtree} (Equation~\eqref{eq:probSbounds} in its proof in particular), Remark~\ref{rem:offspring}, and the Borel-Cantelli lemma yields that almost surely there is at most a finite number of children of $v^*$ that produce a subtree of size at least $k^*+1$ by time $\cB(v^*)+\cP(v^*)=\tau_\infty$.  As a result, we observe that each subtree of size at least $k^*+1$ has a finite size. That is, for some random $L\in\N_0$ there are $L$ many subtrees rooted at children of $v^*$ with sizes $S_1,\ldots, S_L$, such that each $S_i$ is at least $k^*+1$ and finite, almost surely. 

    We now write 
    \be 
    \E[Y^k_n(v^*)]=\sum_{v\in \cU_\infty}\E[Y^k_n(v)\mathbbm 1_{\{v^*=v\}}]=\sum_{v\in \cU_\infty}\sum_{i=1}^n \E[\mathbbm 1_{\cE_{i,k}(v)}\mathbbm 1_{\{v^*=v\}}], 
    \ee 
    where $\cE_{i,k}(v)$ denotes the event that the subtree, rooted at $v\oplus (i)$, has size at least $k+1$ by time $\cB(v)+\cP(v)$. We claim that the upper  bound $\E[\mathbbm 1_{\cE_{i,k}(v)}\mathbbm 1_{\{v^*=v\}}]\leq \E[\mathbbm 1_{\cE_{i,k}(v)}]\P(v^*=v)$ holds, so that we obtain 
    \be 
    \E[Y^k_n(v^*)]\leq \sum_{v\in \cU_\infty}\E[Y^k_n(v)]\P(v^*=v). 
    \ee
    Since the constant $C$ in the upper bound for $\E[Y^k_n(v)]$ in Proposition~\ref{prop:subtree} does not depend on the choice of $v$ (see Remark~\ref{rem:offspring}), we arrive at \be\label{eq:expsubtreebound} \sum_{k=1}^{k^*-1}\E[Y_n^k(v^*)]=\begin{cases} \cO(\log n), &\mbox{when }\alpha=2, \\
    \cO(n^{2-\alpha}) &\mbox{when }\alpha\in(1,2).
    \end{cases}
    \ee
    To prove the claimed inequality $\E[\mathbbm 1_{\cE_{i,k}(v)}\mathbbm 1_{\{v^*=v\}}]\leq \E[\mathbbm 1_{\cE_{i,k}(v)}]\P(v^*=v)$, we argue that the events $\cE_{i,k}(v)$ and $\{v^*=v\}$ are negatively correlated. Let $\omega$ be a realization of the exponential random variables $(E_i^{(v)})_{i\in\N,v\in\cU_\infty}$ used in the construction of the branching process $\bp$ and suppose $v=(v_1,\ldots , v_k)$ for some $v_1,\ldots , v_k\in\N$ and $k\in\N$ (the case $v=\varnothing$ follows 
    analogously). { We define 
    \be 
    A_v \coloneqq \left\{ (u,\ell) \in \cU_\infty \times \N : u=(v_1,\ldots,v_j), \ell \leq v_{j+1} , j \in \{0,\ldots,k-1\}\right\} , 
    \ee 
    so that the time at which the vertex $v=(v_1,\ldots,v_k)$ is born is exactly
    \begin{equation*}
\cB(v)=E^{(\varnothing)}_1+\sum_{j=2}^{v_1}\frac{E^{(\varnothing)}_j}{f(j-1)}+\sum_{i=2}^k \sum_{j=1}^{v_i}\frac{E^{\scriptsize (v_1, \ldots, v_{i-1})}_j}{f(j)}
=
\sum_{(u,\ell) \in A_v} \frac{E^{(v)}_\ell}{\Tilde{f}(u,\ell)},
    \end{equation*}
    with $\Tilde{f}: \cU_\infty \times \N \to \R_{>0}$ chosen accordingly.
    Now note that the event $\{v^* = v\}$ is decreasing in the random variables $( E_{\ell}^{(u)} )_{(u,\ell) \in A_v}$ and $( E_{\ell}^{(v)} )_{\ell \in \N}$ and increasing in all other exponential random variables. This holds, since when one decreases one of the random variables in $( E_{\ell}^{(u)})_{(u,\ell)\in A_v}$ or $( E_{\ell}^{(v)})_{\ell \in \N}$, the explosion time of $v$ decreases. (The explosion time of some other vertices will also decrease, but by exactly the same amount.) Furthermore, the event $\{v^*=v\}$ is increasing in all other variables, since the explosion time of $v$ is not affected by increasing one of the exponential random variables not in $( E_{\ell}^{(u)} )_{(u,\ell)\in A_v}$ or $( E_{\ell}^{(v)} )_{\ell \in \N}$, but the explosion time of other vertices can only get larger. 
    Contrary, the event $\mathcal{E}_{i,k}(v)$ is increasing in $(E^{(v)}_{\ell})_{\ell \in \N}$, decreasing in $(E^{(v\oplus i \oplus \sigma)}_{\ell})_{\sigma \in \cU_\infty, \ell \in \N}$, and independent of all other variables. Thus we see that the event $\{v^* = v \}$ is decreasing in a variable $E^{(u)}_\ell$ exactly when the event $\mathcal{E}_{i,k}(v)$ is increasing in this variable, and vice versa. The claimed bound thus follows by the FKG inequality.
    }
    
    Using~\eqref{eq:expsubtreebound}, when $\alpha=2$, it follows that for any $\eps>0$ there exists a fixed finite constant $K=K(\eps)$ such that with probability at least $1-\eps/2$ the following event $\mathcal C_n(v^*)$ holds: At time $\cB(v^*) +\cP(v^*)=\tau_\infty$, when $v^*$ has produced infinite offspring (and all other vertices have produced finitely many children only), the subtrees rooted at the first $n$ children of $v^*$ satisfy: 
    \begin{itemize}
        \item There are at most $\frac{K(\eps)}{2k^*}\log n$ many  subtrees of size at least $2$ and at most $k^*$ (which combined contain at most $\frac{K(\eps)}{2}\log n$ many vertices).
        \item The number of vertices contained in subtrees of size at least $k^*+1$ is at most $K(\eps)$.
        \item The remaining subtrees (at least $n-K(\eps)(\log (n)/(2 k^*)+1)$ many) consist of exactly one vertex.
    \end{itemize}
    Here, the first point follows by combining~\eqref{eq:expsubtreebound} with Markov's inequality and by choosing $K(\eps)$ sufficiently large so that this event holds with probability at least $1-\eps/4$. The second point follows from what we concluded above in~\eqref{eq:expsubtreebound}. Namely, there are almost surely finitely many children of $v^*$, say $L$ many, that have grown a subtree of (finite) size at least $k^*+1$ by time $\tau_\infty$, say their sizes are $S_1,\ldots, S_L$. As $L$ is almost surely finite, and $S_1, \ldots, S_L$ are almost surely finite, it follows that 
    \be 
    \lim_{k\to\infty}\P(S_1+\cdots +S_L\geq k)=0. 
    \ee 
    As a result, for $K(\eps)$ large, the second point holds with probability at least $1-\eps/4$. Finally, the third point combines the first two with the fact that the remaining children of $v^*$ have not produced any children by time $\tau_\infty$.  
    
    Now that we have an understanding of the sizes of subtrees, rooted at children of $v^*$ at time $\tau_\infty$, we aim to compare this to the times $\tau_n$, when the tree is of size $n$. Again using~\cite[Theorem~1.2]{oliveira2005connectivity}, for $\alpha\in(1, 2]$  there almost surely exists some (random) $N\in \N$ and a distinguished vertex $v^*$ in $\bp(\tau_N)$ (or $T_N$ in discrete time), such that all vertices born after time $\tau_N$ are either children of $v^*$ or descendants of $v^*$ (i.e.\ grandchildren, great-grandchildren, etc.), and all vertices but $v^*$ have finitely many children after time $\tau_N$. This vertex $v^*$ is the same vertex as above, the unique vertex that has an infinite degree in the limiting tree. Hence, the infinite limiting tree $T_\infty$ consists of a finite tree $T=T_N=\bp(\tau_N)$ that contains a unique distinguished vertex $v^*$, and additionally an infinite tree $S_{\text{inf}}$, rooted at $v^*$. The tree $S_{\inf}$ satisfies that the root has an infinite degree, and that the children of the root satisfy the event $\cC_n(v^*)$ with probability at least $1-\eps/2$ (when thinking of $v^*$ as the root of $S_{\text{inf}}$). Further, as $|T|=|T_N|\leq N$, and $N$ is almost surely finite, the probability that the tree $T$ contains at most  $K(\eps)$ vertices is at least $1-\eps/2$ for $K(\eps)$ sufficiently large.
    
    Finally, there exists some $N'\geq N$ such that $v^*$ obtains the maximum degree after time $\tau_{N'}$ almost surely (or in the tree $T_n$ for all $n\geq N'$, in discrete time). Hence, for all $n\geq N'$, we observe that the number of vertices in $T_n$ that are \emph{not} children of the maximum degree vertex $v^*$ is bounded from above by the sizes of all subtrees rooted at the first $n$ children of $v^*$ in $S_{\text{inf}}$, plus the size of $T$. We can thus conclude, for $\alpha=2$, that the total number of vertices in $T_n$ that are \emph{not} children of $v^*$ is   at least $K(\eps)\log n$ with probability at most $\eps$, for $K(\eps)$ large enough. 
\end{proof}


\section*{Acknowledgements}

The catalyst of this paper was the RandNET Workshop on Random Graphs which took place in Eindhoven, the Netherlands during August 22--30, 2022. We thank the organizers, Serte Donderwinkel, Christina Goldschmidt, Remco van der Hofstad, and Joost Jorritsma, for putting together a fantastic event, which brought our group together. We also thank Eurandom for the hospitality, as well as the participants of the workshop, particularly Tom Hutchcroft, for helpful discussions. We are grateful to Nathan Ross for insightful discussions at the early stages of this project. In addition, we thank Luc Devroye, David Gamarnik, and G\'abor Lugosi for valuable discussions. 

Bas Lodewijks has received funding from the European Union’s Horizon 2022 research and innovation programme under the Marie Sk\l{}odowska-Curie grant agreement no.~$101108569$, ``DynaNet", and has been supported by the grant GrHyDy ANR-20-CE40-0002. 
C\'eline Kerriou was supported by the DFG project~$444092244$ ``Condensation in random geometric graphs'' within the priority programme SPP~$2265$.

This material is based upon work supported in part by the National Science Foundation (NSF) under Grant No.~DMS-1928930, while C\'eline Kerriou\ and Mikl\'os Z.\ R\'acz\ were in residence at the Simons Laufer Mathematical Sciences Institute (SLMath/MSRI) in Berkeley, California, participating in the Probability and Statistics of Discrete Structures program during Spring~2025. 

\section*{AI Statement}

All ideas, proofs, and results in this paper are due to the human authors, and all the writing was done by the human authors. 
AI assistance was used solely for spellchecking, editing Figure~\ref{fig:subtreeweight}, and checking the correctness of mathematical calculations.


\bibliographystyle{plain}
\bibliography{bib}




\end{document}